\documentclass[12pt,reqno]{amsart}

\title[2-complex with matrix spaces]
{A 2-complex containing  Sobolev spaces of matrix fields}

\author{Jay Gopalakrishnan}
\address{Portland State University, PO Box 751, Portland OR 97201, USA }
\email{gjay@pdx.edu}

\author{Kaibo Hu}
\address{School of Mathematics, the University of Edinburgh, James Clerk Maxwell Building, Peter Guthrie Tait Rd, Edinburgh EH9 3FD, UK}
\email{kaibo.hu@ed.ac.uk}

\author{Joachim Sch\"oberl}
\address{Institute of Analysis and Scientific Computing, TU Wien, Wiedner Hauptstr. 8-10, 1040 Wien, Austria}
\email{joachim.schoeberl@tuwien.ac.at}

\usepackage{amssymb}	
\usepackage{tikz}
\usepackage{tikz-cd}
\usepackage{amsmath}
\usepackage{mathrsfs}
\usepackage{amsthm}
\usepackage{verbatim}
\usepackage{subeqnarray}
\usepackage{graphicx}
\usepackage{cancel}
\usepackage{color}
\usepackage{float}
\usepackage{bm}
\usepackage{hyperref}
\usepackage{amsfonts}
\usepackage{amscd}
\usepackage{cases}
\usepackage{xcolor}
\usepackage{mathtools}
\usepackage[margin=2.5cm]{geometry}
\usepackage{ytableau}
\usepackage[bbgreekl]{mathbbol}

\newcommand{\bb}[1]{\mathbb{{#1}}}
\newcommand{\cl}[1]{\mathcal{{#1}}}
\newcommand{\Ho}{{\mathring{H}}}
\DeclareMathOperator{\T}{{\scriptstyle{\top}}}

\newcommand{\clos}[2]{\overline{#1}^{#2}}
\newcommand{\Hcc}{H_{\mathrm{cc}}}
\newcommand{\Hcd}{H_{\mathrm{cd}}}
\newcommand{\HcdT}{H_{\mathrm{cd}\scalebox{0.6}{\ensuremath{\top}}}}
\newcommand{\Hdd}{H_{\mathrm{dd}}}

\newcommand{\Hocc}{\Hot_{\mathrm{cc}}}
\newcommand{\Hocd}{\Hot_{\mathrm{cd}}}
\newcommand{\HocdT}{\Hot_{\mathrm{cd}\scalebox{0.6}{\ensuremath{\top}}}}
\newcommand{\Hodd}{\Hot_{\mathrm{dd}}}

\newcommand{\toHcc}{\,\ot{\!\mathcal{H}}_{\mathrm{cc}}}
\newcommand{\toHcd}{\,\ot{\!\mathcal{H}}_{\mathrm{cd}}}
\newcommand{\toHdd}{\,\ot{\!\mathcal{H}}_{\mathrm{dd}}}

\newcommand{\tHcc}{\mathcal{H}_{\mathrm{cc}}}
\newcommand{\tHcd}{\mathcal{H}_{\mathrm{cd}}}
\newcommand{\tHdd}{\mathcal{H}_{\mathrm{dd}}}

\newcommand{\hHcc}{\hat{H}_{\mathrm{cc}}}
\newcommand{\hHcd}{\hat{H}_{\mathrm{cd}}}
\newcommand{\hHdd}{\hat{H}_{\mathrm{dd}}}

\newcommand{\RT}{\mathcal{R}\mathcal{T}}
\newcommand{\ND}{\mathcal{N}\mathcal{D}}
\newcommand{\pol}{\mathcal{P}}
\newcommand{\LR}{L_{2, \mathbb{R}}}

\newcommand{\LL}{\mathcal{L}}

\newcommand{\Tg}{{T}_{\mathrm{g}}}
\newcommand{\Tc}{{T}_{\mathrm{c}}}
\newcommand{\Td}{{T}_{\mathrm{d}}}
\newcommand{\Rg}{{R}_{\mathrm{g}}}
\newcommand{\Rc}{{R}_{\mathrm{c}}}
\newcommand{\Rd}{{R}_{\mathrm{d}}}

\newcommand{\Socc}[1]{\mathop{\mathring{S}^{({#1})}_{\mathrm{cc}}}}
\newcommand{\Socct}[1]{\mathop{\mathring{\mathcal{S}}^{({#1})}_{\mathrm{cc}}}}
\newcommand{\Sodd}[1]{\mathop{\mathring{S}^{({#1})}_{\mathrm{dd}}}}
\newcommand{\Soddt}[1]{\mathop{\mathring{\mathcal{S}}^{({#1})}_{\mathrm{dd}}}}
\newcommand{\Socd}[1]{\mathop{\mathring{S}^{({#1})}_{\mathrm{cd}}}}
\newcommand{\Socdt}[1]{\mathop{\mathring{\mathcal{S}}^{({#1})}_{\mathrm{cd}}}}

\newcommand{\Sc}[1]{\mathop{{S}^{({#1})}_{\mathrm{c}}}}
\newcommand{\Soc}[1]{\mathop{\mathring{S}^{({#1})}_{\mathrm{c}}}}

\newcommand{\Scc}[1]{\mathop{S^{({#1})}_{\mathrm{cc}}}}
\newcommand{\Scct}[1]{\mathop{\mathcal{S}^{({#1})}_{\mathrm{cc}}}}
\newcommand{\Sdd}[1]{\mathop{S^{({#1})}_{\mathrm{dd}}}}
\newcommand{\Sddt}[1]{\mathop{\mathcal{S}^{({#1})}_{\mathrm{dd}}}}
\newcommand{\Scd}[1]{\mathop{S^{({#1})}_{\mathrm{cd}}}}
\newcommand{\Scdt}[1]{\mathop{\mathcal{S}^{({#1})}_{\mathrm{cd}}}}

\newcommand{\Ddd}{\mathop{D_{\mathrm{dd}}}}

\newcommand{\Rgg}{\mathop{R_{\mathrm{gg}}}}
\newcommand{\Rggt}{\mathop{{\tilde{R}_{\mathrm{gg}}}}}
\newcommand{\Rcc}{\mathop{R_{\mathrm{cc}}}}
\newcommand{\Rcct}{\mathop{{\tilde{R}_{\mathrm{\mathrm{cc}}}}}}
\newcommand{\Dgg}{\mathop{D_{\mathrm{gg}}}}
\newcommand{\Dgc}{\mathop{D_{\mathrm{\mathrm{gc}}}}}
\newcommand{\Dgct}{\mathop{\tilde{D}_{\mathrm{gc}}}}
\newcommand{\Dgch}{\mathop{\hat{D}_{\mathrm{gc}}}}
\newcommand{\Dcc}{\mathop{D_{\mathrm{cc}}}}
\newcommand{\Rgd}{\mathop{R_{\mathrm{gd}}}}
\newcommand{\Rgc}{\mathop{R_{\mathrm{gc}}}}
\newcommand{\RgcT}{\mathop{R_{\mathrm{gc}\scalebox{0.6}{\ensuremath{\top}}}}}
\newcommand{\Rgct}{\mathop{\tilde{R}_{\mathrm{gc}}}}
\newcommand{\Rgch}{\mathop{\hat{R}_{\mathrm{gc}}}}
\newcommand{\Ugch}{\mathop{\hat{U}_{\mathrm{gc}}}}
\newcommand{\Rcd}{\mathop{R_{\mathrm{cd}}}}

\newcommand{\Dcd}{\mathop{D_{\mathrm{cd}}}}

\newcommand{\Dgd}{\mathop{D_{\mathrm{gd}}}}

\definecolor{red}{rgb}{0.8,0,0}
\definecolor{darkorange}{rgb}{1,0.4,0}
\definecolor{lightorange}{rgb}{1,0.6, 0}
\definecolor{yellow}{rgb}{1,0.8, 0}

\newtheorem{theorem}{Theorem}[section]
\newtheorem{lemma}[theorem]{Lemma}

\newtheorem{corollary}[theorem]{Corollary}
\theoremstyle{remark}
\newtheorem{definition}[theorem]{Definition}
\newtheorem{remark}[theorem]{Remark}
\newtheorem{example}{Example}

\tikzcdset{row sep/hypercomplexsize/.initial=1.8cm}
\tikzcdset{row sep/hyper2/.initial=2cm}

\newcommand{\om}{\varOmega}
\makeatletter
\newcommand*\ot[1]{\mathpalette\othelper{#1}}
\newcommand*\othelper[2]{%
        \hbox{\dimen@\accentfontxheight#1%
                \accentfontxheight#11.9\dimen@
                $\m@th#1\widetilde{\mathring{#2}}$%
                \accentfontxheight#1\dimen@
        }%
}
\newcommand*\accentfontxheight[1]{
        \fontdimen5\ifx#1\displaystyle
                \textfont
        \else\ifx#1\textstyle
                \textfont
        \else\ifx#1\scriptstyle
                \scriptfont
        \else
                \scriptscriptfont
        \fi\fi\fi3
}
\makeatother
\newcommand{\Hot}{\ot{H}}
\newcommand{\Ht}{{\Hot^{-1}}}
\newcommand{\HtRT}{{\Hot^{-1}_{\RT}}}
\newcommand{\HtND}{{\Hot^{-1}_{\ND}}}
\newcommand{\HtPl}{{\Hot^{-1}_{\pol_1}}}
\newcommand{\Hts}[1]{\Hot^{{#1}}}
\newcommand{\Hs}[1]{\Hts{{#1}}}
\newcommand{\Hm}{{{H}}^{-1}}
\newcommand{\veps}{\varepsilon}
\newcommand\tr{\operatorname{tr}}
\def\d{\partial}
\newcommand\inc{\operatorname{inc}}
\newcommand\skw{\operatorname{skw}}
\newcommand\vskw{\operatorname{vskw}}
\newcommand\mskw{\operatorname{mskw}}

\newcommand\sym{\operatorname{sym}}
\newcommand\dev{\operatorname{dev}}
\newcommand\grad{\operatorname{grad}}
\newcommand\dfo{\operatorname{def}}
\newcommand\deff{\operatorname{def}}
\newcommand\dfm{\operatorname{def}}
\renewcommand\div{\operatorname{div}}

\newcommand\curl{\operatorname{curl}}

\newcommand\hess{\operatorname{hess}}

\newcommand\ran{\operatorname{range}}

\newcommand\supp{\operatorname{supp}}

\newcommand\R{\mathbb{R}}

\newcommand\V{{\mathbb{V}}}

\newcommand{\vphi}{\varphi}

\newcommand{\D}{\mathcal{D}}
\usepackage{stmaryrd}

\DeclareSymbolFontAlphabet{\mathbbm}{bbold}
\newcommand\id{{\mathbbm{i}}}

\newsavebox\youngAB
\newsavebox\youngAC
\newsavebox\youngAD
\newsavebox\youngBB
\newsavebox\youngBC
\newsavebox\youngBD
\newsavebox\youngCB
\newsavebox\youngCC
\newsavebox\youngCD
\newsavebox\youngDB
\newsavebox\youngDC
\newsavebox\youngDD

\newcommand{\kh}[1]{{[\color{blue}KH:~#1}]}

\usepackage{cite}

\begin{document}

\maketitle

\begin{abstract}  
  Using a generalization of complexes, called 2-complexes, this paper
  defines and analyzes new Sobolev spaces of matrix fields and their
  interrelationships within a commuting diagram. These spaces have
  very weak second-order derivatives. An example is the space of
  matrix fields of square-integrable components whose row-wise
  divergence followed by yet another divergence operation yield a
  function in a standard negative-order Sobolev space. Similar spaces
  where the double divergence is replaced by a curl composed with
  divergence, or a double curl operator (the incompatibility
  operator), are also studied. Stable decompositions of such spaces in
  terms of more regular component functions (which are continuous in
  natural norms) are established.  Appropriately ordering
  such Sobolev spaces with and
  without boundary conditions (in a weak sense), we discover 
  duality relationships between them.
  Motivation to study such Sobolev spaces, from a finite
  element perspective and implications for weak well-posed variational
  formulations are pointed out.
	\\
	\vspace*{0.25cm}
	
\end{abstract}

\section{Introduction}

Substantial improvements in numerical techniques for solving partial
differential equations (PDEs) to address current scientific challenges
have come from connections to and preservation of the differential and
algebraic structures inherent in the PDEs. Ample examples are offered
by the history of finite element techniques.  The earliest finite
elements~\cite{Coura43}, Lagrange elements, consisted of {\em
  scalar-valued} functions. Developments in {\em vector-valued} finite
elements followed, starting with elements~\cite{RaviaThoma77} of
continuous normal ($n$) components. These ``$n$-continuous'' elements
were supplemented with ``$t$-continuous'' {vector-valued}
N{\'{e}}d{\'{e}}lec elements with continuous tangential ($t$)
components~\cite{Nedel80}. Further families of vector-valued elements
were unearthed continuing this line of work. Although these elements
were developed separately, today we understand them together as
fitting into a cochain subcomplex of a de Rham complex of Sobolev
spaces, thanks to intensive research into finite element exterior
calculus
(FEEC)~\cite{arnold2018finite,Arnold.D;Falk.R;Winther.R.2006a,Hiptm99}.
It is now clear how to generalize from scalar and vector fields to
tensor fields, as long as the tensors have the algebraic structure of
$k$-forms in the de Rham complex, i.e., higher order {\em alternating
  tensor-valued} finite elements in any dimension naturally fit into
FEEC.

This paper, while building on these developments, is motivated by
other types of tensors. Problems in continuum mechanics, differential
geometry and general relativity call for a study of tensors with other
types of symmetries. Indeed, even restricting to second-order tensors,
the need for study is evident from the increasing current interest in
matrix-valued finite element functions.  The earliest of these are the
``$nn$-continuous'' symmetric matrix fields (i.e., symmetric
matrix-valued functions $\sigma$ with continuous $(\sigma n)\cdot n$)
of the HHJ (Hellan-Herrmann-Johnson) elements~\cite{Comod89}, now
enjoying a revival~\cite{ArnolWalke20,
  Sinwel:09,PechsteinSchoeberl:11,PechsSchob18} in the TDNNS method
and elsewhere.  A seemingly disjoint (but potentially connected)
recent development is the ``$nt$-continuous'' trace-free matrix finite
element developed~\cite{GopalLederSchob20, GopalLederSchob20a,
  GopalKogleLeder23} for viscous fluid stresses in the context of the
MCS (Mass-Conserving Stress-yielding) method.  To add to this picture,
Regge
elements~\cite{christiansen2011linearization,li2018regge,GopalNeuntSchob23,gopalakrishnan2023analysis}
with ``$tt$-continuous'' symmetric matrix-valued elements are finding
more and more uses.  How does one connect these disparate developments
of $nn$, $nt$, and $tt$-continuous matrix finite elements?  The prior
synthesis (mentioned in the previous paragraph) involved spaces of the
de~Rham complex, all connected by fundamental first-order differential
operators (grad, curl, and divergence, in three dimensions).  In
contrast, what seems to be natural for the matrix finite elements are
other {\em second-order} differential operators.

The goal of this work is to take a step toward understanding what
{\em Sobolev spaces} and their arrangements might reveal a unified structure
where such second-order differential operators and matrix fields arise
naturally.  Although motivated by finite elements, this work does not
contain finite elements. The scope is limited to a study of
infinite-dimensional Sobolev spaces of matrix fields, their
interrelationships, and connections to standard Sobolev spaces. We
focus on spaces of scalar, vector, and matrix valued functions on
{\em three-dimensional} (3D) domains $\om$. Study of higher
order tensor fields on higher dimensional domains is certainly
interesting, but requires more algebraic machinery (such as group
representations and Young tableaux) to work with tensor symmetries.

In 3D however, the relevant symmetries can be captured by the familiar
symmetrization and deviatoric operations,
\begin{equation}
  \label{eq:sym-dev}
  \sym\tau = \frac 1 2 (\tau + \tau^{\T}), \qquad
  \dev\tau = \tau - \frac 1 3 \tr(\tau) \id, \qquad \tau \in \bb M,   
\end{equation}
where $\bb M = \bb R^{3 \times 3}$ denotes the vector space of
$3 \times 3$ real ($\bb R$) matrices, $\tr(\tau)$ denotes the trace of a matrix
$\tau \in \bb M$, and $\id$ denotes the $3\times 3$ identity matrix.
Here and throughout, $\tau^{\T}$, also written as $\T \tau$, denotes
the (pointwise) transpose of a matrix field~$\tau$.  The operations
in~\eqref{eq:sym-dev} generate subspaces of symmetric matrices and
trace-free matrices which we denote by
\[
  \bb S = \sym \bb M, \qquad \bb T = \dev \bb M. 
\]
Let $\bb V = \bb R^3$. We are interested in structures
connecting Sobolev spaces of functions with values in $\bb R$,
$\bb V$, $\bb S$ and $\bb T$ of the following form:
\begin{equation}
\label{eq:RVST}
  \begin{tikzcd}
    \R \arrow{r}     \arrow{d}
    &
    \V \arrow{r}     \arrow{d}
    &
    \V \arrow{r}     \arrow{d}
    &
    \R               \arrow{d}
    \\
    \V \arrow{r}     \arrow{d}
    &                         
    \bb S \arrow{r}  \arrow{d}
    &                         
    \bb T \arrow{r}  \arrow{d}
    &                         
    \V               \arrow{d}
    \\ 
    \V \arrow{r}     \arrow{d}
    &                         
    \bb T \arrow{r}  \arrow{d}
    &                         
    \bb S \arrow{r}  \arrow{d}
    &                         
    \V              \arrow{d}
    \\
    \R \arrow{r}    
    &
    \V \arrow{r}    
    &
    \V \arrow{r}    
    &
    \R              
  \end{tikzcd}
\end{equation}
Such diagrams where $\bb R$, $\bb V$, $\bb S$ and $\bb T$ are replaced
by appropriate Sobolev spaces of functions taking values in them, are
studied here. The first such diagram is introduced below
in~\eqref{eq:1}, which contain first-order derivative operators as
well as key algebraic operations $\T$, $\sym,$ and $\dev$. Certain
combinations of these operations result in basic second-order
derivative operators marked in diagram~\eqref{eq:1-with-2nd-order}.

The tensors along the four edges of~\eqref{eq:RVST} follow the
$\R$-$\V$-$\V$-$\R$ pattern of the well-known 3D de Rham complex
\begin{equation}
  \label{eq:deRham-smooth}
  \begin{tikzcd}[ampersand replacement=\&, column sep=1.2cm]
    C^\infty \ar[r, "{\grad}"]
    \&
    C^\infty \otimes \bb V  \ar[r, "{\curl}"]
    \&
    C^\infty \otimes \bb V  \ar[r, "{\div}"]
    \&
    C^\infty 
  \end{tikzcd}
\end{equation}
of infinitely smooth $(C^\infty)$ scalar and vector fields on $\om$.
Recall that a ``complex'' is a sequence of linear spaces $X_i$ and
linear maps $A_i:X_i \to X_{i+1}$, traditionally expressed by 
\begin{equation}
  \label{eq:seq}  
  \begin{tikzcd}[ampersand replacement=\&]
    \cdots\quad
    X_{k-2}
    \ar[r, "{A_{k-2}}"]
    \&
    X_{k-1} \ar[r, "A_{k-1}"]
    \&
    X_k \ar[r, "A_k"]
    \&
    X_{k+1} \ar[r, "{A_{k+1}}"]
    \&
    X_{k+2}
    \quad\cdots, 
  \end{tikzcd}
\end{equation}
satisfying $A_{i+1} \circ A_i = 0$ for all $i$.
In~\cite{olver1982differential}, ``$\ell$-complexes'' arose, which are
sequences~\eqref{eq:seq} with the property
$A_{i+\ell}\circ \cdots \circ A_{i+1}\circ A_{i} =0$ for all $i$ and
some fixed integer $\ell$ (so, e.g., a 1-complex is a complex in the
usual sense). As we shall see, diagrams of the form~\eqref{eq:RVST}
that we study here have a 2-complex structure (which explains the
title of this paper).  Definition~\ref{def:2complex} below formalizes
the 2-complex notion in the context of such diagrams.

Other examples of complexes, beyond the de Rham
complex~\eqref{eq:deRham-smooth}, include the well-known elasticity
complex~\cite{ArnolHu21,arnold2002mixed,Arnold2006a,PaulyZuleh23,Krone60},
also named after
Calabi or Kr\"oner,
\begin{equation}\label{elasticity}
\begin{tikzcd}[column sep=large]
  C^{\infty}\otimes \bb V
  \arrow{r}{\sym\grad}
  & C^{\infty}\otimes \bb S \arrow{r}{\curl \T \curl}
  & C^{\infty}\otimes \bb S \arrow{r}{\div}
  & C^{\infty} \otimes  \bb V,
\end{tikzcd}
\end{equation}
the hessian complex~\cite{ArnolHu21,hu2025distributional} 
\begin{equation}\label{hess-complex}
  \begin{tikzcd}[column sep=large]
    C^{\infty}  \arrow{r}{\grad\grad}
    & C^{\infty}\otimes \bb S \arrow{r}{\curl}
    & C^{\infty}\otimes \bb T     \arrow{r}{\div}
    & C^{\infty}\otimes \bb V,
  \end{tikzcd}
\end{equation}
and the $\div\div$ complex~\cite{ArnolHu21,PaulyZuleh20}
\begin{equation}\label{divdiv-complex}
  \begin{tikzcd}[column sep=large]
    C^{\infty} \otimes \bb V  \arrow{r}{\dev\grad}
    & C^{\infty}\otimes  \bb T    \arrow{r}{\sym\curl}
    & C^{\infty} \otimes \bb S     \arrow{r}{\div\div}
    & C^{\infty}.
  \end{tikzcd}
\end{equation}
These complexes can be systematically derived from the de Rham
complex~\eqref{eq:deRham-smooth} using the Bernstein-Gelfand-Gelfand
(BGG) construction, originally developed in algebraic and geometric
contexts~\cite{vcap2001bernstein,bernstein1975differential} and more
recently adapted to certain Sobolev spaces~\cite{ArnolHu21,vcap2023bgg}.
We shall see that analogous complexes,
with other ``$H^{-1}$~based'' Sobolev spaces,
defined shortly in
\eqref{eq:Hcc-cd-dd-norms}--\eqref{eq:H-cc-dd-cd-defn}, also arise
naturally from the 2-complexes and the diagrams of the
type~\eqref{eq:RVST} that we study here.

On the theme of {\em weakly regular $H^{-1}$~based Sobolev spaces,}
which is pervasive in this paper,
some motivating examples shed more light.  Let 
\begin{equation}
  \label{eq:L2-inner}
  (u, v) := \int_\om uv\; dx
\end{equation}
for scalar fields $u, v$, and in addition, for vector or matrix fields
$u$ and $v$, we continue to use the same notation $(u, v)$ to denote
the Lebesgue integral (when it exists) over $\om$ of the dot product
$u \cdot v$, or the Frobenius product $u: v$, respectively, of $u$ and
$v$.  All function spaces on $\om$ are defined precisely in
Subsection~\ref{ssec:preliminaries-spaces}, but for expediency, we use
the standard space $L_2$ and the space $H^{-1}$ with weaker topology
in the quick discussion of two examples below, both showing the role
of weak regularity, and each illuminating the role of one of two
algebraic operations $\sym$ and $\dev$.

\begin{example}
  The stress tensor $\sigma$ in linear elasticity is a matrix field
  which must satisfy $\sigma = \sym(\sigma)$ due to conservation of
  angular momentum.  Well-posed formulations for the
  Hellinger-Reissner principle in linear elasticity seek a symmetric
  matrix field (the stress tensor) $\sigma : \om \to \bb S$ in some
  Sobolev space $\Sigma$ and a vector field $u: \om \to \bb V$
  (displacement) in some Sobolev space $V$ satisfying
  \begin{equation}
    \label{eq:10}
    \begin{aligned}
      (A \sigma,  \tau)  +  ( u,  \div \tau)
      & = 0 &&\text{ for all } \tau \in \Sigma,
      \\
      (\div \sigma, v)
      & = (f,   v) &&\text{ for all } v \in V,
    \end{aligned}
  \end{equation}
  where $A$, $f$, and $\div$ denotes the compliance tensor, the load
  vector field, and row-wise divergence of a matrix field,
  respectively.  A ``regular choice'' is $V = L_2 \otimes \bb V$ and
  \begin{equation}
    \label{eq:8}
    \Sigma = \{ \tau \in L_2 \otimes \bb S: \div \tau \in L_2\otimes \bb V\}.
  \end{equation}
  Construction of finite elements for this $\Sigma$ is difficult and
  had remained an open problem for decades, as noted
  in~\cite{Arnol02}.  (If the symmetry condition on $\sigma$ were
  absent, then three copies of the $n$-continuous finite
  elements would have been sufficient.)  
  An alternative choice of ``weak regularity'' is
  \begin{equation}
    \label{eq:23}
    \Sigma = \{ \tau \in L_2 \otimes \bb S: \div \tau \in V^*\}
  \end{equation}
  where $V^*$ is weaker than $L_2$ integrals of
  the form $(\div \tau, v)$ in~\eqref{eq:10} are relaxed to a duality pairing
  $(\div\tau)(v)$ in $V$. The TDNNS
  method~\cite{PechsteinSchoeberl:11,PechsSchob18} with $nn$-continuous stresses
  can be seen as a
  discretization of such a formulation with $V = \Ho(\curl)$, a space
  defined shortly in \eqref{eq:22}. Theorem~\ref{thm:duality} shows
  that the condition $\div \tau \in V^*$ in~\eqref{eq:23} is equivalent to
  $\div \div \tau \in H^{-1}$. This motivates us to study
  Sobolev spaces with this weak regularity condition, namely the
  spaces~$\Hdd$ and $\tHdd$ defined in \eqref{eq:Hdd-defn} and
  \eqref{eq:Hdd-smoother-defn}, respectively.
\end{example}

\begin{example}
  Viscous stresses in Stokes flow with fluid velocity
  $u: \om \to \bb V$ can be extracted from the symmetric part of
  $\sigma = 2 \nu\grad u$, where $\nu$ is the kinematic viscosity. The
  incompressibility constraint $\div u=0$, a well known source of
  challenges in numerical simulation~\cite{JohnLinkeMerdo17}, now
  emerges as an algebraic constraint: $\sigma = \dev \sigma$. The
  definition of $\sigma$ and flow equations suggest that we should
  find $\sigma$ in a space $\Sigma$ of trace-free matrix fields, $u$
  in some space $V$ of vector fields,
  and the pressure $p$ in some space $Q$ of scalar
  fields such that
  \begin{align*}
    (\nu^{-1}\sigma,  \tau)
      + (u, \div \tau)
    & = 0 && \text{for all } \tau \in \Sigma,
    \\
    (\div\sigma, v)  +  (\div v, p )
    & = - (f, v ) && \text{for all } v \in V,
    \\
     (\div u, q )  & = 0 && \text{for all } p \in Q.
  \end{align*}
  for some given source field~$f$.  The MCS
  method~\cite{GopalLederSchob20,
    GopalLederSchob20a,GopalKogleLeder23} sets $V = \Ho(\div)$ (a
  space defined shortly in \eqref{eq:22}) and $Q = \div V$. Then,
  instead of a ``regular choice''
  $\Sigma = \{  \tau \in L_2 \otimes \bb T: \div \sigma \in L_2
  \otimes \bb V\}$ that would make the integrals like
  $(\div\sigma, v)$ well defined, the MCS formulation proposes a
  choice of ``weak regularity,'' namely
  \begin{equation}
    \label{eq:24}
    \Sigma = \{ \tau \in L_2 \otimes \bb T:  \div \sigma \in V^*\}
  \end{equation}
  for which simple $nt$-continuous finite elements work, after relaxing
  $(\div\sigma, v)$ to a duality pairing $(\div\sigma)(v)$ in~$V$.
  Theorem~\ref{thm:duality} shows
  that the condition $\div \tau \in V^*$ in~\eqref{eq:24} is equivalent to
  $\curl \div \tau \in H^{-1}$. This motivates us to study
  Sobolev spaces with this weak regularity condition, namely the
  spaces~$\Hcd$ and $\tHcd$ defined in \eqref{eq:Hcd-defn} and
  \eqref{eq:Hcd-smoother-defn}, respectively.
\end{example}



\subsection{Preliminaries and spaces}  \label{ssec:preliminaries-spaces}

Let $\om$ be a bounded open connected subset of the Euclidean space $\bb R^3$
with Lipschitz boundary. Let $L_2$ denote the space of
square-integrable $\bb R$-valued functions on~$\om$, or equivalently
the space of square-integrable $\bb R$-valued functions on~$\bb R^3$
supported on $\bar\om$.
Let $\cl D(\om)$ denote the
Schwartz space of smooth test functions on $\om$ that are compactly
supported in $\om$.  The dual of any topological space $X$ is denoted
by $X^*$.  The space of distributions on $\om$ is denoted by
$\cl D (\om)^*$.  The space of vector fields on $\om$ with square-integrable
components is denoted by $L_2 \otimes \bb V$ and the notation is
similarly extended to $\bb T$ and $\bb S$-valued fields on~$\om$ as
well as to other spaces, e.g., $\cl D(\om)^* \otimes \bb S$ denotes the space of 
symmetric matrix-valued fields whose components are distributions on~$\om$.

We use the standard Sobolev spaces $H^s(\bb R^3)$
and $H^s(\om)$ for any $s \in \bb R$. (see e.g.,~\cite{AdamsFourn03,McLea00}). We omit the domain $\om$ from
the notation when no confusion can arise and simply write $H^s$ for
$H^s(\om)$.  For scalar functions $u: \om \to \bb R$ and vector
functions $v, q : \om \to \bb V$, let
\begin{equation}
\label{eq:std-norms}
\begin{gathered}
  \| u \|_{H^1}^2
  =  \|u \|_{L_2}^2 + \| \grad u \|_{L_2}^2, \\
  \| v \|_{H(\curl)}^2
  =  \| v \|_{L_2}^2 + \| \curl v \|_{L_2}^2, \qquad   
  \|  q \|_{H(\div)}^2
  =  \| q \|_{L_2}^2 + \| \div v \|_{L_2}^2.
\end{gathered}
\end{equation}
These are norms of well-known Hilbert spaces of functions on $\om$,
namely
\begin{subequations}
  \label{eq:std-Sobolev-spaces}
  \begin{align}
    H(\grad) & \equiv H^1   = \{ u \in L_2: \grad u \in L_2 \otimes \bb V\},
    \\
    H(\curl) & = \{ v \in L_2 \otimes \bb V: \curl v \in L_2 \otimes \bb V\},
    \\
    H(\div) & = \{ q \in L_2 \otimes \bb V: \div q \in L_2\}.
  \end{align}
\end{subequations}
Using the standard norms in~\eqref{eq:std-norms}, the closures
\begin{equation}
  \label{eq:22}
  \Ho(\grad) = \clos{\cl D(\om)}{\| \cdot \|_{H^1}},
  \quad
  \Ho(\curl) = \clos{\cl D(\om) \otimes \bb V}{\| \cdot \|_{H(\curl)}},
  \quad
  \Ho(\div) = \clos{\cl D(\om) \otimes \bb V}{\| \cdot \|_{H(\div)}},  
\end{equation}
give well-known zero-trace subspaces of the spaces
in~\eqref{eq:std-Sobolev-spaces}.

Further spaces are defined using similar closures, but using the set of 
$\cl D(\om)$-functions extended by zero to all $\bb R^3$ and
closing  the set  using the $H^s(\bb R^3)$~norm.
Set 
\begin{equation}
  \label{eq:Hminus-closure}
  \Hts{s} = \clos{\cl D(\om)}{\| \cdot \|_{H^s(\bb R^3)}}, \qquad s \in \bb R,
\end{equation}
and $ \| u \|_{\Hts s} := \| u \|_{H^s(\bb R^3)}.$ This space is often
just denoted by $\widetilde{H}^s(\om)$, and in our setting, is also the
same as another often-occurring space in the literature,
$H^s_{\overline{\om}}(\bb R^3) =  \{ u \in H^{s}(\bb R^3): \; \supp u \subset \bar \om\}$ (see e.g.~\cite[Theorem~3.29]{McLea00}), i.e.,
\begin{equation}
  \label{eq:Hts}
  \Hts{s} = \{ u \in H^{s}(\bb R^3): \; \supp u \subset \bar \om\}.
\end{equation}
Since the $i$th
partial derivative $\d_i$ satisfies
$\|\d_i \varphi \|_{\Hs{s}} \le \| \varphi \|_{\Hs{s+1}}$ for all
$\varphi \in \cl D(\om)$, and since $\cl D (\om)$ is dense in $\Hs{s}$
by definition~\eqref{eq:Hminus-closure}, we conclude that for any real $s$, 
\begin{equation}
  \label{eq:di-cty}
  \d_i : \Hs{s+1}\to \Hs{s} \quad \text{is continuous}.
\end{equation}
It is well known~\cite[Theorem~3.30]{McLea00} that $\Hts{s}$ is also 
identifiable with a standard dual space
\begin{equation}
  \label{eq:Hs-dual-char}
  \Hts{s} = (H^{-s})^*
\end{equation}
for any $s \in \bb R$. 
The case $s=-1$ is of particular interest here.  The space $\Ht$, not
to be confused with $H^{-1} = \Ho(\grad)^*$, satisfies,
per~\eqref{eq:Hs-dual-char},
\begin{align}
  \label{eq:7}
  \Ht &  = H(\grad)^*,
\end{align}
and furthermore, even if $\Ht$ is not embedded in a space of distributions on
$\om$, it can be characterized using tempered distributions on
$\bb R^3$ supported on the closure of $\om$, due to~\eqref{eq:Hts}.
Therefore the norm of any $u $ in $\Ht$ can be computed either using
the $H^{-1}(\bb R^3)$-norm of the extension of $u$ by zero to all
$\bb R^3$, or by duality using~\eqref{eq:7}.  Finally, we note that it is also well-known \cite[Theorem~3.33]{McLea00} when $s>0$, $\Hts{s}$ is contained in
\[
  \Ho^s:= \clos{\cl D(\om) }{\| \cdot \|_{H^s(\om)}}
\]
and, moreover,  $\Hts{s}$ and $\Ho^s$ are equal if $s>0$ and
$s-\frac 1 2 $ is not an integer, so e.g.,  $\Ho^1  = \Hts{1}$.

For a general  $s \in \bb R$, we define the norms 
\begin{align*}
  \| v \|_{\Hs s (\curl)}^2
  & = \| v\|_{\Hs s}^2 + \| \curl v \|_{\Hs s}^2,
  &
    \| q \|_{\Hs s (\div)}^2
  & = \| q \|_{\Hs s}^2 + \| \div q \|_{\Hs s}^2.
\end{align*}
and set
\begin{align}
  \label{eq:13}
  \Hs s (\curl) = \clos{ \cl D(\om) \otimes \bb V}{ \| \cdot \|_{\Hs s(\curl)}},
  \qquad
  \Hs s (\div) = \clos{ \cl D(\om) \otimes \bb V}{ \| \cdot \|_{\Hs s (\div)}}.
\end{align}
These spaces with $s=-1$ feature in a central diagram introduced shortly.

Next we introduce  key spaces of matrix-valued
fields, which are also needed for the diagram.
Note that when the standard differential operators $\div$ and
$\curl$ are applied to matrix-valued fields, we do so row-wise.  The
next definitions involve second-order differential operators on
matrix-valued functions $g: \om \to \bb S$, $\tau: \om \to \bb T$, and
$\sigma: \om \to \bb S$,  such as  the incompatibility operator 
\begin{equation}
  \label{eq:inc-def}
  \inc g := \curl \T \curl g.
\end{equation}
Let
\begin{subequations}
  \label{eq:Hcc-cd-dd-norms}
\begin{align}
  \| g \|_{\Hocc}^2
  & =
    \| g\|_{{\Ht}}^2 + \| \curl g \|_{{\Ht}}^2 + \| \inc g \|_{{\Ht}}^2 
  \\
  \| \tau \|_{\Hocd}^2
  & =  \|\tau \|_{{\Ht}}^2 +
    \| \div \tau \|_{{\Ht} }^2 + \| \sym \curl \T \tau \|_{{\Ht}}^2 +
    \| \curl \div \tau \|_{{\Ht}}^2
  \\
  \| \sigma \|_{\Hodd}^2
  & = \|  \sigma\|_{{\Ht}}^2 + \| \div \sigma \|_{{\Ht}}^2 +
    \| \div \div \sigma \|_{{\Ht}}^2,
\end{align}
\end{subequations}
and $\| \tau \|_{\HocdT} = \| \T \tau \|_{\Hocd}$. Let
\begin{align}
  \label{eq:H-cc-dd-cd-defn}
  \Hocc = \clos{\cl D(\om) \otimes \bb S}{ \| \cdot \|_{\Hocc}}, \quad
  \Hocd = \clos{\cl D(\om) \otimes \bb T}{ \| \cdot \|_{\Hocd}}, \quad
  \Hodd = \clos{\cl D(\om) \otimes \bb S}{ \| \cdot \|_{\Hodd}}.
\end{align}
The space $ \HocdT = \{ \tau^{\T}: \; \tau \in \Hocd\} $ will also be
needed. Clearly, in view of~\eqref{eq:Hminus-closure}, the spaces
$\Hocc, \Hocd$ and $\Hodd$ are subspaces of $\Ht \otimes \bb S$,
$\Ht \otimes \bb T$, and $\Ht \otimes \bb S$, respectively.

Certain subspaces of $\Hs s$, $\Hs{s}(\curl)$ and $\Hs{s}(\div)$, which we now define, occur often. Let
$\pol_1$ denote the space of linear polynomials.  Using the coordinate
vector $x$ in $\bb R^3$, define
\[
  \RT = \{a + b x: a \in \bb V, b \in \bb R \}, \qquad 
  \ND = \{a + d \times x: a, d \in \bb V\}.
\]
Let
\begin{align*}
  \LR
  & = \{ u \in L_2: \; (u, 1) = 0 \},\quad
    \Hs{s}_{\bb R} = \{ u \in \Hs s: \; u(1) =0\},
  \\
  \Hs{s}_{\RT}(\curl)
  & = \{ v \in \Hs{s}(\curl): v(r) = 0 \text{ for all } r \in \RT\},
  \\
  \Hs{s}_{\ND}(\div)
  & = \{ q \in \Hs{s}(\div): q(r) = 0 \text{ for all } r \in \ND\},
  \\
  \Hs{s}_{\pol_1}
  & = \{ w \in \Hs{s}: \; w(p) =0 \text{ for all } p \in \pol_1\}.
\end{align*}
Here and throughout, the action of a distribution $w$ on a function
$p$ in $\cl D(\bb R^3)$ is denoted by $w(p)$. In the above subspaces
of distributions, note that only the value of $p|_{\om}$ on $\om$ is needed 
to evaluate the action $w(p)$ since $w$ is supported on $\bar\om$.
Note also that $\Hs{s}_{\RT}(\curl)$ and  $\Hs{s}_{\ND}(\div)$
are closed subspaces of $\Hs{s}(\curl)$ and $\Hs{s}(\div)$.

\subsection{A diagram connecting the Sobolev spaces}

Using the above-defined notation, we can now precisely
introduce one of the objects of study in this paper. It is
the following diagram
connecting the above-defined Sobolev spaces of scalar-, vector-, and
matrix-valued distributions on $\bb R^3$:
\begin{equation}
  \label{eq:1}    
  \begin{tikzcd}
    [
    row sep=huge, 
    ampersand replacement=\&
    ]
    \Ho(\grad)
    \arrow{r}{\grad}
    \arrow[d, "\grad"]
    \&[1em]  
    \Ho(\curl)
    \arrow[r, "\curl"]
    \arrow[d, "\deff"]
    \&
    \Ho(\div) 
    \arrow{r}{\div}
    \arrow[d, "\frac 1 2 \T\dev\grad"]
    \&
    \LR
    \arrow[d, "\frac 1 3 \grad"]
    \\  
    \Ho(\curl)
    \arrow{r}{\deff}
    \arrow[d, "\curl"]
    \&
    \Hocc
    \arrow{r}{\curl}
    \arrow[d, "\T\curl"]
    \&
    \Hocd
    \arrow{r}{\div}
    \arrow[d, "\sym\curl\T"]
    \& \HtRT(\curl)
    \arrow[d, "\frac 1 2 \curl"]      
    \\
    \Ho(\div)
    \arrow[r, "\frac 1 2 \dev\grad"]
    \arrow[d, "\div"]
    \& 
    \HocdT
    \arrow{r}{\sym\curl}
    \arrow[d, "\div\T"]
    \&
    \Hodd
    \arrow{r}{\div}
    \arrow[d, "\div"]
    \&
    \HtND(\div)
    \arrow[d, "\div"]
    \\
    \LR
    \arrow{r}{\frac 1 3 \grad}
    \&
    \HtRT(\curl)
    \arrow{r}{\frac 1 2 \curl}
    \&
    \HtND(\div)
    \arrow{r}{\div}
    \&
    \HtPl
  \end{tikzcd}  
\end{equation}
Here $\dfm u = \sym \grad u$ for vector fields $u$ denotes the deformation operator, where $\grad u$ is
the matrix field whose $(i,j)$th component is $\d u_i/\d x_j$.
Note that information in the diagram~\eqref{eq:1} is repeated
across the diagonal, i.e., the diagram is symmetric about the
diagonal.
The properties collected in the next section show that each of the
indicated operators is linear and continuous in the  norms of
the indicated domain and codomain, and that each component cell in the
diagram commutes. A different  but similar diagram starting with analogous {\em spaces without boundary conditions} $H(\grad), H(\curl),$ and $H(\div)$,  is found later in Section~\ref{sec:duality}.

In the commutative diagram~\eqref{eq:1}, the ``objects'' (or
``vertices'') are the spaces. The ``morphisms'' (or ``arrows'') are
the indicated first-order differential operators. Compositions of
morphisms are referred to as ``paths''. Clearly, paths in~\eqref{eq:1}
always go right or down from an object. The following definition of a 
``2-complex'' is motivated by~\cite{olver1982differential}.

\begin{definition}
  \label{def:2complex}
  A path is a {\bf{complex}} if the composition of two successive
  morphisms in it vanish. We say that a path is a {\bf{2-complex}} if
  the composition of three successive morphisms in it vanish.
\end{definition}

We show (in the next section, in Theorem~\ref{thm:2-complex}) that all
paths in the diagram~\eqref{eq:1} are 2-complexes.  The analogous
diagram for spaces without boundary conditions also shares the same
property, as we shall see in Section~\ref{sec:duality}.

Before concluding this introduction, a few remarks on comparison with the BGG approach are in order. The BGG construction of~\cite{ArnolHu21,vcap2023bgg} produces
analogues of~\eqref{elasticity},~\eqref{hess-complex}
and~\eqref{divdiv-complex} with Sobolev spaces
$H^{q}\otimes \mathbb{W}$ or
$H(D, \mathbb{W}):=\{\sigma\in L_{2}\otimes \mathbb{W}: D\sigma \in
L_{2}\otimes \tilde{\mathbb{W}}\},$ for appropriate
$\bb W, \tilde{\bb{W}} \in \{ \bb S, \bb T\}$ and operators $D$ from
the above complexes.  The hessian, elasticity, and div-div complexes
were also studied individually in other
works~\cite{PaulyZulehner:23,geymonat2005some,PaulyZuleh20}.  It
should not be surprising that some individual results in this paper
may be alternately derived using the prior approaches, e.g., the
commutativity identities \eqref{eq:earlier-identities} are extensively
used in BGG works, and the regular decomposition for two of the
``slightly more regular spaces'' 
in Section~\ref{sec:slightly-more-regular}, 
$\toHcc$ and $\toHdd$, can be approached using the
technique of~\cite[Theorem 3]{ArnolHu21} with minor changes.
However, such individual results do not fully address the objectives of this
paper. For instance, the spaces defined in~\eqref{eq:Hcc-cd-dd-norms}
do not emerge from~\cite{ArnolHu21,vcap2023bgg} as canonical spaces
with a unified definition; rather, they exhibit a cohesive pattern
only through the perspective of the 2-complexes
in~\eqref{eq:1}. Consequently, the analytical results for these spaces,
such as regular decompositions, differ significantly from those
in~\cite{ArnolHu21,vcap2023bgg}. Moreover, the 2-complex
in~\eqref{eq:1} unifies several key spaces, including the Hessian,
elasticity, and $\div\div$ complexes, potentially inspiring novel
constructions across diverse applications. This unification can be 
reminiscent of the BGG diagram~\cite{ArnolHu21,vcap2023bgg}. However,
a critical distinction is that BGG diagrams involve full matrix spaces
requiring subsequent symmetry reduction, whereas~\eqref{eq:1} directly
incorporates spaces of tensors with the symmetrizations.

\subsection{Outline}

The  next section (Section~\ref{sec:cont-commut}) begins by  gathering a
number of identities from which the commutativity properties in the diagram~\eqref{eq:1} become
evident. We prove the 2-complex property of~\eqref{eq:1}, show how the
elasticity complex, the hessian complex and the div-div complex
emerges from the diagram. In Section~\ref{sec:regul-decomp}, we prove
that the newly introduced $H^{-1}$~based Sobolev spaces of weak regularity admit
decompositions with smoother component functions that vary
continuously with the decomposed function
(Theorems~\ref{thm:reg-dec-Hcc}, \ref{thm:reg-dec-Hdd}, and
\ref{thm:reg-dec-Hcd}).  We construct right inverses (in
Theorem~\ref{thm:rt-inv-cts}) of the operators in~\eqref{eq:1} as well
as of second-order differential operators that emerge from the
diagram, from which it follows that the ranges of the differential
operators considered are closed. This can be used to prove exactness
of derived complexes.  Slightly smoother versions of the matrix-valued
Sobolev spaces are then considered in
Section~\ref{sec:slightly-more-regular} and shorter regular
decompositions for them are proved. Finally, in
Section~\ref{sec:duality}, we mention extensions to the case of
analogous spaces without boundary conditions. The main result of that
section is Theorem~\ref{thm:duality} which shows how the diagrams of spaces with
and without boundary conditions are in correspondence through duality.

\section{Continuity, commutativity, and 2-complex properties}
\label{sec:cont-commut}

In this section, we show that the diagram~\eqref{eq:1} is a commuting
diagram and has the 2-complex property.

In addition to $\sym, \dev, \tr$ and $\skw \tau = \tau - \sym \tau$,
we use the algebraic operation $S: \bb M \to \bb M$ defined by
$
  S \tau = \tau^{\T} - \tr(\tau) \id, 
$  
whose inverse can be easily computed to be
\[
  S^{-1} \tau = \tau^{\T}  - \frac 1 2 \tr (\tau) \id.
\]
We often use the summation convention and the alternating symbol
$\veps^{ijk}\equiv \veps_{ijk}$ whose value equals $+1, -1,$ or 0
according to whether $ijk$ is a even, odd or not a permutation of 1,
2, 3.  Using Cartesian unit vectors $e_i\equiv e^i$ and the summation
convention, we write a vector $v$ as $v = v_i e^i$.  Using $\veps$,
one can express an isomorphism between skew-symmetric matrices in
$\bb K = \skw \bb M$ and their axial vectors in $\bb V$, given by
$\mskw: \bb V \to \bb K$,
$\mskw(v^i e_i) = -\veps^{ijk} v_k e_i \otimes e_j$.  Let
$\vskw : \bb M \to \bb V$ be defined by
$\vskw = \mskw^{-1} \circ \skw.$ For distributional  fields $w$,
vector fields $v$ and matrix fields $\tau$ on three-dimensional
domains, it is easy to see that the following identities hold:
\begin{subequations}
  \label{eq:earlier-identities}
  \begin{gather}
    \label{eq:div-mskw}
  \div \mskw v = -\curl v,
  \\
  \label{eq:mskw-grad}
  \mskw \grad w  = -\curl (w \id),
  \\
  \label{eq:mskw-curl}
  \mskw \curl v  = 2 \skw \grad v,
  \\
  \label{eq:skw-curl}
  2 \skw \curl \tau = \mskw \div S \tau,
  \\
  \label{eq:S-grad}
  S \grad v = -\curl \mskw v,
  \\
  \label{eq:tr-curl}
  \tr \curl \tau = -2  \div \vskw \tau,
\end{gather}
\end{subequations}



We start with two simple lemmas. Lemma~\ref{lem:commute-identities}
contains identities involving second-order partial differential
operators and Lemma~\ref{lem:dense-mean0} gives
density of the following smooth spaces with moment conditions:
\begin{equation}
  \label{eq:11}  
\begin{aligned}
  \cl D_{\bb R}
  & = \{ \varphi \in \cl D(\om): (\varphi, 1) = 0\},
  \\
  \cl D_{\RT}
  & = \{ \varphi \in \cl D(\om) \otimes \bb V  : (\varphi, r) = 0
    \text{ for all } r \in\RT \},
  \\
  \cl D_{\ND}
  & = \{ \varphi \in \cl D(\om) \otimes \bb V  : (\varphi, r) = 0
    \text{ for all } r \in\ND \},
  \\
  \cl D_{\pol_1}
  & = \{ \varphi \in \cl D(\om)  : (\varphi, p) = 0
    \text{ for all } p \in\pol_1 \}.
\end{aligned}
\end{equation}
For any two norms $\|\cdot \|_1$ and $\| \cdot \|_2$,
we write
\[
  \| a\|_1 \lesssim \| b \|_2
\]
to indicate that there is some constant $C>0$ independent of $a$ and
$b$ such that the inequality  $\| a\|_1 \le C \| b \|_2$ holds.

\begin{lemma}\label{lem:commute-identities}
  The identities
  \begin{gather}
    \label{eq:2}
    \div \T \curl \tau = \curl \div \T \tau, 
    \\
    \label{eq:3}
    \curl \T \grad u = \T \grad \curl u = \T \dev \grad \curl u,
    \\
    \label{eq:4}
    \div \sym \curl \T \tau = \frac 1 2 \curl \div  \tau, 
    \\
    \label{eq:5}
    \curl \dfm u = \frac 1 2 \T \grad \curl u = \frac 1 2 \T \dev \grad \curl u,
    \\
    \label{eq:6}
    \frac 1 2 \div \T \dev \grad u = \frac 1 3 \grad \div u
  \end{gather}
  hold for any vector-valued distribution $u$ and matrix-valued
  distribution $\tau$.
\end{lemma}
\begin{proof}
  To prove~\eqref{eq:2}, 
  we express row-wise curl using $\veps^{ijk}$, the
  summation convention, and standard Cartesian unit vectors~$e_i$, 
  \begin{align*}
    \div \T \curl \tau
    & = e_i\d_j [\T \curl \tau]^{ij}
      = e_i\d_j [\curl \tau]^{ji} = e_i\d_j \veps^{ikl} \d_k \tau_{jl}
      = e_i\veps^{ikl} \d_k  \d_j[\T \tau]_{lj}
    \\
    & = e_i\veps^{ikl} \d_k  [\div \T \tau]_{l} = \curl \div \T \tau.
  \end{align*}
  The first equality in~\eqref{eq:3} is proved similarly. For the
  second equality in~\eqref{eq:3}, it suffices to note that the
  $(i, j)$th component of the matrix field $\grad \curl u$ equals
  $\d_i \veps^{jkl} \d_k u_l$, so its trace, obtained with $i=j$ in
  this expression, vanishes.

  Identity~\eqref{eq:4} follows using $\div \circ \curl =0$
  and~\eqref{eq:2}: 
  \begin{align*}
    \div \sym \curl \T \tau
    & =  \div \frac 1 2 \big( \curl \T \tau + \T \curl \T \tau \big)
    \\
    & = \frac 1 2 \div  \T \curl \T \tau 
    = \frac 1 2 \curl \div  \T \T \tau.      
  \end{align*}
  Equation~\eqref{eq:5} follows from~\eqref{eq:3} and
  $\curl \circ \grad =0$ in an analogous fashion.
  The proof of~\eqref{eq:6} using analogous techniques is also elementary.
\end{proof}

\begin{lemma}
  \label{lem:dense-mean0}
  The spaces in~\eqref{eq:11}, namely
  $\cl D_{\bb R}$,
  $ \cl D_{\RT}$,
  $\cl D_{\ND}$, and 
  $\cl D_{\pol_1}$, are  dense in 
  $\Hs{s}_{\bb R}$
  $\Hs{s}_{\RT}(\curl)$,
  $\Hs{s}_{\ND}(\div)$, and 
  $\Hs{s}_{\pol_1}$, respectively, for any $s \in \bb R$.
\end{lemma}
\begin{proof}
  The proofs of all the four stated density results are similar. We
  only detail the second.  Fix a nontrivial scalar function
  $b(x) \in \cl D(\om)$ satisfying $b(x) \ge 0$. Let $\rho_i$ be a
  basis of the four-dimensional space $\RT,$ normalized so that
  \begin{equation}
    \label{eq:12}
    (b \rho_i, \rho_j) = \delta_{ij}
  \end{equation}
  where $\delta_{ij}$ denotes the Kronecker delta symbol.
  Let $v \in \Hs{s}_{\RT}(\curl)$. 
  In view of   \eqref{eq:13}, we can find a
  sequence $\varphi_n \in \cl D(\om) \otimes \bb V$ such that
  \begin{equation}
    \label{eq:14}
    \lim_{n \to \infty} \| \varphi_n - v \|_{\Hs{s}(\curl)} =0.
  \end{equation}
  Let 
  \begin{equation}
    \label{eq:15}
    \psi_n (x) = \varphi_n(x) - \sum_{i=1}^4 (\varphi_n, \rho_i) \; b(x) \,\rho_i(x).
  \end{equation}
  Then $\psi_n$ is in $\cl D(\om)\otimes \bb V$ and
  $(\psi_n, \rho_j) =0$ due to~\eqref{eq:12}, i.e.,
  $\psi_n \in \cl D_{\RT}$.
  Moreover, $\psi_n$ converges to $v$ in $\Hs{s}(\curl)$ as we now show:
  indeed, since $v(\rho_i) =0$,
  \[
    (\varphi_n, \rho_i)
    =
    (\varphi_n -v)(\rho_i) \le \| \varphi_n - v\|_{\Hs{s}} \| \rho_i\|_{H^{-s}} 
  \]
  by \eqref{eq:Hs-dual-char}. Hence
  \eqref{eq:14} implies that
  \begin{equation}
    \label{eq:phin-rhoi}
    \lim_{n \to \infty} (\varphi_n, \rho_i) = 0
  \end{equation}
  for each~$\rho_i$. Now it is evident from~\eqref{eq:15} that $\psi_n$
  converges to $v$ in $\Hs{s}$~norm since $\varphi_n$ does.
  Moreover,
  \begin{align*}
    \curl (\psi_n - v)
    & = \curl(\varphi_n - v) -
       \sum_{i=1}^4 (\varphi_n, \rho_i) \; \curl( b\,\rho_i),
  \end{align*}
  where, on the right hand side, the first term converges to zero in
  $\Hts{s}$ by~\eqref{eq:14}, and the second term converges to zero
  by~\eqref{eq:phin-rhoi}. Thus $\psi_n$ and $\curl \psi_n$ converges
  in $\Hts s$ to $v$ and $\curl v$, respectively.  Hence $\cl D_{\RT}$
  is dense in $\Hs{s}_{\RT}(\curl)$.
\end{proof}


\begin{theorem}
  \label{thm:diagram-commute-cty}
  The diagram \eqref{eq:1} commutes and every differential operator in
  it maps continuously (with respect to the norms of the indicated
  domains and codomains).
\end{theorem}
\begin{proof}
  {\em Commutativity.}
  The commutativity of the diagram cells in (row, column)-positions
  $(2, 3), (1, 2),$ and $(1, 3)$ follows respectively from
  identities~\eqref{eq:4},~\eqref{eq:5}, and~\eqref{eq:6} of
  Lemma~\ref{lem:commute-identities}. The commutativity at positions
  across the diagonal also follow from these. At the remaining
  positions, it is obvious.

  Next, let us prove the stated continuity properties.
  The continuity of the operators in the first row
  and column is standard. For the remaining operators, we use
  \eqref{eq:di-cty} and the following steps.
  We begin the maps in the second row
  of~\eqref{eq:1}.

  {\em Continuity of $\dfm : \Ho(\curl) \to \Hocc$.}  For any
  $u \in \cl D (\om) \otimes \bb V \subset \Ho(\curl)$, note that
  $g = \dfm u \in \cl D(\om ) \otimes \bb S \subset \Hocc$ satisfies
  \[
    \curl g
     = \frac 1 2 \T \dev \grad \curl u = \frac 1 2 \T  \grad \curl u
   \]
   due to~\eqref{eq:5} and~\eqref{eq:3}. This implies that
   $\inc g := \curl \T \curl g =0.$
   Hence, using \eqref{eq:di-cty}, 
   \begin{align*}
     \| \dfm u \|_{\Hocc}^2
     & = 
       \| g\|_{\Ht}^2 + \| \curl g \|_{\Ht}^2 + \| \inc g \|_{\Ht}^2
     \\
     & = \frac 1 2 \|\grad u \|_{\Ht}^2 +  \frac 1 4 \| \grad \curl u \|_{\Ht}^2
      \lesssim  \| u \|_{H(\curl)}^2,
   \end{align*}
   which proves the continuity of the deformation operator by density.

   {\em Continuity of  $\curl : \Hocc \to \Hocd$.} It suffices to
   observe that $\tau = \curl g$, for any
   $g \in \cl D(\om) \otimes \bb S \subset \Hocc$, satisfies
   $\tau = \dev \tau \in \cl D(\om) \otimes \bb T \subset \Hocd$ and
   \begin{align*}
     \div \tau & = 0,
     & \curl\div \tau
     & = 0,
     & \sym\curl \T \tau & = \inc g.
   \end{align*}
   This shows that $\| \tau \|_{\Hocd} \lesssim \|g \|_{\Hocc}$ and the
   continuity follows by density.

   {\em Continuity of $\div : \Hocd \to \HtRT(\curl)$.} Let
   $\tau \in \cl \D(\om) \otimes \bb T \subset \Hocd$.  Then, for any
   $r = a + b x \in \RT$, $a\in \bb V, b \in \bb R$, we have
   \begin{align*}
     (\div \tau, r)
     & = -(\tau,   \grad r) = -(\tau, b \id) = 0 
   \end{align*}
   since $\tau : \id$ vanishes for $\tau(x) \in \bb T$.  Next, by the
   definition of $\Hocd$~norm,
   \begin{align*}
     \| \div \tau \|_{{\Ht}}^2 + \| \curl \div \tau \|_{{\Ht}}^2
     \le \| \tau \|_{\Hocd}^2.
   \end{align*}
   Since the left hand side equals $\| \div \tau \|_{\Ht(\curl)}^2$,
   the continuity follows by density.

   {\em Continuity of $\dev \grad : \Ho(\div) \to \HocdT$.}
   Let
   $\tau = \dev \grad q$ for some
   $q \in \cl D(\om) \otimes \bb V \subset \Ho(\div).$
   Apply~\eqref{eq:6} to get 
   \begin{align*}
     \div \T \tau & = \frac 2 3 \grad \div q,
   \end{align*}
   which implies $\curl\div \T \tau = 0$. Also, since
   \[
     \sym\curl \tau = \sym\curl \frac 1 3 (\div q) \id,
   \]
   all terms in the norm $\| \tau \|_{\HocdT}$ can be bounded by the
   ${\Ht}$-norms of the first order of derivatives of $\div q$ and
   $q$, so using \eqref{eq:di-cty},
   $\| \tau \|_{\HocdT} \lesssim \| q\|_{H(\div)}$ and the continuity
   follows by density. 
   
   {\em Continuity of $\sym\curl \T : \Hocd \to \Hodd$.} This is a
   bounded operator since $\sigma = \sym \curl \T \tau$ for any
   $\tau \in \cl D(\om) \otimes \bb T \subset \Hocd$ satisfies, due
   to~\eqref{eq:4},
   \begin{align*}
     \div \sigma &  = \frac 1 2 \curl \div \tau,
   \end{align*}
   which in turn implies $\div \div \sigma = 0.$ Thus
   $\| \sigma\|_{\Hodd } \lesssim \| \tau \|_{\Hocd}$.

   {\em Continuity of  $\div: \Hodd \to \HtND(\div)$.} First note that
   for
   any $\sigma \in \cl D(\om) \otimes \bb S \subset \Hodd$
   and $r = a + b \times x \in \ND$, $a, b \in \bb V$, 
   \begin{align*}
     (\div \sigma, r)
     & = -(\sigma, \grad r)
       = 0
   \end{align*}
   because
   $\sigma: \grad r = \sigma_{ij} \d_j [b \times x]_i = \veps^{ijp}
   \sigma_{ij} b_p$ vanishes due to the symmetry
   $\sigma_{ij} = \sigma_{ji}$. Thus
   $\div \sigma$ is in $\HtND(\div)$.
   Moreover, by the definition of the
   $\Hodd$-norm
   \[
     \| \div \sigma \|_{\Ht}^2 + \| \div \div \sigma \|_{\Ht}^2
     \le \| \sigma \|_{\Hodd}^2.
   \]
   The left hand side exactly equals
   $\| \div \sigma \|_{\Ht(\div)}^2$ so the continuity follows by density.

   {\em Continuity of $\grad : \LR \to \HtRT(\curl)$.}  Let
   $u \in \cl D(\om) \cap \LR = \cl D_{\bb R}$.  By
   Lemma~\ref{lem:dense-mean0} $\cl D_{\bb R}$ is dense in $\LR$.
   For any $r = a + b x \in \RT$, $a\in \bb V, b\in \bb R$, integrating
   by parts using the compact support of $u$, 
   \begin{align*}
     (\grad u, r)
     = -(u, \div r)
     = -3 (u, b) = 0 
   \end{align*}
   since $u$ has zero mean value on $\om$. Hence $\grad u$ is in
   $\HtRT(\curl)$, so by Lemma~\ref{lem:dense-mean0},
   $\grad \LR \subseteq  \HtRT(\curl)$.
   The needed boundedness estimate is immediate
   from \eqref{eq:di-cty}.

   {\em Continuity of $\curl: \HtRT(\curl) \to \HtND(\div)$.}
   By Lemma~\ref{lem:dense-mean0}, $\cl D_{\RT}$ is dense in
   $\HtRT(\curl)$. Let $v \in \cl D_{\RT}$. Then, given any 
   $r = a + b \times x \in \ND$ with some $a, b \in \bb V$,
   \[
     (\curl v, r) = (v, \curl r) = 2 (v, b) = 0
   \]
   since $v$ is orthogonal to $\RT$. Hence $\curl v$ is in
   $\HtND(\div)$. Combined with the obvious norm bound, the
   continuity follows from density.

   {\em Continuity of $\div:  \HtND(\div) \to \HtPl$.}
   Let $q \in \cl D_{\ND}$. Then, for any $p \in \pol_1$,
   \[
     (\div q, p) = -( q, \grad p) = 0 
   \]
   since $\grad p$ is in $\ND$ and $q$ is orthogonal to $\ND$.
   Therefore $\div q$ is in $\HtPl$.  In view of the obvious
   norm bound, the density of $\cl D_{\ND}$  in
   $\HtND(\div)$ (given by Lemma~\ref{lem:dense-mean0}) finishes
   the proof.
\end{proof}

\begin{theorem}
  \label{thm:2-complex}
  All paths in the diagram~\eqref{eq:1} are 2-complexes.
\end{theorem}
\begin{proof}
  Consider paths of the following form:
  
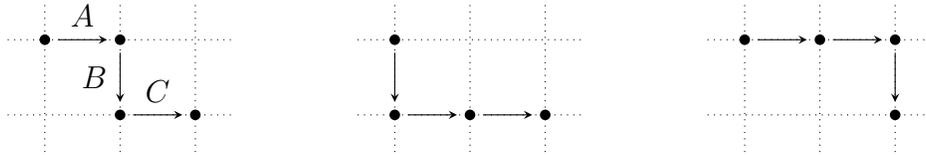
\[
    \begin{tikzpicture}[>=stealth, nodes={outer sep=1pt}]
      \clip (0.5, 0.5) rectangle (3.5, 2.5);
      \draw[dotted] (0,0) grid (4, 4);

      \node (U) at (1, 3) {};
      \node (V) at (1, 1) {};
      \node (W) at (2, 3) {};
      \node (X) at (1, 2) {};
      \node (Y) at (2, 2) {};
      \node (Z) at (2, 1) {};
      \node (O) at (3, 1) {};
      \node (Q) at (3, 2) {};
      

      \fill (U) circle (2pt);
      \fill (W) circle (2pt);
      \fill (X) circle (2pt);
      \fill (Y) circle (2pt);
      \fill (Z) circle (2pt);
      \fill (O) circle (2pt);

      \draw[->] (X) -- (Y) node[midway, above] {$A$};
      \draw[->] (Y) -- (Z) node[midway, left] {$B$};
      \draw[->] (Z) -- (O) node[midway, above] {$C$};  
    \end{tikzpicture}
    \qquad\qquad
    \begin{tikzpicture}[>=stealth, nodes={outer sep=1pt}]
      \clip (0.5, 0.5) rectangle (3.5, 2.5);
      \draw[dotted] (0,0) grid (4, 4);
      
      \node (U) at (1, 3) {};
      \node (V) at (1, 1) {};
      \node (W) at (2, 3) {};
      \node (X) at (1, 2) {};
      \node (Y) at (2, 2) {};
      \node (Z) at (2, 1) {};
      \node (O) at (3, 1) {};

      \fill (U) circle (2pt);
      \fill (V) circle (2pt);
      \fill (W) circle (2pt);
      \fill (X) circle (2pt);
      \fill (Z) circle (2pt);
      \fill (O) circle (2pt);

      \draw[->] (X) -- (V); 
      \draw[->] (V) -- (Z);
      \draw[->] (Z) -- (O);
    \end{tikzpicture}
    \qquad\qquad
    \begin{tikzpicture}[>=stealth, nodes={outer sep=1pt}]
      \clip (0.5, 0.5) rectangle (3.5, 2.5);
      \draw[dotted] (0,0) grid (4, 4);
      
      \node (U) at (1, 3) {};
      \node (V) at (1, 1) {};
      \node (W) at (2, 3) {};
      \node (X) at (1, 2) {};
      \node (Y) at (2, 2) {};
      \node (Z) at (2, 1) {};
      \node (O) at (3, 1) {};
      \node (Q) at (3, 2) {};
      

      \fill (X) circle (2pt);
      \fill (Y) circle (2pt);
      \fill (O) circle (2pt);
      \fill (Q) circle (2pt);

      \draw[->] (X) -- (Y); 
      \draw[->] (Y) -- (Q);
      \draw[->] (Q) -- (O);
    \end{tikzpicture}
  \]
  If the first path above is a 2-complex, then by the commutativity
  properties of Theorem~\ref{thm:diagram-commute-cty}, the second and
  third are also 2-complexes. Proving that the first path is a
  2-complex, i.e., showing that $C \circ B \circ A =0$ for all such
  $A, B, C$ in~\eqref{eq:1}, is (tedious but) elementary using the identities of
  \eqref{eq:earlier-identities} and Lemma~\ref{lem:commute-identities}.
  For example,  with
  $A = \grad: \Ho(\grad) \to \Ho(\curl), B = \dfo: \Ho(\curl) \to \Hocc,$ and $
  C = \curl: \Hocc \to \Hocd$,
  we have
  $C \circ B \circ A = \curl \circ \frac 1 2 (\grad + \T \grad ) \circ
  \grad = \frac 1 2 \curl \circ \T \grad \circ \grad = \frac 1 2 \curl
  \circ \T \grad \circ \grad = \frac 1 2 \T \grad \circ \curl \circ
  \grad =0 $, where we have used the identity~\eqref{eq:3} of
  Lemma~\ref{lem:commute-identities}.
  Similarly, it is elementary to see that $B \circ H \circ G =0$ 
  in paths taken from~\eqref{eq:1} of the form in the first diagram below,
  
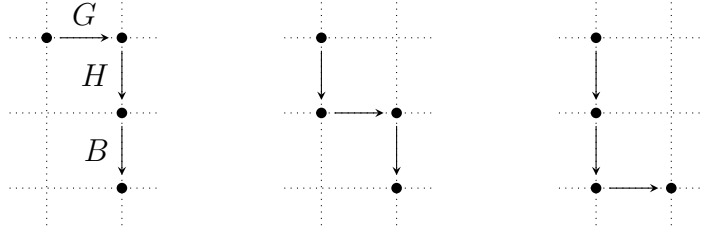
\[
    \begin{tikzpicture}[>=stealth, nodes={outer sep=1pt}]
      \clip (0.5, 0.5) rectangle (2.5, 3.5);
      \draw[dotted] (0,0) grid (4, 4);
      \node (U) at (1, 3) {};
      \node (V) at (1, 1) {};
      \node (W) at (2, 3) {};
      \node (X) at (1, 2) {};
      \node (Y) at (2, 2) {};
      \node (Z) at (2, 1) {};
      \node (O) at (3, 1) {}; 
      \fill (U) circle (2pt);
      \fill (W) circle (2pt);
      \fill (Y) circle (2pt);
      \fill (Z) circle (2pt);
      \fill (O) circle (2pt);
      \draw[->] (U) -- (W) node[midway, above] {$G$};
      \draw[->] (W) -- (Y) node[midway, left] {$H$};
      \draw[->] (Y) -- (Z) node[midway, left] {$B$};  
    \end{tikzpicture}
    \qquad\qquad
    \begin{tikzpicture}[>=stealth, nodes={outer sep=1pt}]
      \clip (0.5, 0.5) rectangle (2.5, 3.5);
      \draw[dotted] (0,0) grid (4, 4);
      \node (U) at (1, 3) {};
      \node (V) at (1, 1) {};
      \node (W) at (2, 3) {};
      \node (X) at (1, 2) {};
      \node (Y) at (2, 2) {};
      \node (Z) at (2, 1) {};
      \node (O) at (3, 1) {};


      \fill (U) circle (2pt);
      \fill (X) circle (2pt);
      \fill (Y) circle (2pt);
      \fill (Z) circle (2pt);
      \fill (O) circle (2pt);

      \draw[->] (U) -- (X);
      \draw[->] (X) -- (Y);
      \draw[->] (Y) -- (Z);
    \end{tikzpicture}
    \qquad\qquad
    \begin{tikzpicture}[>=stealth, nodes={outer sep=1pt}]
      \clip (0.5, 0.5) rectangle (2.5, 3.5);
      \draw[dotted] (0,0) grid (4, 4);
      \node (U) at (1, 3) {};
      \node (V) at (1, 1) {};
      \node (W) at (2, 3) {};
      \node (X) at (1, 2) {};
      \node (Y) at (2, 2) {};
      \node (Z) at (2, 1) {};
      \node (O) at (3, 1) {};


      \fill (U) circle (2pt);
      \fill (V) circle (2pt);
      \fill (X) circle (2pt);
      \fill (Z) circle (2pt);
      \fill (O) circle (2pt);

      \draw[->] (U) -- (X);
      \draw[->] (X) -- (V);
      \draw[->] (V) -- (Z);
    \end{tikzpicture}
  \]
  so by commutativity, all paths of the three types shown above are
  also 2-complexes. These types of paths exhaust all possibilities.
\end{proof}

Next, we consider the fundamental {\em  second-order differential
  operators} inherent in~\eqref{eq:1}: the incompatibility operator
defined in~\eqref{eq:inc-def}, the hessian operator
$ \hess:= \dfo \circ \grad, $ $\grad \circ \div$, $\curl\circ \div$
and $\div\circ \div$.  These appear along the diagonals of the
following diagram:
\begin{equation}
  \label{eq:1-with-2nd-order}
  \begin{tikzcd} 
    [
    row sep=hypercomplexsize,
    ampersand replacement=\&
    ]
    \Ho(\grad)
    \arrow{r}{\grad}
    \arrow[d, "\grad"{sloped}]
    \arrow[dr, "\hess"{sloped}]
    \&[1em]  
    \Ho(\curl)
    \arrow[r, "\curl"]
    \arrow[d, "\deff"{sloped}]
    \arrow[dr, "\curl\dfo"{sloped}]
    \&
    \Ho(\div) 
    \arrow{r}{\div}
    \arrow[d, "\frac 1 2 \T\dev\grad"{sloped}]
    \arrow[dr, "\frac 1 3 \grad \div"{sloped}]            
    \&
    \LR
    \arrow[d, "\frac 1 3 \grad"{sloped}]
    \\  
    \Ho(\curl)
    \arrow{r}{\deff}
    \arrow[d, "\curl"{sloped}]
    \&
    \Hocc
    \arrow{r}{\curl}
    \arrow[d, "\T\curl"{sloped}]
    \arrow[dr, "\inc"{sloped}]              
    \&
    \Hocd
    \arrow{r}{\div}
    \arrow[d, "\sym\curl\T"{sloped}]
    \arrow[dr, "\frac 1 2 \curl \div"{sloped}]            
    \& \HtRT(\curl)
    \arrow[d, "\frac 1 2 \curl"]      
    \\
    \Ho(\div)
    \arrow[r, "\frac 1 2 \dev\grad"]
    \arrow[d, "\div"{sloped}]
    \& 
    \HocdT
    \arrow{r}{\sym\curl}
    \arrow[d, "\div\T"{sloped}]
    \&
    \Hodd
    \arrow{r}{\div}
    \arrow[d, "\div"{sloped}]
    \arrow[dr, "\div \div"{sloped}]
    \&
    \HtND(\div)
    \arrow[d, "\div"{sloped}]
    \\
    \LR
    \arrow{r}{\frac 1 3 \grad}
    \&
    \HtRT(\curl)
    \arrow{r}{\frac 1 2 \curl}
    \&
    \HtND(\div)
    \arrow{r}{\div}
    \&
    \HtPl
  \end{tikzcd}  
\end{equation}
The second-order operators below the diagonal are not shown as they
are mirrored by those shown above the diagonal. Note that each
indicated second-order operator is a composition of adjacent
first-order operators at the top and right, or by commutativity, the
adjacent left and bottom  operators. We can now read off less regular versions of
\eqref{elasticity}, \eqref{hess-complex}, and \eqref{divdiv-complex} from the diagram~\eqref{eq:1-with-2nd-order}, as stated next.

\begin{corollary}
  \label{cor:3complexes}
  The following paths in~\eqref{eq:1-with-2nd-order}
  are complexes:
  \begin{enumerate}
  \item The hessian complex:
    \begin{equation}
      \label{eq:hessian-complex}
    \begin{tikzcd} 
      [
      column sep=huge,
      ampersand replacement=\&
      ]
      \Ho(\grad)
      \arrow[dr, "\hess"{sloped}]
      \&[1em]  
      \&
      \&
      \\  
      \&
      \Hocc
      \arrow{r}{\curl}
      \&
      \Hocd
      \arrow{r}{\div}
      \& \HtRT(\curl)
    \end{tikzcd}
    \end{equation}

\item The elasticity complex:
  \begin{equation}
    \label{eq:elasticity-complex}
    \begin{tikzcd} 
      [
      column sep=huge,
      ampersand replacement=\&
      ]
      \Ho(\curl)
      \arrow{r}{\deff}
      \&
      \Hocc
      \arrow[dr, "\inc"{sloped}]              
      \&
      \&
      \\
      \& 
      \&
      \Hodd
      \arrow{r}{\div}
      \&
      \HtND(\div)
    \end{tikzcd}  
  \end{equation}

\item The $\div\div$ complex:
  \begin{equation}
    \label{eq:div-div-complex}
    \begin{tikzcd} 
      [
      column sep=huge,
      ampersand replacement=\&      
      ]
      \Ho(\div)
      \arrow[r, "\frac 1 2 \dev\grad"]
      \& 
      \HocdT
      \arrow{r}{\sym\curl}
      \&
      \Hodd
      \arrow[dr, "\div \div"{sloped}]
      \&
      \\
      \&
      \&
      \&
      \HtPl
    \end{tikzcd}
  \end{equation}
  \end{enumerate}
\end{corollary}
\begin{proof}
  These statements follow from the 2-complex properties of
  Theorem~\ref{thm:2-complex} and elementary manipulations with
  first-order differential operators. For instance, to prove the last,
  $ \sym\curl \circ \dev \grad w = -\frac 1 3 \sym\curl (\tr(\grad
  w)\id) = -\frac 1 3 \sym \mskw \grad(\tr(\grad w)) = 0,$ where we
  have used~\eqref{eq:mskw-grad}; and of course,
  $\div \div \circ \sym \curl$ must vanish due to the 2-complex
  property of Theorem~\ref{thm:2-complex}.
\end{proof}

\section{Regular decompositions, density, and continuous right inverses}
\label{sec:regul-decomp}

Right inverses of exterior derivatives that are continuous in
appropriate Sobolev norms were given in~\cite{CostaMcInt10}, inspired
by the classical work of~\cite{Bogov79}. In this section, we leverage
their results to show regular decompositions of some Sobolev spaces of
matrix fields, prove density of smooth functions in them, and
construct right inverses of the differential operators
in~\eqref{eq:1-with-2nd-order}.

The right inverses of the derivative operators
in~\eqref{eq:1-with-2nd-order} that act on or produce matrix fields
are the subscripted $D$ and $R$~operators labeling the diagonally
upward arrows and rightward arrows in the following diagram:
\begin{equation}
  \label{eq:reverse-arrows}
  \begin{tikzcd} 
    [
    row sep=hyper2,
    column sep=large,
    ampersand replacement=\&
    ]
    \Ho(\grad)
    \arrow{r}{\grad}
    \arrow[d, "\grad"{sloped}]
    \arrow[dr, "\hess"{sloped}]
    \&[1em]  
    \Ho(\curl)
    \arrow[r, "\curl"]
    \arrow[d, "\deff"{sloped}]
    \arrow[dr, "\curl\dfo"{sloped}]
    \arrow[l, shift left=2, "{\Tg}"{sloped}, swap]
    \&
    \Ho(\div) 
    \arrow{r}{\div}
    \arrow[d, "\;\;\frac 1 2\!\! \T\dev\grad\!"{sloped}]
    \arrow[dr, "\frac 1 3 \grad \div"{sloped}]
    \arrow[l, shift left=2, "{\Tc}"{sloped}, swap]
    \&
    \LR
    \arrow[d, "\frac 1 3 \grad"{sloped}]
    \arrow[l, shift left=2, "{\Td}"{sloped}, swap]
    \\  
    \Ho(\curl)
    \arrow{r}{\deff}
    \arrow[d, "\curl"{sloped}]
    \arrow[dr, "\T \curl \dfo"{sloped}]
    \&
    \Hocc
    \arrow{r}{\curl}
    \arrow[d, "\T\curl"{sloped}]
    \arrow[dr, "\inc"{sloped}]
    \arrow[ul, shift left=2, "\Dgg"{sloped}, swap]
    \arrow[l, shift left=2, "{\Rgg}"{sloped}, swap]
    \&
    \Hocd
    \arrow{r}{\div}
    \arrow[d, "\sym\curl\T"{sloped}]
    \arrow[dr, "\quad\frac 1 2 \curl \div"{sloped}]
    \arrow[ul, shift left=2, "\Dgc"{sloped}, swap]
    \arrow[l, shift left=2, "{\Rgc}"{sloped}, swap]
    \&
    \HtRT(\curl)
    \arrow[d, "\frac 1 2 \curl"{sloped}]
    \arrow[ul, shift left=2, "{\Dgd}"{sloped}, swap, shorten <=1ex]
    \arrow[l, shift left=2, "{\Rgd}"{sloped}, swap]
    \\
    \Ho(\div)
    \arrow[r, "\frac 1 2 \dev\grad"]
    \arrow[d, "\div"{sloped}]
    \arrow[dr, "\frac 1 3 \grad \div"{sloped}]
    \& 
    \HocdT
    \arrow{r}{\sym\curl}
    \arrow[d, "\div\T"{sloped}]
    \arrow[dr, "\frac 1 2 \curl \div \T"{sloped}]
    \arrow[ul, shift left=2, "\Dgc \T"{sloped}, swap]
    \arrow[l, shift left=2, "{\RgcT}"{sloped}, swap]
    \&
    \Hodd
    \arrow{r}{\div}
    \arrow[d, "\div"{sloped}]
    \arrow[dr, "\div \div"{sloped}]
    \arrow[ul, shift left=2, "{\Dcc}"{sloped}, swap]
    \arrow[l, shift left=2, "{\Rcc}"{sloped}, swap]
    \&
    \HtND(\div)
    \arrow[d, "\div"{sloped}]
    \arrow[ul, shift left=2, "{\Dcd}"{sloped}, swap, shorten <=1ex]
    \arrow[l, shift left=2, "{\Rcd}"{sloped}, swap]
    \\
    \LR
    \arrow{r}{\frac 1 3 \grad}
    \&
    \HtRT(\curl)
    \arrow{r}{\frac 1 2 \curl}
    \arrow[l, shift left=2, "{\Rg}"{sloped}, swap]
    \arrow[ul, shift left=2, "{\Dgd}"{sloped}, swap, shorten <=1ex]
    \&
    \HtND(\div)
    \arrow{r}{\div}
    \arrow[l, shift left=2, "{\Rc}"{sloped}, swap]
    \arrow[ul, shift left=2, "\T{\Dcd}"{sloped}, swap, shorten <=1ex]
    \&
    \HtPl
    \arrow[ul, shift left=2, "{\Ddd}"{sloped}, swap]
    \arrow[l, shift left=2, "{\Rd}"{sloped}, swap]
  \end{tikzcd}  
\end{equation}
The downward arrows in~\eqref{eq:reverse-arrows} can also be provided
with corresponding upward right inverses. They are not marked to
reduce clutter and because the same information is contained in the
horizontal arrows back and forth.  Below, we detail the construction
of each new right inverse operator. Note that some of the $D$ and $R$
operators map from and into {\em subspaces} of the respective domains
and codomains indicated in~\eqref{eq:reverse-arrows}. The codomain
subspaces consist of (more) {\em regular} functions. The domain
subspaces are {\em kernels} of one of the (multiple) operators acting
on the space. These subspaces are given precisely in the case by case results
below (and summarized in Theorem~\ref{thm:rt-inv-cts}). Throughout, we denote the null space of a linear operator $A: X \to Y$
by
\[
\ker(A:X).
\]
Note, e.g., $\ker(\curl: \Ho(\curl))$ is different
from $\ker(\curl: \Hocc) = \{g \in \Hocc: \curl g =0\}$.

Regular decompositions for standard Sobolev spaces based on the
exterior derivative can be inferred from the results
of~\cite{CostaMcInt10}. In this section, we also show how to combine
their results with our previous results to produce regular
decompositions for the following new spaces of matrix fields:
\begin{align*}
  \hHcc & = \{ g \in \Ht \otimes \bb S: \curl g \in \Ht \otimes \bb V,
          \inc g \in \Ht \otimes \bb S \},
  \\
  \hHcd & = \{ \tau \in \Ht \otimes \bb T: \div \tau \in \Ht \otimes \bb V,
          \sym\curl \tau \in \Ht \otimes \bb S,
          \curl\div\tau \in \Ht \otimes \bb V \},
  \\
  \hHdd & = \{ \sigma \in \Ht \otimes \bb S: \div \sigma \in \Ht \otimes \bb V,
          \div \div\sigma \in \Ht \}.
\end{align*}
They are normed, respectively, by
$\| \cdot \|_{\Hocc}, \| \cdot \|_{\Hocd}, \| \cdot \|_{\Hodd}$, the
norms defined in~\eqref{eq:Hcc-cd-dd-norms}.  Obviously, the spaces 
defined in~\eqref{eq:H-cc-dd-cd-defn} are subspaces of these spaces,
i.e.,
\begin{equation}
  \label{eq:Hcc-dd-cd-inclusions}
  \Hocc \subseteq \hHcc, \quad 
  \Hocd \subseteq \hHcd, \quad
  \Hodd \subseteq \hHdd.   
\end{equation}
The theorems in this section (Theorems~\ref{thm:reg-dec-Hcc}, \ref{thm:reg-dec-Hdd}, and~\ref{thm:reg-dec-Hcd}) improve these inclusions to equalities,
thus also proving the density of their respective subspaces of
compactly supported smooth functions.

Our results are under  the additional
assumptions on~$\om$ that it is simply connected and
that its boundary is connected. Then the topology of $\om$ is trivial. 
Covering $\om$ by subdomains starlike with respect to a ball and using
regularized Bogovski{\u{i}} operators in each subdomain,
\cite[Theorem~4.9]{CostaMcInt10} proves that there exist continuous
linear operators
\begin{subequations}
 \label{eq:std-rt-inverses} 
\begin{align}
  \label{eq:17}
  {\Tg}: \Hts{s} \otimes \bb V \to \Hts{s+1}, \qquad 
  {\Tc}: \Hts{s} \otimes \bb V  \to \Hts{s+1} \otimes \bb V,\qquad 
  {\Td}: \Hts{s}  \to \Hts{s+1}, 
\end{align}
satisfying
\begin{align}
  \label{eq:CM-grad-T1}
  \grad ({\Tg} v)
  & = v    \qquad \text{ for all } v \in \Hts{s} \otimes \bb V
    \text{ with } \curl v  =0,
  \\
  \label{eq:CM-curl-T2}
  \curl( {\Tc} q)
  & = q   \qquad \text{ for all } q \in \Hts{s} \otimes \bb V
    \text{ with } \div q =0,
  \\
  \label{eq:CM-div-T3}
  \div ({\Td} u )
  & = u  \qquad \text{ for all } u \in \Hts{s} \text{ with }
    u(1) = 0, 
\end{align}
\end{subequations}
for any real number $s$, where $\Hts{s}$ is the subspace of
distributions on $\bb R^3$ defined in~\eqref{eq:Hts}. The last condition
in~\eqref{eq:CM-div-T3}
is a zero mean condition on $u$ given through a functional action that makes sense even for negative $s$. For $s\ge 0$, it can be expressed using the $L_2$ inner product as $(u, 1) = 0$.

It is standard to use \eqref{eq:std-rt-inverses} to produce regular
decompositions of $\Ho(\curl)$ and $\Ho(\div)$. Indeed, any 
$u \in \Ho(\curl)$ can be decomposed into 
\begin{equation}
  \label{eq:Hocurl-reg-dec}
  u = \Soc 0 u + \grad \Soc 1 u, \quad
  \text{ with } 
  \Soc 0 u = \Tc \curl u, \; \Soc 1 u =   \Tg(u - \Tc \curl u), 
\end{equation}
as can be immediately verified using~\eqref{eq:CM-curl-T2}
and~\eqref{eq:CM-grad-T1}. By the continuity properties of $\Tg$ and $\Tc$,
the operators $\Soc 0 : \Ho(\curl) \to \Ho^1 \otimes \bb V$ and $\Soc 1 : \Ho(\curl) \to \Ho^1$  are continuous. Since $\Soc 0 u$ and $\Soc 1 u$ have
$\Ho^1$-regularity (higher than what may be expected of~$u$), this is
referred to as a ``regular decomposition'' of $\Ho(\curl)$.  The process of arriving
at this decomposition can be viewed as first generating a zero curl
function $u - \Tc \curl u$ and then moving {\em left} of $\Ho(\curl)$
in the diagram~\eqref{eq:reverse-arrows} to create a potential in
$\Ho^1$ using the operator $\Tg$. For the matrix-valued function
spaces in the middle of~\eqref{eq:reverse-arrows}, the process is
similar, but we have more options to move, such as {\em up, left, or
  diagonally,} and our regular decompositions that follow have
multiple potentials.

\subsection{Regular decomposition of $\Hocc$}

We start with a result that can also be found in~\cite[Theorem 2]{ArnolHu21} as a special case of existence of regular potential. Here we provide an explicit construction (see also \cite{vcap2023bounded}). 
\begin{lemma}
  \label{lem:Dcc}
  There is a linear map
  $\Dcc: \ker(\div: \Hts s \otimes \bb S) \to \Hts{s+2} \otimes \bb S$ such that 
  \[
    \sigma  =\inc {\Dcc} \sigma, \qquad
    \| {\Dcc} \sigma \|_{\Hts{s+2}} \lesssim \| \sigma \|_{\Hts s}.
  \]
  for any $s \in \bb R$.
\end{lemma}
\begin{proof}
  Let $\sigma \in \Hts s \otimes \bb S$ and $\div \sigma=0$. Applying ${\Tc}$ to row
  vectors of $\sigma$, whose components are all distributions in
  $\bb R^3$ supported on $\bar \om$ per~\eqref{eq:Hts}, we find
  from~\eqref{eq:CM-curl-T2}  that there is a
  $\eta \in \Hts {s+1}\otimes \bb M$ such that the identity 
  \[
    \sigma = \curl \eta
  \]
  holds in $\bb R^3$. Also, since $\sigma$ is symmetric,
  \[
    \skw \sigma = 0 = \skw \curl \eta = \frac 1 2   \mskw \div S \eta
  \]
  by \eqref{eq:skw-curl}.  Hence
  $S \eta \in \Hts{s+1} \otimes \bb M$ has vanishing divergence in
  all $\bb R^3$ (and obviously $S \eta$ is supported on $\bar
  \om$). Applying ${\Tc}$ row-wise to $S \eta$, we conclude
  from~\eqref{eq:CM-curl-T2}  that there is a
  $\gamma \in \Hts{s+2}  \otimes \bb M$ such that
  \[
    S \eta = \curl \gamma.
  \]
  Set $g = \sym \gamma$ in
  $\Hts{s+2} \otimes \bb S$.  Then
  \begin{align*}
    \sigma
    & = \curl \eta = \curl S^{-1} \curl \gamma \\
    & = \curl S^{-1} \curl (\skw \gamma + g) = \inc g 
  \end{align*}
  where the last equality is due to \eqref{eq:S-grad} and \eqref{eq:tr-curl}.
  By the continuity of
  ${\Tc}$, the linear map $\sigma \mapsto g$ we just constructed is
  continuous and is the needed map~${\Dcc}$.
\end{proof}

\begin{lemma}
  \label{lem:Rcct}
  There is a linear map ${\Rggt}: \ker(\inc: \hHcc) \to \Ho^1 \otimes \bb V$ such
  that for any  $g \in \ker(\inc: \hHcc)$,
  \[
    \curl \dfo {\Rggt} g = \curl g, \qquad  \| {\Rggt} g \|_{H^1} \lesssim
    \| \curl g \|_{\Ht}.
  \]
\end{lemma}
\begin{proof}
  Given any $g \in \ker(\inc: \hHcc)$, since
  $\inc g = \curl (\T \curl g) = 0$, applying ${\Tg}$ to each row vector
  of $\T \curl g$ in $\Ht \otimes \bb V$ and
  using~\eqref{eq:CM-grad-T1}, we obtain a
  $q \in L_2(\bb R^3) \otimes \bb V$ satisfying
  \[
    \T \curl g = \grad q
  \]
  on all $\bb R^3$. Moreover, by~\eqref{eq:tr-curl}, $\curl g$ has
  zero trace, so
  \[
    0 = \tr( \T \curl g) = \tr(\grad q) = \div q.
  \]
  Hence $u = \frac 1 2 {\Tc} q$ is in $\Ho^1 \otimes \bb V$ and satisfies
  $q = \frac 1 2 \curl u$. Therefore,  using~\eqref{eq:5}, 
  \begin{align*}
    \curl g
    & = \T \grad q = \frac 1 2 \T \grad \curl u  =  \curl \dfo u. 
  \end{align*}
  Denoting the linear map $g \mapsto u$ by ${\Rggt}$, the proof is 
  now completed using the continuity of ${\Tg}$ and ${\Tc}$.
\end{proof}

\begin{lemma}
  \label{lem:Dgg}
  There is a linear map
  $\Dgg: \ker(\curl : \Hts s \otimes \bb S)\to \Hts{s+2}$ such that 
  \[
    g = \hess \Dgg g, \qquad \| \Dgg g \|_{\Hts{s+2}} \lesssim \| g \|_{\Hts s}.
  \]
  for any $s \in \bb R$ and $g \in \Hs{s}$ with vanishing curl.
\end{lemma}
\begin{proof}
  Let $g \in \Hts s $ have zero curl. Then applying ${\Tg}$ to the row
  vectors of $g$ and using \eqref{eq:CM-grad-T1}, we obtain a
  $u \in \Hts{s+1} \otimes \bb V$, supported on $\bar \om$, such that
  the identity
  \[
    g = \grad u
  \]
  holds on all $\bb R^3$. Applying $\sym$ to both sides,
  $
    g = \dfo u.
  $
  A further application of $\curl$ on both sides
  yields
  \[
    0 = \curl g = \curl \dfo u = \frac 1 2 \T \grad \curl u,
  \]
  by \eqref{eq:5}, i.e., all first order derivatives of $\curl u$ vanish.
  Hence there must
  exist a constant vector $b \in \bb V$ such that $ \curl u = b $
  holds on all $\bb R^3$. But $u$ is supported on $\bar \om$, so $b$
  must be the zero vector. Now that we have shown $\curl u =0$,
  putting $ w = {\Tg} u$ and using \eqref{eq:CM-grad-T1}, we
  find that $w \in \Hts{s+2}$ satisfies $\grad w = u$ and 
  \begin{align*}
    \hess(w) & = \dfo ( \grad w)
     = \dfo(u) = g.
  \end{align*}
  The linear map $g \mapsto w$ we just constructed is the needed operator $\Dgg$.
\end{proof}

\begin{theorem}
  \label{thm:reg-dec-Hcc}
  There exist three continuous linear operators
  \[
    \Socc 0 : \hHcc \to \Ho^1 \otimes \bb S, \qquad 
    \Socc 1 : \hHcc \to \Ho^1 \otimes \bb V, \qquad
    \Socc 2 : \hHcc \to \Ho^1,
  \]
  such that any $g \in \hHcc$ can be decomposed into
  \begin{equation}
    \label{eq:reg-dec-Hcc}
    g = \Socc 0 g + \dfo \Socc 1 g + \hess \Socc 2 g.
  \end{equation}
  Consequently, $\hHcc = \Hocc$.
\end{theorem}
\begin{proof}
  Put $\Socc 0 g := {\Dcc} \inc g$. By Lemma~\ref{lem:Dcc},
  \[
    \inc ( g - \Socc 0 g) =0.
  \]
  Consequently, by Lemma~\ref{lem:Rcct},
  $\Socc 1 g := {\Rggt} ( g - \Socc 0 g)$ is in $\Ho^1 \otimes \bb V$
  satisfies
  \begin{equation}
    \label{eq:19a}
    \curl \big(  g - \Socc 0 g - \dfo \Socc 1 g \big) = 0.
  \end{equation}
  Applying  Lemma~\ref{lem:Dgg} with $s=-1$, we find that 
  $\Socc 2 g : = \Dgg(g - \Socc 0 g - \dfo \Socc 1 g)$ satisfies  
  \[
    g - \Socc 0 g - \dfo \Socc 1 g = \hess \Socc 2 g,
  \]
  and has the required continuity property, thus completing the proof
  of~\ref{eq:reg-dec-Hcc}.

  To prove that $\hHcc = \Hocc$, in view of
  \eqref{eq:Hcc-dd-cd-inclusions}, it suffices to prove that
  any $g \in \hHcc$,
  decomposed as above into $g = \Socc 0 g + \dfo \Socc 1 g + \hess \Socc 2 g$,
  is in $\Hocc$.  By the density of $\cl D (\om)$ in $\Ho^1(\om)$,
  there are
  $g_m \in \cl D(\om) \otimes \bb S$,
  $u_m \in \cl D(\om) \otimes \bb V$,
  and
  $w_m \in \cl D(\om)$ such that
  \[
    \| g_m - \Socc 0 g \|_{H^1} \to 0, \qquad
    \| u_m - \Socc 1 g \|_{H^1} \to 0, \qquad
    \| w_m - \Socc 2 g \|_{H^1} \to 0,
  \]
  as $m \to \infty$. Hence, by \eqref{eq:di-cty},
  $g_m + \dfo u_m + \hess w_m \in \cl D(\om) \otimes \bb S$ converges
  to $g$ in $\| \cdot \|_{\Hocc}$-norm, proving that $g \in \Hocc$.
\end{proof}

\subsection{Regular decomposition of $\Hodd$}

\begin{lemma}
  \label{lem:Ddd}
  There is a linear map
  ${\Ddd}: \Hts{s}_{\pol_1} \to \Hts{s+2} \otimes \bb S $ such
  that for any $s \in \bb R$ and  $w \in \Hts{s}_{\pol_1}$,
  \begin{equation}
    \label{eq:Dm}
    \div \div {\Ddd} w = w, \qquad
    \| {\Ddd} w \|_{\Hts{s+2}} \lesssim \| w \|_{\Hts{s}}.
  \end{equation}
\end{lemma}
\begin{proof}
  Let $w \in \cl D_{\pol_1}$. Since $(w, 1)=0$,
  by~\eqref{eq:CM-div-T3}, $q = {\Td}w$ satisfies
  \[
    \div q = w, \qquad \| q \|_{\Hts{s+1}} \lesssim \| w \|_{\Hts s}.
  \]
  Since $q$ is supported on $\bar \om$, 
  we may integrate by parts to see that
  $0 = (w, p_1) = (\div q, p_1) = -(q, \grad p_1)$ for any $p_1 \in \pol_1$. Thus all
  components of $q$ have zero mean on
  $\om$. Applying~\eqref{eq:CM-div-T3} again, we then obtain a
  $\tau \in \Hts{s+2}\otimes \bb M$ such that $\div\tau = q$. Let
  $u = \vskw \tau$. Then $\div \skw\tau = \div \mskw u = -\curl u$
  by~\eqref{eq:div-mskw}.  Collecting these observations, and putting
  $\sigma = \sym \tau$, 
  \begin{align*}
    w & = \div q = \div \div \tau = \div\div ( \sym \tau + \skw \tau)
    \\
      & = \div \div \sigma  - \div \curl u  = \div \div \sigma.
  \end{align*}
  Denote the linear map $w \mapsto \sigma$ we just constructed by
  ${\Ddd} w$. By the continuity of ${\Td}$, we see that ${\Ddd}$ satisfies the
  norm estimate in~\eqref{eq:Dm} for all $w \in \cl D_{\pol_1}$. Hence
  by the density result of Lemma~\ref{lem:dense-mean0}, ${\Ddd}$ has a
  unique continuous extension, which is the required map.
\end{proof}

\begin{lemma}
  \label{lem:Rddt}
  There is a linear map
  ${\Rcct} : \ker(\div\div: \hHdd) \to \Ho^1\otimes \bb T$
  such  that for all
  $\sigma \in \ker(\div\div: \hHdd),$
  \[
    \div \sym \curl {\Rcct} \sigma = \div \sigma, \qquad
    \| {\Rcct} \sigma \|_{H^1} \lesssim \| \div \sigma \|_{\Ht}.
  \]
\end{lemma}
\begin{proof}
  Consider a  $\sigma \in \hHdd$ with $\div\div \sigma=0$. Then, since
  $\div \sigma \in \Ht \otimes \bb V$ has vanishing divergence,
  $u = {\Tc} \div\sigma$ is in $L_2\otimes \bb V$
  and satisfies $\curl u = \div \sigma$ by
  \eqref{eq:CM-curl-T2}.
  Next, we claim that $(u, b) =0$ for any $b \in \bb V$. To see this,
  first note that the distribution $\div \sigma$ satisfies 
  \[
    (\div \sigma)(b \times x) = \sigma (\grad (b \times x)) =
    \sigma (\mskw b) = 0 
  \]
  due to the symmetry of $\sigma$.
  Relating  to $u$, 
  \begin{align*}
    0
    & = (\div \sigma)( b \times x)
      = (\curl u ) ( b \times x)
      = (u, \curl (b \times x) ) = 2 (u, b).
  \end{align*}
  Hence we may apply ${\Td}$ to each component of $\frac 1 2 u$ and
  use~\eqref{eq:CM-div-T3} to get a $\tau \in \Ho^1 \otimes \bb M$
  such that $\div \tau = \frac 1 2 u$, which implies 
  \begin{align*}
    \div \sigma
    & = \curl u = \frac 1 2  \curl \div \tau = \frac 1 2  \curl \div \dev \tau
      =  \div \sym \curl \T \dev \tau.
  \end{align*}
  Here we have used~\eqref{eq:4} and the fact that $\curl \div$
  vanishes on matrix fields that are scalar multiples of the identity.
  Denoting the map $\sigma \mapsto \T \dev \tau$ by ${\Rcct}$, the
  continuity of ${\Tc}$ and ${\Td}$ finishes the proof.
\end{proof}

\begin{theorem}
  \label{thm:reg-dec-Hdd}
  There exist three continuous linear operators
  \[
    \Sodd 0 : \hHdd \to \Ho^1 \otimes \bb S, \qquad 
    \Sodd 1 : \hHdd \to \Ho^1 \otimes \bb T, \qquad 
    \Sodd 2 : \hHdd \to \Ho^1 \otimes \bb S,
  \]
  such that any $\sigma \in \hHdd$
  can be decomposed into
  \begin{equation}
    \label{eq:reg-dec-Hdd}
    \sigma = \Sodd 0 \sigma + \sym\curl \Sodd 1 \sigma + \inc \Sodd 2 \sigma.
  \end{equation}
  Consequently, $\hHdd = \Hodd$.
\end{theorem}
\begin{proof}
  Let  $\sigma \in \hHdd$ and $\Sodd 0 \sigma := {\Ddd} \div \div \sigma.$
  Note that $\div\div \sigma$ is in $\HtPl$, the domain of
  ${\Ddd}$, because
  the hessian of $p$ is zero for any $p \in \pol_1$ and 
  \[
    (\div \div \sigma) (p) = \sigma(\hess p) = 0.
  \]
  By 
  Lemma~\ref{lem:Ddd},
  \[
    \div\div ( \sigma - \Sodd 0 \sigma) = 0.
  \]
  Next, set $\Sodd 1 \sigma := {\Rcct} (\sigma - \Sodd 0 \sigma)$ in
  $\Ho^1\otimes \bb T$.  By Lemma~\ref{lem:Rddt},
  \[
    \div \big( \sigma - \Sodd 0 \sigma -  \sym\curl \Sodd 1 \sigma \big) =0.
  \]
  By Lemma~\ref{lem:Dcc}, setting
  $\Sodd 2 \sigma := {\Dcc} ( \sigma - \Sodd 0 \sigma - \sym\curl \Sodd 1
  \sigma )$, we find that
  \[
    \sigma - \Sodd 0 \sigma - \sym\curl \Sodd 1  \sigma = \inc \Sodd 2 \sigma,
  \]
  thus completing the proof of \eqref{eq:reg-dec-Hdd}.

  To conclude, it suffices to prove that $\hHdd \subseteq \Hodd$, due
  to \eqref{eq:Hcc-dd-cd-inclusions}.  Decompose any
  $\sigma \in \hHdd$ into
  $\sigma = \Sodd 0 \sigma + \sym\curl \Sodd 1 \sigma + \inc \Sodd 2
  \sigma$.
  By the density of $\cl D (\om)$ in $\Ho^1$,
  there are
  $\sigma_m \in \cl D(\om) \otimes \bb S$,
  $\tau_m \in \cl D(\om) \otimes \bb T$,
  and
  $g_m \in \cl D(\om) \otimes \bb S$ such that
  \[
    \| \sigma_m - \Sodd 0 \sigma \|_{H^1} \to 0, \qquad
    \| \tau_m - \Sodd 1 \sigma \|_{H^1} \to 0, \qquad
    \| g_m - \Sodd 2 \sigma  \|_{H^1} \to 0,
  \]
  as $m \to \infty$. Hence, by \eqref{eq:di-cty},
  $\sigma_m + \sym\curl \tau_m + \inc g_m \in \cl D(\om) \otimes \bb S$ converges
  to $\sigma$ in $\| \cdot \|_{\Hodd}$~norm, thus proving that $\sigma \in \Hodd$.
\end{proof}

\subsection{Regular decomposition of $\Hocd$}

Next, we turn to constructing a regular decomposition of $\Hocd$. (The
case of $\HocdT$ obviously follows from that of $\Hocd$.) Unlike the
three-term decompositions of $\Hocc$ and $\Hodd$ cases, now we are only
able to construct a decomposition with four terms.  We begin with a
preparatory lemma.

\begin{lemma}
  \label{lem:Rcdh+Dcdh+Ucdh}
  Let $K = \ker( \curl\div: \hHcd)$. 
  There are  linear maps
  ${\Rgch}: K \to \Ho^1 \otimes \bb S 
  $,
  $\Dgch: K \to \Ho^1 \otimes \bb V 
  $,
  and
  $\Ugch: K \to \Ho^1 \otimes \bb V 
  $,
  such that any $\tau \in \hHcd$ with $\curl\div \tau =0$ can be  
  decomposed into
  \begin{equation}
    \label{eq:RDU-cd-decomp}
    \tau = \curl {\Rgch} \tau + \curl\dfo \Ugch \tau + \T \dev \grad \Dgch \tau
  \end{equation}
  and the following continuity bound holds:
  \begin{equation}
    \label{eq:RDU-cd-decomp-est}
    \| {\Rgch} \tau \|_{H^1} + \| \Dgch \tau \|_{H^1} + \| \Ugch \tau
    \|_{H^1} \lesssim \| \tau \|_{\Hocd}.
  \end{equation}
  If in addition, $\tau$ is in $L_2 \otimes \bb T$, then $\Ugch$ can
  be taken to be zero provided~\eqref{eq:RDU-cd-decomp-est} is replaced by
  \[
    \| {\Rgch} \tau \|_{H^1} \lesssim \| \tau \|_{\Hocd},
    \qquad
    \| \Dgch \tau \|_{H^1} \lesssim  \|\tau\|_{L_2} +  \| \tau \|_{\Hocd}.
  \]  
\end{lemma}
\begin{proof}
  Given any  $\tau \in \hHcd$ with $\curl\div \tau =0$, put $w = {\Tg} \div\tau$. Then
  by~\eqref{eq:CM-grad-T1},
  \begin{equation}
    \label{eq:w-estimate}
    \grad w = \div \tau, \qquad \| w \|_{L_2} \lesssim \| \div\tau \|_{\Ht}.
  \end{equation}
  Since $\tr (\T \tau)=0$, we  know that $S \T \tau = \tau$, so
  \begin{equation}
    \label{eq:16}
    \skw\curl \T\tau =
    \frac 1 2 \mskw \div S \T \tau =
    \frac 1 2 \mskw \div \tau =
    \frac 1 2 \mskw \grad w =
    -\frac 1 2 \curl (w \id),
  \end{equation}
  where we have used~\eqref{eq:skw-curl} and~\eqref{eq:div-mskw}.

  Let $\sigma = \sym\curl \T \tau$.  By the identity~\eqref{eq:4} of
  Lemma~\ref{lem:commute-identities},
  $\div \sigma = \curl \div \tau =0$, so applying
  Lemma~\ref{lem:Rcct}, $g = {\Rggt} \sigma$ is in $\Ho^1 \otimes \bb S$
  and satisfies
  \begin{equation}
    \label{eq:g-est}
    \sigma = \inc g, \qquad \| g \|_{H^1}
    \lesssim \|\sym \curl \T \tau \|_{\Ht}        
  \end{equation}
  Combined with~\eqref{eq:16}, we have the twin identities 
  \begin{align*}
    \sym \curl \T \tau &  = \curl \T \curl g, 
    \\
    \skw \curl \T \tau & = -\frac 1 2 \curl( w \id).
  \end{align*}
  Adding these equations, we find that
  $\curl ( \T \tau - \T \curl g + \frac 1 2 w \id) =0$. Hence,
  applying ${\Tg}$ to each of the row vectors of
  $ \T \tau - \T \curl g + \frac 1 2 w \id$ and
  using~\eqref{eq:CM-grad-T1}, we obtain a $q \in L_2 \otimes \bb V$ such that 
  \begin{subequations}
    \label{eq:20-all}
    \begin{gather}
      \label{eq:20}
      \grad q = \T \tau - \T \curl g + \frac 1 2 w \id,
      \\
      \label{eq:20-est}
      \| q \|_{L_2} \lesssim \| \tau - \curl g + w \|_{\Ht}.
    \end{gather}
  \end{subequations}
  In fact, $q|_\om$ is in $\Ho(\div)$. To see this, take traces on
  both sides of~\eqref{eq:20}. Recall that $\tr \tau=0$. Also,
  $\tr(\curl g) =0$ by~\eqref{eq:tr-curl}. Hence we conclude that
  $\frac 3 2 w =\div q$, an identity that holds in all $\bb R^3$ with
  $q$ and $w$ supported only on $\bar \om$. Since $w \in L_2$, this in
  particular shows that $q|_\om \in \Ho(\div)$, and the estimate
  \begin{align}
    \label{eq:20-est-1}
    \| q \|_{H(\div)}
    & \lesssim \| \tau \|_{\Ht}
      + \| g \|_{H^1} + \| w \|_{L_2} \lesssim \| \tau \|_{\Hocd}
  \end{align}
  follows from the estimates of~\eqref{eq:20-est},~\eqref{eq:g-est},
  and \eqref{eq:w-estimate}.

  Taking the deviatoric part of both sides of~\eqref{eq:20} and noting
  that $\tau = \dev \tau$, we obtain a preliminary two-term
  decomposition of $\tau,$
   \begin{equation}
     \label{eq:21}
     \tau = \curl g + \T \dev \grad q.
   \end{equation}
   However, here $q$ is not in $\Ho^1 \otimes \bb V$, in general.   
   To improve this to the needed result, we use $r = {\Td} \div q$,
   which has the same divergence as $q$, but is in 
   $\Ho^1 \otimes \bb V$:
   \[
     \div r = \div q, \qquad \| r \|_{H^1} \lesssim \| \div q \|_{L_2}
     \lesssim \| \tau \|_{\Hocd},
   \]
   by~\eqref{eq:20-est-1}.  Since $\div(q - r)=0$, putting
   $u = \frac 1 2 {\Tc}(q -r)$ in $\Ho^1\otimes \bb V$,
   by~\eqref{eq:CM-curl-T2},
   \[
     \frac 1 2 \curl u = q - r,
     \qquad \| u \|_{H^1} \lesssim \| q - r \|_{L_2} \lesssim \|\tau\|_{\Hocd}.
   \]
   Hence
   \begin{align*}
     \dev \grad q
     = \dev \grad r + \frac 1 2 \dev \grad \curl u
     = \dev \grad r +  \curl \dfo u.
   \end{align*}
   Substituting this into~\eqref{eq:21} and setting
   ${\Rgch} \tau = g, \Dgch \tau = r,$ and $\Ugch \tau = u$, we see
   that~\eqref{eq:RDU-cd-decomp} and~\eqref{eq:RDU-cd-decomp-est}
   follow.

   To prove the remaining statement, suppose
   $\tau \in L_2 \otimes \bb T \cap \hHcd$. Then, due to the higher
   regularity of $\tau$, observe that $q$ in \eqref{eq:20} is in
   $\Ho^1 \otimes \bb V$ and in place of 
   \eqref{eq:20-est}, we have
   \[
     \| q \|_{H^1} \lesssim \| \tau - \curl g + w \|_{L_2}
     \lesssim \| \tau \|_{L_2} + \| g \|_{H^1} + \| w \|_{L_2}.
   \]
   which can be used in place of~\eqref{eq:20-est-1}.
   There is no
   longer a need to produce the $r$ above, and we may set
   $\Dgch \tau = q \in \Ho^1 \otimes \bb V$.  The
   decomposition~\eqref{eq:21} then concludes the proof.
\end{proof}

\begin{lemma}
  \label{lem:Dcd}
  There is a linear map
  ${\Dcd}: \ker( \div: \Hs{s}_{\ND}(\div) ) \to \Hs{s+2}  \otimes \bb T$
  such that 
  \[
    \curl\div {\Dcd} v = v, \qquad \| {\Dcd} v \|_{\Hs{s+2}} \lesssim \| v \|_{\Hs{s}},
  \]
  for any $s \in \bb R$ and $v$ in $\Hs{s}_{\ND}(\div) $ with zero divergence.
\end{lemma}
\begin{proof}
  Since $\div v =0$, by~\eqref{eq:CM-curl-T2}, $u = {\Tc} v$ is in
  $\Hs{s} \otimes \bb V$ and satisfies $\curl u =v$ in all $\bb R^3$. For any
  constant vector $b \in \bb V$, the action of the distribution $u$ on $b$ satisfies  
  \begin{align*}
    2 \,u(b) = u(\curl (b \times x)) = (\curl u)( b \times x) = v(b \times x) = 0
  \end{align*}
  since $v(r)=0$ for any $r \in \ND$. Hence, applying ${\Td}$ to each
  component of $u$ and using~\eqref{eq:CM-div-T3}, we obtain a
  $\tau \in \Hs{s+2} \otimes \bb M$ such that $\div \tau = u$, i.e.,
  \begin{align*}
    v = \curl \div \tau = \curl \div \dev \tau,
  \end{align*}
  since $\curl \div (\frac 1 3 (\tr\tau ) \id) = 0.$ Denoting the map
  $v \mapsto \dev \tau$ by ${\Dcd}$, the proof is finished by the
  continuity of ${\Tc}$ and ${\Td}$.
\end{proof}

\begin{theorem}
  \label{thm:reg-dec-Hcd}
  There exist four continuous linear operators
  \[
    \Socd 0 : \hHcd \to \Ho^1 \otimes \bb T,\quad
    \Socd 1 : \hHcd \to \Ho^1 \otimes \bb S,\quad
    \Socd 2 : \hHcd \to \Ho^1 \otimes \bb V, \quad
    \Socd 3 : \hHcd \to \Ho^1 \otimes \bb V,\quad    
  \]
  such that any $\tau \in \hHcd$ can be decomposed into
  \begin{equation}
    \label{eq:tau-decomp-Hcd}
        \tau = \Socd 0 \tau + \curl \Socd 1 \tau +  \T \dev \grad \Socd 2 \tau + 
        \curl\dfo \Socd 3 \tau.
  \end{equation}
  It then follows that  $\hHcd = \Hocd$.
\end{theorem}
\begin{proof}
  Let $\tau \in \hHcd$ and put $q = \curl\div\tau$. Obviously
  $\div q =0$ and $q \in \Ht(\div)$. Moreover, for any
  $a, b \in \bb V$ and $r = a + b \times x \in \ND$, since
  $\curl r = 2 b$ is constant, its gradient vanishes, and
  \begin{align*}
    q(r) = (\curl \div \tau )(r) = \tau( \grad \curl r) = 0.
  \end{align*}
  Thus $q$ is in $\HtND(\div)$ and we apply $\Dcd$ to it. Put
  $\Socd 0 \tau := {\Dcd} \curl\div\tau$. By
  Lemma~\ref{lem:Dcd} with $s=-1$, we find that
  $\Socd 0 \tau \in \Ho^1 \otimes \bb T$ and 
  \[
    \curl\div(\tau - \Socd 0 \tau) =0.
  \]
  Hence we may apply Lemma~\ref{lem:Rcdh+Dcdh+Ucdh} to get 
  \[
    \tau - \Socd 0 \tau =
    (\curl {\Rgch}  + \curl\dfo \Ugch  + \T \dev \grad \Dgch )(\tau - \Socd 0 \tau).
  \]
  The decomposition~\eqref{eq:tau-decomp-Hcd}   now follows after setting 
  $\Socd 1 \tau = {\Rgch} (\tau -  \Socd 0 \tau)$,
  $\Socd 2 \tau = \Dgch (\tau -  \Socd 0 \tau)$ and 
  $\Socd 3 \tau = \Ugch (\tau - \Socd 0 \tau)$.

  To prove that $\hHcd = \Hocd$, let $\tau \in \hHcd$ be decomposed as in~\eqref{eq:tau-decomp-Hcd}.
  By the density of $\cl D (\om)$ in $\Ho^1$,
  there are
  $\tau_m \in \cl D(\om) \otimes \bb T$,
  $g_m \in \cl D(\om) \otimes \bb S$,
  $q_m \in \cl D(\om) \otimes \bb V$, and 
  $u_m \in \cl D(\om) \otimes \bb V$,
  such that
  \[
    \| \tau_m - \Socd 0 \tau \|_{H^1} \to 0, \quad 
    \| g_m - \Socd 1 \tau \|_{H^1}\to 0, \quad 
    \| q_m - \Socd 2 \tau \|_{H^1} \to 0, \quad 
    \| u_m - \Sodd 2 \tau  \|_{H^1}\to 0,
  \]
  as $m\to\infty.$ 
  By  \eqref{eq:5}, 
  \begin{align*}
    \| \tau_m
    & + \curl g_m + \T \dev \grad q_m + \curl\dfo u_m  \; - \;
    \tau \|_{\Hocd}
    \\
    & = \| \tau_m - \tau \|_{\Hocd} +
    \|\curl (g_m - \Socd 1 \tau) \|_{L_2}
      +\|  \T \dev \grad (q_m - \Socd 2 \tau) \|_{L_2}
    \\
    & +\| \curl\dfo (u_m  - \Socd 3 \tau) \|_{L_2} \to 0,
  \end{align*}
  which converges to zero as $m \to \infty$ in view of
  \eqref{eq:di-cty}.  Thus $\hHcd \subseteq \Hocd$ and the proof is
  complete due to \eqref{eq:Hcc-dd-cd-inclusions}.
\end{proof}

In view of these results, we shall no longer distinguish between
$\hHcc$ and $\Hocc$, $\hHdd$ and $\Hodd$, nor $\hHcd$ and $\Hocd$.

\subsection{Continuous right inverses}

Let us now complete the discussion of~\eqref{eq:reverse-arrows}.
Several right inverse operators in~\eqref{eq:reverse-arrows} were
already given in previous lemmas.  The right inverses in the top row
of~\eqref{eq:reverse-arrows} are the same operators as
in~\eqref{eq:CM-grad-T1}--\eqref{eq:CM-div-T3}. For example, ${\Tc}$
is a right inverse of $\curl: \Ho(\curl) \to \Ho(\div)$ in the sense
that ${\Tc} : \ker (\div: \Ho(\div)) \to \Ho(\curl)$ is continuous and
$\curl \circ {\Tc}$ equals the identity on $\ker (\div: \Ho(\div))$,
which is just a restatement of~\eqref{eq:CM-curl-T2} with $s=0$.
After construction of the remaining needed right inverses,
Theorem~\ref{thm:rt-inv-cts} below gathers everything together.

\begin{lemma}
  \label{lem:Rcdt+Dcdt}
  There are linear maps
  ${\Rgct} : \ker(\div: \Hocd) \to L_2 \otimes \bb S \subset \Hocc$ and
  $\Dgct: \ker(\div: \Hocd) \to \Ho^1 \otimes \bb V \subset \Ho(\curl)$ such that for all
  $\tau \in \ker(\div: \Hocd)$,
  \[
    \tau = \curl ( {\Rgct} \tau + \dfo \Dgct \tau), \qquad
    \| {\Rgct} \tau \|_{L_2} + \| \Dgct \tau \|_{H^1} \lesssim \| \tau \|_{\Ht}.
  \]
\end{lemma}
\begin{proof}
  Applying ${\Tc}$ to the divergence-free row vectors of $\tau$, we find a
  $\gamma \in L_2 \otimes \bb M$ satisfying $\curl \gamma = \tau$ per
  \eqref{eq:CM-curl-T2}. Put $g = \sym\gamma$ and $v = \vskw \gamma$. Then,
  by~\eqref{eq:S-grad}, 
  \begin{align*}
    \tau
    & = \curl ( \sym \gamma + \skw \gamma) = \curl g - \curl \mskw v
    \\
    & = \curl g - S \grad v 
    \\
    & = \curl g - \T \grad v + (\div v) \id. 
  \end{align*}
  Since $\Hocd$ consists of trace-free matrix fields and since trace of
  $\curl g$ vanishes by~\eqref{eq:tr-curl}, taking the trace of the
  above expression, we find that
  \begin{align*}
    0 = \tr \tau = 2 \div v.
  \end{align*}
  Therefore, by~\eqref{eq:CM-curl-T2},
  $u = -\frac 1 2 {\Tc} v \in \Ho^1 \otimes \bb V,$ satisfies
  $v = -\frac 1 2 \curl u$, so
  \begin{align*}
    \tau
    & = \curl g + \frac 1 2 \T \grad \curl u
    \\
    & = \curl (g + \dfo u)
  \end{align*}
  by~\eqref{eq:5} of Lemma~\ref{lem:commute-identities}. The linear
  maps $\tau \mapsto g$ and $\tau \mapsto u$ are the needed ${\Rgct}$ and
  $\Dgct$.
\end{proof}

\begin{lemma}
  \label{lem:Rgc}
  There is a linear map
  ${\Rgc} : \ker(\div: \Hocd) \to L_2 \otimes \bb S \subset \Hocc$ such that for all
  $\tau \in \ker(\div: \Hocd)$,
  \[
    \tau = \curl  {\Rgc} \tau, \qquad
    \| {\Rgc} \tau \|_{L_2}  \lesssim \| \tau \|_{\Ht}.
  \]
\end{lemma}
\begin{proof}
  Using the operators of Lemma~\ref{lem:Rcdt+Dcdt}, define
  ${\Rgc} \tau = {\Rgct} \tau + \dfo \Dgct \tau$. Then the result follows
  immediately from Lemmas~\ref{lem:Rcdt+Dcdt} and~\eqref{eq:di-cty}.
\end{proof}

\begin{lemma}
  \label{lem:Dgc}
  There is a linear map
  $\Dgc: \ker( \div: \Hocd) \cap \ker(\sym\curl \T: \Hocd) \to \Ho^1
  \otimes \bb V \subset \Ho(\curl)$ such that for all $\tau \in \Hocd$
  with $\div \tau=0$ and $\sym \curl \T \tau =0$,
  \[
    \tau  = \curl\dfo \Dgc \tau, \qquad
    \| \Dgc \tau \|_{H^1} \lesssim \| \tau \|_{\Ht}
  \]    
\end{lemma}
\begin{proof}
  Given any $\tau \in \Hocd$ with $\div \tau=0$, by
  Lemma~\ref{lem:Rcdt+Dcdt},
  $\tau = \curl ({\Rgct} \tau + \dfo \Dgct\tau)$.  When
  $\sym\curl \T \tau$ also vanishes, this implies that
  \[
    0 = \sym\curl \T \tau = 
    \sym\curl \T \curl ({\Rgct} \tau + \dfo \Dgct\tau) 
    = \inc ({\Rgct} \tau).
  \]
  Applying  Lemma~\ref{lem:Rcct} with $g = {\Rgct} \tau$, 
  $
    \curl \dfo {\Rggt} g = \curl g,
  $
  which in turn implies that
  \[
    \tau =  \curl ({\Rgct} \tau + \dfo \Dgct\tau)
    = \curl \dfo( {\Rggt} {\Rgct} \tau + \Dgct \tau).
  \]
  Hence the result follows by setting $\Dgc = {\Rggt} {\Rgct}  + \Dgct $.
\end{proof}

\begin{lemma}
  \label{lem:Dgd}
  There is a linear map
  ${\Dgd}: \ker( \curl: \HtRT(\curl)) \to \Ho^1 \otimes \bb V$ such
  that for all $v \in \ker( \curl: \HtRT(\curl))$
  \begin{equation}
    \label{eq:Dc}
    v = \frac 1 3 \grad \div {\Dgd} v, \qquad
    \| {\Dgd} v \|_{H^1} \lesssim \| v \|_{\Ht}.
  \end{equation}
\end{lemma}
\begin{proof}
  Let $v \in \HtRT(\curl)$ have zero curl. Then $w = {\Tg} v$ is in
  $L_2(\bb R^3) \otimes \bb V$, supported on $\bar \om$, and satisfies
  $v = \grad w$ in all $\bb R^3$ due to~\eqref{eq:CM-grad-T1}. Since
  $v(r)=0$ for all $r \in \RT$, choosing $r = b x$ for any constant $b$, 
  \[
   0 = (\grad w)(r) = (w,  \div r )= 3 (w, b).
 \]
 Hence we may apply ${\Td}$ to each component of $w$ and
 use~\eqref{eq:CM-div-T3} to get a $q \in \Ho^1 \otimes \bb V$
 satisfying $w = \frac 1 3 \grad \div q$. The linear map $v \mapsto w$
 is the required ${\Dgd}$.
\end{proof}

\begin{lemma}
  \label{lem:Rgg}
  There is a linear map ${\Rgg}: \ker(\inc: \Hocc) \to \Ho(\curl)$ such
  that for any  $g \in \ker(\inc: \Hocc)$,
  \[
    \dfo {\Rgg} g = g, \qquad  \| {\Rgg} g \|_{H(\curl)} \lesssim \| g \|_{\Hocc}.
  \]
\end{lemma}
\begin{proof}
  By Lemma~\ref{lem:Rcct}, $u = {\Rggt} g$ satisfies $ \curl (g - \dfo u) = 0$.
  Hence applying ${\Tg}$ to each row
  vector of $g - \dfo u$, we obtain a $v \in L_2 \otimes \bb V$ such
  that $ g - \dfo u = \grad v$. The symmetry of the left hand side implies that 
  \[
    0 = \skw ( g -  \dfo u) = \skw \grad v = \frac 1 2 \mskw \curl v
  \]
  by~\eqref{eq:mskw-curl}.  Hence $\curl v = 0 $ on all $\bb R^3$, so
  the vector field $v|_\om$ is in $\Ho(\curl)$. We have thus shown
  that $g = \dfo( u + v)$.  Letting the map $g \mapsto u + v$ be
  denoted by ${\Rgg}$, the norm bound follows from the continuity of
  ${\Tg}$ and ${\Rggt}$.
\end{proof}

\begin{lemma}
  \label{lem:Rgd}
  There is a linear map
  ${\Rgd}: \HtRT(\curl) \to L_2 \otimes \bb T \subset \Hocd$
  such that for
  all $v \in \HtRT(\curl)$,
  \begin{equation}
    \label{eq:Rgd}
    v = \div {\Rgd} v, \qquad
    \| {\Rgd} v \|_{L_2}    \lesssim \| v \|_{\Ht}.
  \end{equation}
\end{lemma}
\begin{proof}
  Let $v \in \cl D_{\RT}$. Since every component $v_i$ of $v$ has
  zero mean on $\om$, we may apply ${\Td}$ to each and use~\eqref{eq:CM-div-T3}
  to get a $\tau \in L_2(\bb R^3) \otimes \bb M$, supported on
  $\bar \om$, satisfying $v = \div \tau$ on all $\bb R^3$, which
  in particular, implies
  that $\tau$ is in $\Ho(\div)$. Hence we may integrate by parts to conclude that
  \begin{align*}
    (\tau, \grad r) = (\div \tau , r) = 
    (v, r) =0
  \end{align*}
  for all $r \in \RT$. Choosing $r = b x$ for any
  $b \in \bb R$ and noting that $\grad(b x) = b \id$, we find that
  $(\tau, \id) =0$. Hence $t = \tr(\tau) \in L_2$ satisfies
  \begin{align*}
    0 &  = (\tau, \id) 
        = (\dev \tau, \id)  + \frac 1 3  (t \id, \id) = (t, 1).
  \end{align*}
  Now, by~\eqref{eq:CM-div-T3}, $q = {\Td} t$ satisfies $\div q = t$, so
  \begin{align*}
    v & = \div\tau = \div  \dev \tau + \frac 1 3 \div (\tr\tau )\, \id
        = \div  \dev \tau + \frac 1 3 \grad(t)
    \\    
      & = \div (\dev \tau) + \frac 1 3 \grad (\div q)
    \\
      & = \div (\dev \tau) + \frac 1 2 \div(\T \dev \grad  q)
  \end{align*}
  by~\eqref{eq:6} of Lemma~\ref{lem:commute-identities}. Denoting the
  linear maps $v \mapsto \dev \tau + \frac 1 2 \T \dev \grad q$ by ${\Rgd}$,
  the norm estimate in~\eqref{eq:Rgd}
  follows for $v \in \cl D_{\RT}$. The proof is now finished by
  the density result of Lemma~\ref{lem:dense-mean0}.
\end{proof}

\begin{lemma}
  \label{lem:RgcT}
  There is a linear map ${\RgcT} : \ker(\sym\curl: \HocdT) \to \Ho(\div)$
  such that for all $\tau \in \ker(\sym\curl: \HocdT)$,
  \[
    \frac 1 2 \dev \grad {\RgcT} \tau = \tau, \qquad
    \| {\RgcT} \tau \|_{H(\div)} \lesssim \| \tau \|_{\Hocd}.
  \]
\end{lemma}
\begin{proof}
  Let $\tau \in \ker(\sym\curl: \HocdT)$. Then
  $\curl \div \T \tau = \div \sym\curl \tau =0$, so
  by~\eqref{eq:CM-grad-T1}, $w = {\Tg} \div (\T \tau)$ in
  $L_2(\bb R^3)$, supported on $\bar\om$, satisfies
  $\grad w = \div \T \tau$.

  Next, recalling that $\tr \tau = 0$, note that
  \begin{align*}
    2 \skw\curl \tau
    & = \mskw \div S \tau
    && \text{ by~\eqref{eq:skw-curl}}
    \\
    & = \mskw \div \T \tau       = \mskw \grad w
    \\
    & = -\curl (w \id),
    && \text{ by~\eqref{eq:mskw-grad}}.
  \end{align*}
  Hence
  $\curl \tau = \sym\curl \tau + \skw \curl \tau = -\frac 1 2 \curl (w
  \id).$ Applying ${\Tg}$ to each row vector of
  $\tau + \frac 1 2 w \id$, \eqref{eq:CM-grad-T1} we obtain a
  $q \in L_2 \otimes \bb V$ satisfying
  \begin{equation}
    \label{eq:9}
    \tau + \frac 1 2 w \id = \frac 1 2 \grad q.
  \end{equation}
  In particular, applying the $\tr$-operator to both sides
  of~\eqref{eq:9}, we see that the identity $3 w = \div q$ holds on
  all $\bb R^3$, so $q \in \Ho(\div)$.  Furthermore, applying
  $\dev$-operator to both sides of~\eqref{eq:9}, we conclude that
  $\tau = \dev \tau = \dev \grad q$. The map $\tau \mapsto q$ is the
  required operator ${\Rgc}$ and its stated norm bound follows from the
  continuity of~${\Tg}$.
\end{proof}

\begin{lemma}
  \label{lem:Rcc}
  There is a linear map
  ${\Rcc} : \ker(\div\div: \Hodd) \to L_2\otimes \bb T $
  such  that for all
  $\sigma \in \ker(\div\div: \Hodd)$,
  \[
    \sym \curl {\Rcc} \sigma = \sigma, \qquad
    \| {\Rcc} \sigma \|_{L_2} \lesssim \| \sigma \|_{\Ht}.
  \]
\end{lemma}
\begin{proof}
  By Lemma~\ref{lem:Rddt}, $\sigma - \sym\curl {\Rcct} \sigma$ has
  vanishing divergence. Applying ${\Tc}$ to its rows and
  using~\eqref{eq:CM-curl-T2}, we obtain a
  $\rho \in L_2 \otimes \bb M$ such that
  $\sigma - \sym\curl {\Rcct} \sigma = \curl \rho$. Hence
  \begin{align*}
    \sigma
    & = \sym\curl {\Rcct} \sigma +  \curl (\dev \rho + \frac 1 3 \tr(\rho) \id)
    \\
    & = \sym\curl {\Rcct} \sigma +  \curl (\dev \rho) -\frac 1 3
      \mskw\grad  \tr(\rho)
  \end{align*}
  by~\eqref{eq:mskw-grad}. Applying $\sym$-operator to both sides,
  $\sigma = \sym\curl ({\Rcct} \sigma + \dev \rho)$. The linear map
  $\sigma \mapsto {\Rcct} \sigma + \dev \rho$ is the required map
  ${\Rcc}$.
\end{proof}

\begin{lemma}
  \label{lem:Rcd}
    There is a linear map
    ${\Rcd}: \HtND(\div) \to L_2 \otimes \bb S$
    such that for all
    $q$ in $\HtND(\div)$,
    \begin{equation}
      \label{eq:Rd}
      q = \div {\Rcd} q, \qquad 
      \| {\Rcd} q \|_{L_2} \lesssim \| q \|_{\HtND(\div)}.
  \end{equation}
\end{lemma}
\begin{proof}
  Let $q \in \cl D_{\ND}$. Since every component of $q$ has zero
  mean, by~\eqref{eq:CM-div-T3}, there exists a
  $\gamma \in L_2(\bb R^3) \otimes \bb M$, supported on $\bar \om$,
  such that $\div \gamma = q$. In particular, this implies that each
  row vector of $\gamma|_\om$ is in $\Ho(\div)$. Hence, integration by
  parts shows that for any $b \in \bb V$,
  \begin{align*}
    (q, b\times x) = (\div \gamma, b \times x) = -(\gamma, \grad (b \times x) ) =
    -(\gamma, \mskw b).   
  \end{align*}
  Since $(q, b \times x )=0$ for all $q \in \cl D_{\ND}$, all terms
  above vanish, so $0= (\gamma, \mskw b) = (\mskw u, \mskw b)$ with
  $u = \vskw \gamma,$ or equivalently, $(u, b)=0$. Therefore,
  applying~\eqref{eq:CM-div-T3} to each component of
  $u \in L_2 \otimes \bb V$, we obtain a
  $\tau \in \Ho^1 \otimes \bb M$ such that $\div \tau = - 2 u$.

  Collecting these observations, and putting $\sigma = \sym \gamma$,
  \begin{align*}
    q
    = \div \gamma
    & = \div( \sym \gamma +  \mskw u)
    \\
    & = \div \sigma - \curl u
    && \text{by~\eqref{eq:div-mskw}},
    \\
    & = \div \sigma + \frac 1 2 \curl \div \tau
    \\
    & = \div \sigma + \frac 1 2 \curl \div \dev \tau
    \\
    & = \div( \sigma  + \sym  \curl \T \dev \tau)
    &&\text{by~\eqref{eq:4}}.
  \end{align*}
  Set ${\Rcd} q = \sigma + \sym \curl \T \dev \tau$. Then by the
  continuity of ${\Td}$ in~\eqref{eq:CM-div-T3}, the norm estimate in
  \eqref{eq:Rd} follows for any $q \in \cl D_{\ND}$. The proof is
  finished using the density result of Lemma~\ref{lem:dense-mean0}.
\end{proof}

\begin{lemma}
  \label{lem:bottom-row}
  There are continuous linear maps
  ${\Rg} : \ker( \curl: \HtRT(\curl)) \to \LR$, 
  ${\Rc}: \ker(\div: \HtND(\div)) \to (L_2 \otimes \bb V) \cap  \HtRT(\curl)$,
  and
  ${\Rd} : \HtPl \to (L_2 \otimes \bb V)\cap \HtND(\div)$,
  such that for any  $v \in \HtRT(\curl)$ with $\curl v=0$,
  $q \in \HtND(\div)$ with $\div q =0$, and $ w \in \HtPl$, we have 
  \begin{alignat*}{3}
    & \frac 1 3 \grad {\Rg} v = v, 
    &\qquad
    &
    \frac 1 2 \curl {\Rc} q = q, 
    &\qquad
    &
      \div {\Rd} w = w,
    \\
    &
      \| {\Rg} v \|_{L_2} \lesssim \| v \|_{\Ht}
    &
    &
      \| {\Rc} q \|_{L_2} \lesssim \| q \|_{\Ht},
    &
    &
      \| {\Rd} w \|_{L_2} \lesssim  \| w \|_{\Ht}.
  \end{alignat*}
\end{lemma}
\begin{proof}
  Let us construct the last operator first.  The functional action of
  any $w \in \HtPl$ on constant functions vanish, so we
  use~\eqref{eq:CM-div-T3} to conclude that
  ${\Td} w \in L_2 \otimes \bb V$ satisfies $\div {\Td} w = w$.

  We proceed to correct ${\Td}$ to obtain orthogonality to $\ND$. 
  Let  $u \in \Ho(\curl)$ satisfy
  \begin{align*}
    (\curl u, \curl v)   = ({\Td} w, \curl v), 
    \qquad 
    \div u  = 0,
  \end{align*}
  for all $v \in \Ho(\curl)$, a constrained formulation that is well
  known to be uniquely solvable~\cite{Monk03b}. The first equation
  above implies that
  \begin{gather}
    \label{eq:18}
    \| \curl u \|_{L_2} \le \| {\Td} w \|_{L_2} \lesssim \| w \|_{\Ht}, 
    \\ \label{eq:19}
    \curl \curl u = \curl {\Td} w.
  \end{gather}
  Put ${\Rd} w := {\Td}w - \curl u \in L_2 \otimes \bb V$. It  satisfies
  \[
    \curl {\Rd} w = 0, \qquad \div {\Rd} w = w, \qquad
    ({\Rd} w, r )=0
  \]
  for all $r \in \ND$. Indeed, the first equation follows
  from~\eqref{eq:19} and the second from $\div {\Td} w = w$. To see the
  third, first note that since the functional action of $w$ on any
  $p_1 \in \pol_1$ vanish,
  \[
    0 = w(p_1) = (\div {\Td} w )(p_1) = ({\Td} w, \grad p_1),
  \]
  so 
  $({\Td} w, a) = 0$ for any $a \in \bb V.$ Combined with 
  $(\curl u, a)=(u, \curl a) =0$, we have $({\Rd} w, a)=0$.
  Moreover,  since $\curl ( b x \cdot x) = -2 b \times x$ for any
  $b \in \bb V$, we have 
  \begin{align*}
    0 = (\curl {\Rd} w)(b x \cdot x) = ({\Rd} w, \curl ( b x \cdot x)) =
    ({\Rd} w, -2 b \times x),
  \end{align*}
  so ${\Rd} w $ is $L_2$-orthogonal to $a + b \times x$ for any $a, b \in \bb V$.
  By the continuity of ${\Td}$ and~\eqref{eq:18}, we also obtain
  the norm bound $\| {\Rd} w \|_{L_2} \lesssim \| w \|_{\Ht}$.

  The construction of ${\Rc}$ is similar: set
  ${\Rc} q := 2 ({\Tc} q - \grad w)$ where $w \in \Ho^1$ solves
  $(\grad w, \grad z) = ({\Tc}q, \grad z)$ for all $z \in
  \Ho^1$. Clearly, $\frac 1 2 \curl {\Rc} q =q$. Also, for any
  $a \in \bb V$, since the action of $q$ on any $\ND$-function
  vanishes,
  $0 = q( a \times x ) = (\curl {\Tc}q)( a \times x) = ({\Tc} q, 2 a)$, so
  $({\Rc} q, a)=0$. Moreover, for any $b \in \bb R$,
  \[
    ({\Rc} q, a + b x) = ({\Rc}q, bx) = \frac 1 2 ({\Rc}q, \grad(b x \cdot x)) =
    \frac 1 2 ( \div {\Rc} q)( b |x|^2)
  \]
  which must vanish since $\div {\Rc} q=0$ by construction, so ${\Rc} q$
  is $L_2$-orthogonal to $\RT$.

  Finally, simply setting ${\Rg} = 3 {\Tg}$, it is easy to verify the
  stated properties of ${\Rg}$.
\end{proof}

\begin{theorem}
  \label{thm:rt-inv-cts}
  Each $R$ and $D$ operator in~\eqref{eq:reverse-arrows} is a
  continuous right inverse of the differential operator marked above
  it.
\end{theorem}
\begin{proof} The result is proved by prior lemmas, as seen by the
  following pointers, which consider the operators, row by row, from
  left to right, but omitting the obvious ones in the first row and
  any obvious symmetrically opposite ones.
  
  \begin{itemize}

  \item    
    $\Dgg: \ker(\curl : \Hocc )\to \Ho^1 $ is a continuous right
    inverse of $\hess: \Ho(\grad) \to \Hocc$ by Lemma~\ref{lem:Dgg}
    applied with  $s=-1$.

  \item
    $\Dgc: \ker( \div: \Hocd) \cap \ker(\sym\curl \T: \Hocd) \to \Ho^1
    \otimes \bb V \subset \Ho(\curl)$ is a continuous right inverse of
    $\curl\dfo : \Ho(\curl) \to \Hocd$ by Lemma~\ref{lem:Dgc}.

  \item
    ${\Dgd}: \ker( \curl: \HtRT(\curl)) \to \Ho^1 \otimes \bb V
    \subset \Ho(\div)$ is a continuous right inverse of
    $\frac 1 3 \grad \div : \Ho(\div) \to \HtRT(\curl)$ by
    Lemma~\ref{lem:Dgd}.

  \item ${\Rgg}: \ker(\inc: \Hocc) \to \Ho(\curl)$ is a continuous
    right inverse of $\dfo: \Ho(\curl) \to \Hocc$ by
    Lemma~\ref{lem:Rgg}.

  \item ${\Rgc} : \ker(\div: \Hocd) \to L_2 \otimes \bb S \subset \Hocc$
    is a continuous right inverse of $\curl :\Hocc \to \Hocd $ by
    Lemma~\ref{lem:Rgc}.

  \item   ${\Rgd}: \HtRT(\curl) \to L_2 \otimes \bb T \subset \Hocd$
    is a continuous right inverse of $\div: \Hocd \to \HtRT(\curl)$ by
    Lemma~\ref{lem:Rgd}.

  \item ${\RgcT} : \ker(\sym\curl: \HocdT) \to \Ho(\div)$ is a
    continuous right inverse of
    $\frac 1 2 \dev \grad : \Ho(\div) \to \HocdT$ by
    Lemma~\ref{lem:RgcT}.
    
  \item $\Dcc: \ker(\div: \Hodd) \to \Ho^1 \otimes \bb S$ is a
    continuous right inverse of $\inc: \Hocc \to \Hodd$ by
    Lemma~\ref{lem:Dcc} applied with $s=-1$.

  \item
    ${\Dcd}: \ker( \div: \HtND(\div) ) \to \Ho^1 \otimes \bb T
    \subset \Hodd$ is a continuous right inverse of
    $\frac 1 2 \curl\div : \Hocd \to \HtND(\div)$ by
    Lemma~\ref{lem:Dcd} applied with $s=-1$.

  \item 
    ${\Rcc} : \ker(\div\div: \Hodd) \to L_2\otimes \bb T \subset \HocdT$
    is a continuous right inverse of $\sym\curl: \HocdT \to \Hodd$
    by Lemma~\ref{lem:Rcc}.

  \item     ${\Rcd}: \HtND(\div) \to L_2 \otimes \bb S \subset \Hodd $
    is a continuous right inverse of $\div: \Hocd \to \HtND(\div)$
    by Lemma~\ref{lem:Rcd}.

  \item ${\Ddd}: \HtPl \to \Ho^1\otimes \bb S \subset \Hodd $ is
    a continuous right inverse of $\div\div: \Hodd \to \HtPl$ by
    Lemma~\ref{lem:Ddd} applied with $s=-1$.

  \item ${\Rg} : \ker( \curl: \HtRT(\curl)) \to \LR$,
    ${\Rc}: \ker(\div: \HtND(\div)) \to (L_2 \otimes \bb V) \cap
    \HtRT(\curl)$, and
    ${\Rd} : \HtPl \to (L_2 \otimes \bb V)\cap
    \HtND(\div)$, are continuous right inverses
    of $\frac 1 3 \grad: \LR \to \HtRT(\curl)$,
    $\frac 1 2 \curl: \HtRT(\curl) \to \HtND(\div)$,
    and
    $\div: \HtND(\div) \to \HtPl$, respectively, by
    Lemma~\ref{lem:bottom-row}.
  \end{itemize}
\end{proof}

\begin{corollary}
  \label{cor:closed-reange}
  The range of every differential operator
  in~\eqref{eq:reverse-arrows} is closed.
\end{corollary}
\begin{proof}
  This is an immediate consequence of the existence of continuous
  right inverses for each differential operator 
  in~\eqref{eq:reverse-arrows}, as we have shown.

  For example, consider the operator $\curl \dfo.$ Clearly its range
  is contained in both $\ker(\div: \Hocd)$ and
  $\ker(\sym\curl\T : \Hocd)$.  But Lemma~\ref{lem:Dgc} shows that
  the intersection of these kernels is also contained in the range of
  $\curl \dfo$, so
  \[
    \ran( \curl \dfo) = \ker(\div: \Hocd) \cap \ker(\sym\curl\T : \Hocd).
  \]
  By the continuity of the differential operators proved in
  Theorem~\ref{thm:diagram-commute-cty}, both the above kernels are
  closed, and so is their intersection. Hence the range of
  $\curl \dfo$ is closed.

  The proofs for the remaining operators  are similar and simpler. 
\end{proof}

\begin{corollary}
  The hessian complex \eqref{eq:hessian-complex}, the elasticity
  complex \eqref{eq:elasticity-complex}, and the div-div complex
  \eqref{eq:div-div-complex} are exact complexes.
\end{corollary}
\begin{proof}
  To prove that the hessian complex \eqref{eq:hessian-complex} is
  exact it suffices to prove that
  $
  \ran(\hess) \supseteq \ker(\curl: \Hocc),$
  $
    \ran(\curl) \supseteq \ker(\div: \Hocd), 
    $
    and
    $\ran(\div)  \supseteq \HtRT(\curl)$
    (since the reverse inclusions are clear from
  Corollary~\ref{cor:3complexes}).  But these are now obvious by the
  existence of continuous right inverses $\Dgg, {\Rgc},$ and ${\Rgd}$
  given by Lemmas~\ref{lem:Dgg}, \ref{lem:Rgc}, and~\ref{lem:Dgd},
  respectively.

  Similarly, the continuous right inverse operators $\Rgg$, $\Dcc$,
  and ${\Rcd}$, given by Lemmas~\ref{lem:Rgg},~\ref{lem:Dcc},
  and~\ref{lem:Rcd}, prove the exactness of the elasticity complex.

  The exactness of the div div complex similarly follows from the
  continuous right inverse operators of $\RgcT, \Rcc$ and $\Ddd$ of 
  Lemmas~\ref{lem:RgcT}, \ref{lem:Rcc}, and~\ref{lem:Ddd}.
\end{proof}

\section{Slightly more regular spaces of matrix fields}
\label{sec:slightly-more-regular}

Consider the following  slightly more regular spaces of matrix fields, 
contained in the previously introduced spaces $\Hocc$, $\Hocd$ and
$\Hodd$: 
\begin{align}
  \toHcc
  & = \{ g \in L_2 \otimes \bb S:    \inc g \in \Ht \otimes \bb S\},
  \\
  \toHcd
  & = \{ \tau \in L_2 \otimes \bb T:  \curl\div \tau \in \Ht \otimes \bb V\},
  \\
  \toHdd
  & = \{ \sigma \in L_2 \otimes \bb S: \div \div \sigma \in \Ht\},
\end{align}
whose natural norms are defined  respectively by
\begin{align*}
  \| g \|^2_{\toHcc}
  & = \| g \|_{L_2}^2 + \| \inc g \|_{\Ht}^2,
  \\
  \| \tau \|_{\toHcd}^2
  & = \| \tau \|_{L_2}^2 + \| \curl \div \tau \|_{\Ht}^2,
  \\
  \| \sigma \|_{\toHdd}^2
  & = \| \sigma \|_{L_2}^2 + \| \div\div \sigma \|_{\Ht}^2.
\end{align*}
Such spaces of matrix fields and their even smoother versions, have
emerged in recent works~\cite{ArnolHu21, GopalLederSchob20,
  PechsSchob18}.

Note that one possible way to increase the regularity of the prior
spaces $\Hocc$, $\Hocd$ and $\Hodd$ is to uniformly replace $\Ht$ by
$\Hs{s}$ with some $s > -1$ in \eqref{eq:Hcc-cd-dd-norms} (and we
expect the prior analysis to go through with minimal changes for such
modification).  The new spaces of this section, $\toHcc, \toHcd,$ and
$\toHdd$, are not obtained this way. Instead, they are obtained by increasing
the regularity of the matrix-valued function to $L_2$ while
maintaining the same $\Ht$-regularity for its second-order derivative.
We proceed to prove regular decompositions and density of smooth
functions for such spaces.

\begin{theorem}[Regular decomposition of $\toHcc$]
  \label{thm:reg-dec-Hinc}
  There exist continuous linear operators
  $\Socc 0 : \toHcc \to \Ho^1 \otimes \bb S$,
  $\Socct 1 : \toHcc \to \Ho^1 \otimes \bb V$
  such that any $g \in \toHcc$
  can be decomposed into
  \begin{equation}
    \label{eq:Hinc-reg-dec}
    g = \Socc 0 g + \dfo \Socct 1 g.
  \end{equation}
  Consequently,
  \begin{equation}
    \label{eq:Hinc-density}
    \toHcc
    = \clos{ \cl D (\om) \otimes \bb S}{ \| \cdot \|_{\toHcc}}.    
  \end{equation}
\end{theorem}
\begin{proof}
  Since any $g \in \toHcc$ in also in $\Hocc$, we apply the operators
  $\Socc 0$ and $\Socc 1$ of Theorem~\ref{thm:reg-dec-Hcc} and follow
  along the lines of its proof to obtain~\eqref{eq:19a},
  $ \curl \big( g - \Socc 0 g - \dfo \Socc 1 g \big) = 0.$ Setting
  $\Socc 2 g = \Dgg(g - \Socc 0 g - \dfo \Socc 1 g)$ we apply
  Lemma~\ref{lem:Dgg}, but this time with $s=0$ since $g$ is now in
  $L_2 \otimes \bb S$, to get that
  \[
    \Socc 2 g \in \Hts{2} \otimes \bb S,
    \qquad 
    g - \Socc 0 g - \dfo \Socc 1 g = \hess \Socc 2 g .   
  \]
  Now~\eqref{eq:Hinc-reg-dec} follows setting
  $\Socct 1 g = \Socc 1 g + \grad \Socc 2 g$.

  To prove~\eqref{eq:Hinc-density}, decompose $g \in \toHcc$ as
  above. By the density of $\cl D(\om)$ in $\Ho^1$, there are
  $g_m \in \cl D(\om) \otimes \bb S$ and
  $u_m \in \cl D(\om) \otimes \bb V$ such that
  \[
    \| g_m - \Socc 0 g \|_{H^1} \to 0, \qquad \| u_m - \Socc 1 g
    \|_{H^1} \to 0, \qquad
  \]
  as $m \to \infty$. Hence, $g_m + \dfo u_m \in \cl D(\om) \otimes \bb S$ and 
  \begin{align*}
    \| g - (g_m + \dfo u_m)\|_{\toHcc}^2
    \le \| g - g_m \|_{L_2}^2 + \| \inc (g - g_m) \|_{\Ht}^2
    \lesssim \| g - g_m \|_{H^1}^2
  \end{align*}  
  by \eqref{eq:di-cty}. Since the last term converges to zero, we
  have proved that $\toHcc$ is contained in the closure of
  $\cl D (\om) \otimes \bb S$. The reverse inclusion is obvious.
\end{proof}

\begin{theorem}[Regular decomposition of $\toHdd$]
  \label{thm:reg-dec-Hdivdiv}
  There exist  continuous linear operators
  $\Sodd 0 : \Hodd \to \Ho^1 \otimes \bb S$ and 
  $\Soddt 1 : \Hodd \to \Ho^1 \otimes \bb T$
  such that any $\sigma \in \toHdd$
  can be decomposed into
  \begin{equation}
    \label{eq:Hdivdiv-reg-dec}
    \sigma = \Sodd 0 \sigma + \sym\curl \Soddt 1 \sigma.
  \end{equation}
  Consequently,
  \begin{equation}
    \label{eq:Hdivdiv-density}
    \toHdd
    = \clos{ \cl D (\om) \otimes \bb S}{ \| \cdot \|_{\toHdd}}.
  \end{equation}
\end{theorem}
\begin{proof}
  We proceed along the lines of the proof of
  Theorem~\ref{thm:reg-dec-Hdd}, but now with the more regular
  $\sigma$ in $\toHdd$, to obtain that
  \[
    \div \big( \sigma - \Sodd 0 \sigma -  \sym\curl \Sodd 1 \sigma \big) =0,
  \]
  with $\Sodd 0$ and $\Sodd 1$ as defined there. At this point, we apply 
  Lemma~\ref{lem:Dcc}, now with $s=0$,
  setting
  $\Sodd 2 \sigma := {\Dcc} ( \sigma - \Sodd 0 \sigma - \sym\curl \Sodd 1
  \sigma )$, to  find that
  \[
    \Socc 2 \sigma \in \Hts{2}\otimes \bb S,
    \qquad \sigma - \Sodd 0 \sigma - \sym\curl \Sodd 1  \sigma = \inc \Sodd 2 \sigma.
  \]
  The result now follows setting
  $\Soddt 1 \sigma = \Sodd 1 \sigma + \T \curl \Sodd 2 \sigma$.

  To conclude, let $\sigma \in \toHdd$ and use the just proved
  decomposition~\eqref{eq:Hdivdiv-reg-dec} to split it into 
  $\sigma = \Sodd 0 \sigma + \sym\curl \Soddt 1 \sigma.$ By the density
  of $\cl D (\om)$ in $\Ho^1$, there are
  $\sigma_m \in \cl D(\om) \otimes \bb S$ and
  $\tau_m \in \cl D(\om) \otimes \bb T$ such that
  \[
    \| \sigma_m - \Sodd 0 \sigma \|_{H^1} \to 0, \qquad
    \| \tau_m - \Soddt 1 \sigma \|_{H^1} \to 0, \qquad
  \]
  as $m \to \infty$. Hence, by \eqref{eq:di-cty} and~\eqref{eq:4}
  $\sigma_m + \sym\curl \tau_m\in \cl D(\om) \otimes \bb S$ converges
  to $\sigma$ in $\| \cdot \|_{\toHdd}$~norm, thus proving that
  $\sigma$ is contained in the closure of $\cl D (\om) \otimes \bb
  S$. Since the reverse inclusion is obvious,
  \eqref{eq:Hdivdiv-density} is proved.
\end{proof}

\begin{theorem}[Regular decomposition of $\toHcd$]
  \label{thm:reg-dec-Hcurldiv}
  There exist three continuous linear operators
  \[
    \Socd 0 : \toHcd \to \Ho^1 \otimes \bb T,\quad
    \Socd 1 : \toHcd \to \Ho^1 \otimes \bb S,\quad
    \Socdt 2 : \toHcd \to \Ho^1 \otimes \bb V,
  \]
  such that any $\tau \in \toHcd$ can be decomposed into
  \begin{equation}
    \label{eq:Hcurldiv-dec}
    \tau = \Socd 0 \tau + \curl \Socd 1 \tau +  \T \dev \grad \Socdt 2 \tau.
  \end{equation}
  Consequently, 
  \begin{equation}
    \label{eq:Hcurldiv-density}
    \toHcd
    = \clos{ \cl D (\om) \otimes \bb T}{ \| \cdot \|_{\toHcd}}.
  \end{equation}
\end{theorem}
\begin{proof}
  We follow along the lines of the proof of
  Theorem~\ref{thm:reg-dec-Hcd} to find that given any
  $\tau \in \toHcd$, using the same $\Socd 0$ operator defined there, we
  have $\curl \div (\tau - \Socd 0 \tau) =0$. But now, since the higher
  regularity of $\tau$ implies that $\tau - \Socd 0 \tau$ is in
  $L_2 \otimes \bb T$, we may apply the two-term decomposition of
  Lemma~\ref{lem:Rcdh+Dcdh+Ucdh} (taking $\Ugch$ there to be zero) instead
  of the three-term decomposition to obtain
  \[
    \tau - \Socd 0 \tau =
    (\curl {\Rgch}  +  \T \dev \grad \Dgch )(\tau - \Socd 0 \tau).
  \]
  The decomposition~\eqref{eq:Hcurldiv-dec} now follows after
  setting $\Socd 1 \tau = {\Rgch} (\tau - \Socd 0 \tau)$, and
  $\Socdt 2 \tau = \Dgch (\tau - \Socd 0 \tau)$.

  To prove \eqref{eq:Hcurldiv-density}, let $\tau \in \toHcd$ be
  decomposed as in \eqref{eq:Hcurldiv-dec}. By the density of
  $\cl D(\om)$ in $\Ho^1$, there are
  $\tau_m \in \cl D(\om) \otimes \bb T$,
  $g_m \in \cl D(\om) \otimes \bb S$ and
  $q_m \in \cl D(\om) \otimes \bb V$ which converge to the
  decomposition components $\Socd 0 \tau$, $ \Socd 1 \tau$ and
  $\T \dev \grad \Socdt 2 \tau$, respectively, in the $H^1$-norm, as
  $m \to \infty$. Hence
  \begin{align*}
    \| \tau_m
    & + \curl g_m + \T \dev \grad q_m -\tau \|_{\toHcd}
    \\
    & \le \| \tau_m - \Socd 0 \tau \|_{\toHcd}
      +
      \| \curl (g_m - \Socd 1 \tau) \|_{\toHcd}
      +
      \| \T \dev \grad (q_m - \Socdt 2 \tau)\|_{\toHcd}
    \\
    &\lesssim
      \| \tau_m - \Socd 0 \tau \|_{H^1}
      +
      \| \curl (g_m - \Socd 1 \tau) \|_{L_2}
      +
      \| \T \dev \grad (q_m - \Socdt 2 \tau)\|_{L_2}
    \\
    & \lesssim
      \| \tau_m - \Socd 0 \tau \|_{H^1}
      +
      \| g_m - \Socd 1 \tau \|_{H^1}
      +
      \| q_m - \Socdt 2 \tau\|_{H^1}
  \end{align*}
  where we have used \eqref{eq:5} (which implies that
  $\curl \div \circ \T \dev \grad =0$) and~\eqref{eq:di-cty}. Since
  the last bound converges to zero as $m \to \infty$, we have just 
  exhibited a sequence in $\cl \D(\om) \otimes \bb T$ that
  approximates any given $\tau \in \toHcd$ arbitrarily closely in
  $\toHcd$~norm.
\end{proof}

\section{Duality}
\label{sec:duality}

In this section, we state extensions to 2-complexes built
using Sobolev spaces that are generally not closures of compactly
supported smooth functions, such as those
in~\eqref{eq:std-Sobolev-spaces}, as well as 
spaces built using $\Hm$ instead of $\Ht$, such as 
\begin{align*}
  \Hm(\curl)
  & = \{ u \in \Hm \otimes \bb V: \curl u \in \Hm \otimes \bb V \},
  \\
  \Hm(\div)
  & = \{ q \in \Hm \otimes \bb V: \div q \in \Hm \},
\end{align*}
spaces of matrix-valued functions,
\begin{subequations}
  \label{eq:Hcc-Hdd-Hcd}
\begin{align}
  \Hcc
  & = \{ g \in \Hm \otimes \bb S: \curl g \in \Hm \otimes \bb V,
    \inc g \in \Hm \otimes \bb S \},
  \\  \nonumber 
  \Hcd & = \{ \tau \in \Hm  \otimes \bb T: \div \tau \in \Hm  \otimes \bb V,
         \sym\curl \tau \in \Hm  \otimes \bb S,
  \\ \label{eq:Hcd-defn}
  & \hspace{6.5cm} \curl\div\tau \in \Hm  \otimes \bb V \},
  \\ \label{eq:Hdd-defn}
  \Hdd & = \{ \sigma \in \Hm  \otimes \bb S: \div \sigma \in \Hm  \otimes \bb V,
         \div \div\sigma \in \Hm  \},
\end{align}
\end{subequations}
and their slightly more regular versions,
\begin{subequations}
  \label{eq:Hcc-Hdd-Hcd-smoother}
\begin{align}  
  \tHcc & = \{ g \in L_2 \otimes \bb S:    \inc g \in \Hm \otimes \bb S\},
  \\ \label{eq:Hdd-smoother-defn}
  \tHdd & = \{ \sigma \in L_2 \otimes \bb S: \div \div \sigma \in \Hm\},
  \\ \label{eq:Hcd-smoother-defn}
  \tHcd & = \{ \tau \in L_2 \otimes \bb T:  \curl\div \tau \in \Hm \otimes \bb V\}.  
\end{align}
\end{subequations}
Analogues of previous results can be proved for 
\begin{equation}
  \label{eq:no-bc-2-complex}
  \begin{tikzcd}
    [
    row sep=hyper2,
    column sep=large,
    ampersand replacement=\&
    ]
    H^1
    \arrow{r}{\grad}
    \arrow[d, "\grad"{sloped}]
    \arrow[dr, "\hess"{sloped}]
    \&[1em]  
    H(\curl)
    \arrow[r, "\curl"]
    \arrow[d, "\deff"{sloped}]
    \arrow[dr, "\curl\dfo"{sloped}]
    \&
    H(\div) 
    \arrow{r}{\div}
    \arrow[d, "\;\;\frac 1 2\!\! \T\dev\grad\!"{sloped}]
    \arrow[dr, "\frac 1 3 \grad \div"{sloped}]
    \&
    L_2
    \arrow[d, "\frac 1 3 \grad"]
    \\  
    H(\curl)
    \arrow{r}{\deff}
    \arrow[d, "\curl"{sloped}]
    \arrow[dr, "\T \curl \dfo"{sloped}]
    \&
    {\Hcc}
    \arrow{r}{\curl}
    \arrow[d, "\T\curl"{sloped}]
    \arrow[dr, "\inc"{sloped}]
    \&
    {\Hcd}
    \arrow{r}{\div}
    \arrow[d, "\sym\curl\T"{sloped}]
    \arrow[dr, "\quad\frac 1 2 \curl \div"{sloped}]
    \&
    {H^{-1}(\curl)}
    \arrow[d, "\frac 1 2 \curl"]      
    \\
    H(\div)
    \arrow[r, "\frac 1 2 \dev\grad"]
    \arrow[d, "\div"{sloped}]
    \arrow[dr, "\frac 1 3 \grad \div"{sloped}]
    \& 
    {\HcdT}
    \arrow{r}{\sym\curl}
    \arrow[d, "\div\T"{sloped}]
    \arrow[dr, "\frac 1 2 \curl \div \T"{sloped}]
    \&
    {\Hdd}
    \arrow{r}{\div}
    \arrow[d, "\div"{sloped}]
    \arrow[dr, "\div \div"{sloped}]
    \&
    {H^{-1}(\div)}
    \arrow[d, "\div"]
    \\
    L_2
    \arrow{r}{\frac 1 3 \grad}
    \&
    {H^{-1}(\curl)}
    \arrow{r}{\frac 1 2 \curl}
    \&
    {H^{-1}(\div)}
    \arrow{r}{\div}
    \&
    {H^{-1}}
  \end{tikzcd}
\end{equation}
using exactly the same techniques. They are summarized in the next
theorem. We use $\cl D(\bar \om)$ to denote the space of restrictions of functions 
in $\cl D(\bb R^3)$ to the closure $\bar \om$.

\begin{theorem}
  \label{thm:no-bc-spaces}
  \hfill
  \begin{enumerate}
  \item  The diagram~\eqref{eq:no-bc-2-complex} commutes.
  \item Every path in it is a 2-complex.
  \item Every differential operator in it is continuous,  has closed
    range, and has a continuous right inverse. 
  \item There are regular decompositions for $\Hcc, \Hdd,$ and $\Hcd$. Namely, 
    there exist continuous linear operators 
    $ \Scc 0 : \Hcc \to {H}^1 \otimes \bb S, \;
      \Scc 1 : \Hcc \to {H}^1 \otimes \bb V, \;
      \Scc 2 : \Hcc \to {H}^1,
      \Sdd 0 : \Hdd \to {H}^1 \otimes \bb S, \;
      \Sdd 1 : \Hdd \to {H}^1 \otimes \bb T, \;
      \Sdd 2 : \Hdd \to {H}^1 \otimes \bb S,
      \Scd 0 : \Hcd \to {H}^1 \otimes \bb T,\;
      \Scd 1 : \Hcd \to {H}^1 \otimes \bb S,\;
      \Scd 2 : \Hcd \to {H}^1 \otimes \bb V, \;
      \Scd 3 : \Hcd \to {H}^1 \otimes \bb V,\;    
      $
      such that
      \begin{align}
        \label{eq:reg-dec-Hcc-nobc}
        g
        & = \Scc 0 g + \dfo \Scc 1 g + \hess \Scc 2 g,
        && g \in \Hcc,
        \\
        \label{eq:reg-dec-Hdd-nobc}
        \sigma
        & = \Sdd 0 \sigma + \sym\curl \Sdd 1 \sigma + \inc \Sdd 2 \sigma,
          && \sigma \in \Hdd,
        \\
        \label{eq:reg-dec-Hcd-nobc}
        \tau
        & = \Scd 0 \tau + \curl \Scd 1 \tau +  \T \dev \grad \Scd 2 \tau + 
          \curl\dfo \Scd 3 \tau,
        && \tau \in \Hcd.
      \end{align}

    \item The spaces
      $\cl D(\bar\om) \otimes \bb S, \cl D(\bar\om) \otimes \bb T,$
      and $\cl D(\bar\om) \otimes \bb S$ are dense in $\Hcc, \Hcd,$
      and $\Hdd$, respectively.
    \item There are regular decompositions for $\tHcc, \tHdd,$ and
      $\tHcd$. Namely, there are continuous linear operators
      $\Scct 1 : \tHcc \to \Ho^1 \otimes \bb V$,
      $\Scdt 2 : \tHcd \to \Ho^1 \otimes \bb S$,
      $\Sddt 1 : \tHdd \to \Ho^1 \otimes \bb T$,
      such that
      \begin{align}
        g
        & = \Scc 0 g + \dfo \Scct 1 g,
        && g \in \tHcc, 
        \\
        \tau
        & = \Scd 0 \tau + \curl \Scd 1 \tau +  \T \dev \grad \Scdt 2 \tau,
        && \tau \in \tHcd, 
        \\
        \sigma
        & = \Sdd 0 \sigma + \sym\curl \Sddt 1 \sigma,
        && \sigma \in \tHdd.
      \end{align}

    \item The spaces
      $\cl D(\bar\om) \otimes \bb S, \cl D(\bar\om) \otimes \bb T,$
      and $\cl D(\bar\om) \otimes \bb S$ are dense in $\tHcc, \tHcd,$
      and $\tHdd$, respectively.
  \end{enumerate}  
\end{theorem}
\begin{proof}
  Proceed as in the proofs of Theorems~\ref{thm:diagram-commute-cty},
  \ref{thm:2-complex}, and \ref{thm:rt-inv-cts}.
\end{proof}

Some spaces at the edges  of the diagram~\eqref{eq:no-bc-2-complex} can
be restricted to certain subspaces of interest without affecting the
matrix-valued function spaces to get the following commuting diagram:
\begin{equation}
  \label{eq:no-bc-2-complex-quotient}
  \begin{tikzcd}
    [
    row sep=hyper2,
    ampersand replacement=\&
    ]
    H^1/\pol_1
    \arrow{r}{\grad}
    \arrow[d, "\grad"{sloped}]
    \arrow[dr, "\hess"{sloped}]
    \&[1em]  
    H(\curl)/\ND
    \arrow[r, "\curl"]
    \arrow[d, "\deff"{sloped}]
    \arrow[dr, "\curl\dfo"{sloped}]
    \&
    H(\div)/\RT
    \arrow{r}{\div}
    \arrow[d, "\;\;\frac 1 2\!\! \T\dev\grad\!"{sloped}]
    \arrow[dr, "\frac 1 3 \grad \div"{sloped}]
    \&
    L_2/\bb R
    \arrow[d, "\frac 1 3 \grad"]
    \\  
    H(\curl)/\ND
    \arrow{r}{\deff}
    \arrow[d, "\curl"{sloped}]
    \arrow[dr, "\T \curl \dfo"{sloped}]
    \&
    {\Hcc}
    \arrow{r}{\curl}
    \arrow[d, "\T\curl"{sloped}]
    \arrow[dr, "\inc"{sloped}]
    \&
    {\Hcd}
    \arrow{r}{\div}
    \arrow[d, "\sym\curl\T"{sloped}]
    \arrow[dr, "\quad\frac 1 2 \curl \div"{sloped}]
    \&
    {H^{-1}(\curl)}
    \arrow[d, "\frac 1 2 \curl"]      
    \\
    H(\div)/\RT
    \arrow[r, "\frac 1 2 \dev\grad"]
    \arrow[d, "\div"{sloped}]
    \arrow[dr, "\frac 1 3 \grad \div"{sloped}]
    \& 
    {\HcdT}
    \arrow{r}{\sym\curl}
    \arrow[d, "\div\T"{sloped}]
    \arrow[dr, "\frac 1 2 \curl \div \T"{sloped}]
    \&
    {\Hdd}
    \arrow{r}{\div}
    \arrow[d, "\div"{sloped}]
    \arrow[dr, "\div \div"{sloped}]
    \&
    {H^{-1}(\div)}
    \arrow[d, "\div"]
    \\
    L_2/\bb R
    \arrow{r}{\frac 1 3 \grad}
    \&
    {H^{-1}(\curl)}
    \arrow{r}{\frac 1 2 \curl}
    \&
    {H^{-1}(\div)}
    \arrow{r}{\div}
    \&
    {H^{-1}}
  \end{tikzcd}
\end{equation}
This is because of the following facts:
(a)~$\grad \pol_1 = \bb V$ is in the zero coset of $H(\curl)/\ND$,
(b)~$\curl \ND = \bb V$ is in the zero coset of $H(\div) / \RT$,
(c)~$\div \RT = \bb R$ is in the zero coset of $L_2/\bb R$,
(d)~$\dfo \ND = 0$, and 
(e)~$\dev \grad \RT =0$.

Next we turn to establishing certain duality relationships between the spaces  in the diagram~\eqref{eq:1} and those in the just introduced diagram~\eqref{eq:no-bc-2-complex-quotient}.
First
we need a lemma that enlarges the  domain of
certain functionals.  Consider a $q \in H^{-1}(\div)$. Then, since $q$
is in $H^{-1} \otimes \bb V$, its action on a function in
$\Ho^1 \otimes \bb V$ is well defined, so the action of $q$ on
gradient fields, $q(\grad w)$, is well defined provided 
$\grad w \in \Ho^1 \otimes \bb V$. But in fact, it is also well defined if
$\grad w$ is just in $L_2 \otimes \bb V$, as stated next, where similar
other extensions are also collected.

\begin{lemma}
  \label{lem:extensions-to-range}
  Let $q \in H^{-1}(\div)$, $\sigma \in \Hdd$, $g \in \Hcc$, and
  $\tau \in \Hcd$.  Then $q \circ \grad$, $\sigma \circ \dfo$, extends
  to continuous linear maps such that 
  \begin{subequations}    
    \begin{alignat}{4}
      \label{eq:q-grad}
      q \circ \grad
      & : \Ho^1 \to \bb R,
      &\quad (q \circ \grad)(w )
      &  = - (\div q)(w),
      &&\quad  w \in \Ho^1,
      \\ \label{eq:sigma-eps}
      \sigma \circ \dfo
      & : \Ho^1 \otimes \bb V \to \bb R,
      & \quad (\sigma \circ \dfo) (u)
      & = -(\div \sigma )(u),
      &&\quad  u \in \Ho^1 \otimes \bb V,
      \\ \label{eq:sigma-hess}
      \sigma \circ \hess
      &: \Ho^1 \otimes \bb S \to \bb R,
      & \quad (\sigma \circ \hess)(w)
      & = (\div\div \sigma)(w),
      &&\quad  w \in \Ho^1,
      \\ \label{eq:g-sc}
      g\circ \sym\curl
      & : \Ho^1 \otimes \bb T \to \bb R,
      & \quad (g\circ \sym\curl)(\eta)
      & = (\curl g)(\eta),
      && \quad \eta \in \Ho^1 \otimes \bb T,
      \\ \label{eq:g-inc}
      g \circ \inc
      &: \Ho^1 \otimes \bb S \to \bb R,
      & \quad (g \circ \inc)(\gamma)
      & = (\inc g)(\gamma),
      &&\quad \gamma \in \Ho^1 \otimes \bb S,
      \\
      \label{eq:tau-curl}
      \tau\circ \curl
      & : \Ho^1 \otimes \bb T \to \bb R,
      & \quad (\tau\circ \curl)(\gamma)
      & = (\sym\curl \tau)(\gamma),
      &&\quad \gamma \in \Ho^1 \otimes \bb S,
      \\
      \label{eq:tau-dev-grad}
      \tau \circ \dev\grad
      &: \Ho^1 \otimes \bb V \to \bb R,
      & \quad (\tau \circ \dev\grad)(u)
      & = -(\div \tau)( u),
      &&\quad u \in \Ho^1 \otimes \bb V,
      \\ \label{eq:tau-curl-def}
      \tau \circ \curl \dfo 
      &: \Ho^1 \otimes \bb V \to \bb R, 
      & \quad (\tau \circ \curl \dfo )(u)
      & = \frac 1 2 (\curl \div \tau )(u),
      &&\quad u \in \Ho^1 \otimes \bb V,
    \end{alignat}
  \end{subequations}
  where, component wise, the right hand sides are all duality pairings
  between $\Hm$ and $\Ho^1$.
\end{lemma}
\begin{proof}
  For any $\vphi \in \cl D(\om)$, by the definition of the distributional
  divergence,
  $(\div q )(\vphi) = -q( \grad \vphi) \equiv - (q \circ \grad)
  (\vphi).$ Since $\div q $ is in $\Hm(\div)$, 
  \[
    |(q \circ \grad)(\vphi) | \le \| \div q \|_{\Hm} \| \vphi\|_{\Ho^1}
  \]
  showing that $q \circ \grad$ can be continuously extended to the
  closure of $\cl D(\om)$ in the $\Ho^1$-norm, i.e., to $\Ho^1$. This
  proves the first statement. The remaining statements are proved
  similarly, noting that at times we must use the identities of
  Lemma~\ref{lem:commute-identities}: e.g., for any given 
  $\vphi \in \cl D(\om) \otimes \bb V$, using \eqref{eq:2} and \eqref{eq:5},
  we have
  $(\curl\div \tau)(\vphi) = (\div\T\curl \T\tau)(\vphi) =
  \T \tau  (\curl \T \grad\vphi)
  =
  \T \tau   (\T\dev \grad \curl \vphi)
   = 2\tau (\curl\dfo \vphi).
   $
\end{proof}

It was proved in \cite{GopalLederSchob20a} that the dual of
$\Ho(\div)$ equals $\Hm(\curl)$, both algebraically and topologically.
The next lemma states this together with closely related identities.

\begin{lemma}
  \label{lem:duality-vector-cases}
  The following equalities of spaces hold algebraically and topologically.
  \begin{alignat*}{4}
    &\Ho(\div)^* &=& H^{-1}(\curl), &\qquad& \Ho(\curl)^*&=&H^{-1}(\div),
    \\
    &H(\div)^* &= &\Ht(\curl), &\qquad& H(\curl)^*&=&\Ht(\div).
  \end{alignat*}
\end{lemma}
\begin{proof}
  The first identity was proved in
  \cite[Theorem~2.2]{GopalLederSchob20a}.  Here we prove the second
  and the last (since the third is similar).  The proofs are
  presented step by step below, but a unified strategy for proving all identities
  will be evident: we use the Riesz representative of a functional in
  $X^*$ to show that $X^* \hookrightarrow Y$, and then use a regular
  decomposition of $X$ to prove that $Y \hookrightarrow X^*$.

  {\em Step~1}. \underline{$\Ho(\curl)^* \hookrightarrow \Hm(\div)$:} \quad To prove
  that $\Ho(\curl)^*$ is continuously embedded in $ \Hm(\div)$,
  consider an  $f \in \Ho(\curl)^*$.
  By the Riesz representation theorem, there is a $u_f \in \Ho(\curl)$
  such that
  \[
    f(v) = (u_f, v) + (\curl u_f, \curl v), \qquad v \in \Ho(\curl).
  \]
  Choosing $v \in \cl{D}(\om) \otimes \bb V$ we find that the equality 
  \[
    f = u_f + \curl \curl u_f
  \]
  holds as distributions on $\om$. It immediately follows that
  \begin{align*}
    \| f \|_{\Hm(\div)}^2
    & =   \| u_f  + \curl \curl u_f \|_{\Hm}^2 + \| \div u_f\|_{\Hm}^2
    \\
    & \lesssim \| u_f \|_{\Ho(\curl)} = \| f \|_{\Ho(\curl)^*}.
  \end{align*}
  This proves that the restriction of the distribution $f$ to
  $\bar \om$ is an element of $\Hm(\div)$ and that the embedding of
  $\Ho(\curl)^*$ into $\Hm(\div)$ is continuous.

  {\em Step~2}. \underline{$ \Hm(\div) \hookrightarrow \Ho(\curl)^*$:} \quad Let
  $q \in \Hm(\div)$.  Applying the regular
  decomposition~\eqref{eq:Hocurl-reg-dec} to split any $u$ in
  $\Ho(\curl)$ as $u= \grad \Soc 1 u + \Soc 0 u$, we define a functional
  $f_q$ acting on $u$ by
  \begin{equation}
    \label{eq:fq-defn}
    f_q(u) = -(\div q )(\Soc 1 u) + q ( \Soc 0 u).
  \end{equation}
  By the continuity of $\Soc i$, 
  \begin{align*}
    |f_q(u) |
    & \le \| \div q \|_{\Hm} \| \Soc 1 u \|_{\Ho^1} + \| q \|_{\Hm} \| \Soc 0 u \|_{\Ho^1}
    \\
    & \lesssim
      \| q \|_{\Hm(\div)} \| u \|_{\Ho(\curl)},
  \end{align*}
  so $f_q \in \Ho(\curl)^*$ and
  \begin{equation}
    \label{eq:fq-q-bd}
    \| f_q \|_{\Ho(\curl)^*} \lesssim \| q \|_{\Hm(\div)}.
  \end{equation}
  Note that $f_q$ is a distribution (which can be seen for instance by
  the previously proved imbedding showing that $f_q$ is in
  $\Hm(\div)$). We now show 
  that the distribution $f_q$ is identical to
  $q$. Indeed, by~\eqref{eq:q-grad} of
  Lemma~\ref{lem:extensions-to-range}, for any
  $\vphi \in \cl D(\om) \times \bb V$, the
  definition~\eqref{eq:fq-defn} implies
  $f_q(\vphi ) = q(\grad \Soc 1 \vphi + \Soc 0 \vphi ) = q (\vphi),$
  i.e.,  $f_q = q$. Thus~\eqref{eq:fq-q-bd} implies that
  $\| f_q \|_{\Ho(\curl)^*} \equiv \| q \|_{\Ho(\curl)^*} \lesssim \|
  q \|_{\Hm(\div)}$ showing the stated continuous embedding.

  {\em Step~3}. \underline{$H(\curl)^* \hookrightarrow \Ht(\div)$:} \quad Writing
  $H(\curl)$ as $H(\curl,\om)$ to explicitly indicate the domain, we
  identify any given $f \in H(\curl, \om)^*$ with the following
  distribution in~$\bb R^3$,
  \[
    \tilde f \in \cl D(\bb R^3)': \quad
    \tilde f (\vphi) = f( \vphi|_\om), \quad \vphi \in \cl D(\bb R^3).
  \]
  By the Riesz representation theorem, there is a unique
  $u_f \in H(\curl,\om)$ such that
  $f(v) = (\curl u_f, \curl v) + (u_f, v)$ for all $v$ in
  $H(\curl,\om)$. Let $\tilde c_f$ and $\tilde u_f$ denote the
  extensions by zero of $L_2(\om)$-functions $\curl u_f$ and $u_f$ by
  zero to all $\bb R^3$. Then, for any $\vphi \in \cl D(\bb R^3)$, 
  \[
    \tilde f (\varphi) = f(\varphi|_{\om})
    = (\tilde c_f , \curl \vphi)_{L_2(\bb R^3)} + (\tilde u_f, \vphi)_{L_2(\bb R^3)}.
  \]
  This shows that the following identities hold in $\cl D(\bb R^3)$: 
  \begin{gather*}
    \tilde f  = \curl \tilde c_f + \tilde u_f,
    \qquad 
    \div \tilde f  = \div \tilde u_f.
  \end{gather*}
  These two identities give these corresponding bounds
  \begin{align*}
    \| f \|_{\Ht}
    & = \| \tilde f \|_{H^{-1}(\bb R^3)}
      \le
      \| \curl \tilde c_f\|_{H^{-1}(\bb R^3)} +\| \tilde u_f\|_{H^{-1}(\bb R^3)}
    \\
    & \lesssim
      \| \tilde c_f \|_{L_2(\bb R^3)} + \| \tilde u_f \|_{L_2(\bb R^3)}
      \lesssim \|u_f \|_{H(\curl)},
    \\
    \| \div f \|_{\Ht}
    & = \| \div \tilde f \|_{H^{-1}(\bb R^3)} = \| \div \tilde u_f \|_{H^{-1}(\bb R^3)}
      \lesssim  \| \tilde u_f \|_{L^2(\bb R^3)} = \| u_f \|_{L_2},
  \end{align*}
  which prove that $\| f \|_{\Ht(\div)} \lesssim \| u_f \|_{H(\curl)} = \| f \|_{H(\curl)^*}$ using the Riesz isometry.

  {\em Step~4}. \underline{$\Ht(\div) \hookrightarrow H(\curl)^* $:} \quad In
  exactly the same way the regularized Bogovski{\u{i}} operators of
  \eqref{eq:std-rt-inverses} give the regular decomposition
  \eqref{eq:Hocurl-reg-dec}, the regularized Poincar{\'{e}} operators
  of \cite{CostaMcInt10} show that there are continuous operators
  $\Sc 0 : H(\curl) \to H^1 \otimes \bb V$ and
  $\Sc 1 : H(\curl) \to H^1$ such that any $u \in H(\curl)$ can be
  decomposed into
  \begin{equation}
    \label{eq:Hcurl-reg-dec-nobc}
    u = \Sc 0 u + \grad \Sc 1 u.
  \end{equation}
  Let $q \in \Ht(\div)$. Using the regular decomposition, we define a
  functional \[ f_q (u) = -(\div q)(\Sc 1 u) + q( \Sc u )\]
  where, on
  the right hand side, the functional actions are duality pairings
  between $\Ht$ and $H^1$, well defined in view of \eqref{eq:7}. Hence
    \begin{align*}
      |f_q (u) |
      & \le \| \div q \|_{\Ht}\| \Sc 1 u \|_{H^1} + \| q \|_{\Ht } \| \Sc 0 u \|_{H^1}
        \lesssim \| q \|_{\Ht(\div)} \| u \|_{H(\curl)}
    \end{align*}
    Therefore $f_q$ is in $H(\curl)^*$ and
    \begin{equation}
      \label{eq:fq-u-Hcurl-bd}
      \| f_q \|_{H(\curl)^*} \lesssim \| q \|_{\Ht(\div)}. 
    \end{equation}
    Since we have already shown
    that $H(\curl)^* $ is embedded into the subspace of
    $\bb R^3$-distributions in $ \Ht(\div)$, we conclude that
    $f_q \in \cl D(\bb R^3)'$ satisfies, for any $\vphi \in \cl D(\bb R^3)$,
    \[
      f_q (\vphi) = -(\div q)(\Sc 1 \vphi) + q ( \Sc 0 \vphi) =
      q( \grad \Sc 1 \vphi + \Sc 0 \vphi) = q(\vphi), 
    \]
    i.e., $f_q$ and $q$ coincide as distributions. Combined
    with~\eqref{eq:fq-u-Hcurl-bd}, we have thus shown that 
    $\| q\|_{H(\curl)^*} = \|f_q \|_{H(\curl)^*} \lesssim \| q \|_{\Ht(\div)}.$

    Steps~1 and~2 prove the second identity of the lemma. Steps~3
    and~4 prove the last identity of the lemma. The proofs of the
    remaining identities are similar.
\end{proof}

In the next theorem we show that more dualities such as those in
Lemma~\ref{lem:duality-vector-cases} can be read off the diagrams with
and without boundary conditions.  Let
$\LL: H^1 \to \Ht \equiv (H^1)^*$ denote the Riesz map of $H^1$
defined by
\begin{equation}
\label{eq:Riesz-defn}
  (\LL w)(v) = (w, v)_{H^1} \equiv (\grad w, \grad v) + (w, v),
  \quad w, v \in \Ho^1,
\end{equation}
where $(\cdot, \cdot)_X$ denotes the inner product of $X$, and when
the subscript is absent, we interpret it as $L_2$~products, as described
previously in~\eqref{eq:L2-inner}. By the Riesz representation
theorem, recalling that our spaces are over the real field, $\LL$ is a
linear invertible isometry, so for any $q, r \in \Ht$,
\begin{align*}
  (q, r)_{\Ht}
  & = (\LL^{-1} q, \LL^{-1}r)_{H^1} = r( \LL^{-1} q),
\end{align*}
where we have used~\eqref{eq:Riesz-defn} in the last step.
In particular, if a $v \in L_2 \subset \Hm$ is used in place of $r$ above, the functional action becomes an $L_2$~product and we obtain 
\begin{equation}
  \label{eq:H-minus-product}
  (q, v)_{\Ht} = (\LL^{-1} q, v), \qquad v \in L_2.
\end{equation}
Note that $\LL^{-1}q $ is the result of solving a Neumann problem with
$q$ as the source. When $q$ is a vector or matrix field, by
$\LL^{-1}q $ we mean the respective vector or matrix field obtained by
component-wise application of~$\LL^{-1}$.

\begin{theorem}
  \label{thm:duality}
  The diagrams~\eqref{eq:1} and~\eqref{eq:no-bc-2-complex-quotient} are in
  duality in the sense displayed below, where the second diagram has been 
  rearranged to easily display the duality ($*$) correspondences.
\begin{equation}
  \label{eq:duality-diagram}
  \begin{tikzcd} 
    [
    column sep=-0.5cm,
    ampersand replacement=\&
    ]
    \Ho^1
    \arrow[rr]
    \arrow[dr, "*"]
    \arrow[dd]
    \&
    \&
    \Ho(\curl)
    \arrow[dr, "*"]
    \arrow[dd]
    \arrow[rr]
    \&
    \&
    \Ho(\div)
    \arrow[dr, "*"]
    \arrow[rr]
    \arrow[dd]
    \&
    \&
    \LR
    \arrow[dr, "*"]
    \arrow[dd]
    \\
    \&
    \Hm
    \&
    \&
    \Hm(\div)
    \arrow[ll, crossing over]
    \&
    \&
    \Hm(\curl)
    \arrow[ll, crossing over]
    \&
    \&
    L_2/\bb R
    \arrow[ll, crossing over]    
    \\
    \Ho(\curl)
    \arrow[rr]
    \arrow[dd]
    \arrow[dr, "*"]
    \&
    \&
    \Hocc
    \arrow[rr]
    \arrow[dd]
    \arrow[dr, "*"]
    \&
    \&
    \Hocd
    \arrow[rr]
    \arrow[dd]
    \arrow[dr, "*"]
    \&
    \&
    \HtRT(\curl)
    \arrow[dd]
    \arrow[dr, "*"]
    \\
    \&
    \Hm(\div)
    \arrow[uu, crossing over]
    \&
    \&
    \Hdd
    \arrow[ll, crossing over]
    \arrow[uu, crossing over]
    \&
    \&
    \HcdT
    \arrow[ll, crossing over]
    \arrow[uu, crossing over]
    \&
    \&
    H(\div)/\RT
    \arrow[ll, crossing over] 
    \arrow[uu]
    \\
    \Ho(\div)
    \arrow[dr, "*"]
    \arrow[rr]
    \arrow[dd]
    \&
    \&
    \HocdT
    \arrow[dr, "*"]
    \arrow[rr] 
    \arrow[dd]
    \&
    \&
    \Hodd
    \arrow[dr, "*"]
    \arrow[rr] 
    \arrow[dd]
    \&
    \&
    \HtND(\div)
    \arrow[dr, "*"]
    \arrow[dd]
    \\
    \&
    \Hm(\curl)
    \arrow[uu, crossing over]    
    \&
    \&
    \Hcd
    \arrow[uu, crossing over]
    \arrow[ll, crossing over]    
    \&
    \&
    \Hcc
    \arrow[uu, crossing over]
    \arrow[ll, crossing over]    
    \&
    \&
    H(\curl)/\ND
    \arrow[uu, crossing over]
    \arrow[ll, crossing over]    
    \\
    \LR
    \arrow[rr]
    \arrow[dr, "*"]
    \&
    \&
    \HtRT(\curl)
    \arrow[rr]
    \arrow[dr, "*"]
    \&
    \&
    \HtND(\div)
    \arrow[rr]
    \arrow[dr, "*"]
    \&
    \&
    \HtPl
    \arrow[dr, "*"]
    \\
    \&
    L_2/\bb R
    \arrow[uu, crossing over]
    \&
    \&
    H(\div)/\RT
    \arrow[uu, crossing over]
    \arrow[ll, crossing over]    
    \&
    \&
    H(\curl)/\ND
    \arrow[uu, crossing over]
    \arrow[ll, crossing over]    
    \&
    \&
    H^1/\pol_1
    \arrow[uu, crossing over]
    \arrow[ll, crossing over]    
  \end{tikzcd}  
\end{equation}
\end{theorem}
\begin{proof}
  The strategy of this proof is the same as what was outlined in the
  beginning of the proof of Lemma~\ref{lem:duality-vector-cases}. We
  now focus on the spaces of matrix-valued functions.
  
  {\em Step 1.} \underline{$\Hocc^* \hookrightarrow \Hdd$:} \quad Let
  $f \in \Hocc^*$. Then, by Riesz representation, there is a
  $g_f \in \Hocc$ such that $f(g) = (g_f, g)_{\Hocc}$ for all
  $g \in \Hocc$.  Expanding out the $\Hocc$~inner product and using
  \eqref{eq:H-minus-product}, for any $\vphi \in \cl D(\om) \otimes \bb S$, 
  \begin{align*}
    f(\vphi)
    &
      = (g_f, \vphi)_{\Hm} + (\curl g_f, \curl \vphi)_{\Hm}
      + (\inc g_f, \inc \vphi)_{\Hm}
    \\
    & =
      (\LL^{-1} g_f, \vphi)
      + (\LL^{-1} \curl g_f, \curl \vphi)
      + (\LL^{-1} \inc g_f, \inc \vphi).
  \end{align*}
  Thus
  \begin{equation}
    \label{eq:f-representation}
    f = \LL^{-1} g_f + \curl \LL^{-1} \curl g_f + \inc \LL^{-1} \inc  g_f,
  \end{equation}
  a sum of three terms, which are in $H^1, L_2$, and $\Hm$,
  respectively. Since $H^1 \hookrightarrow L_2 \hookrightarrow \Hm$,
  \begin{align*}
    \| f \|_{\Hm}
    & \le
      \|  \LL^{-1} g_f\|_{\Hm} +
      \|\curl \LL^{-1} \curl g_f\|_{\Hm} +
      \| \inc \LL^{-1} \inc  g_f\|_{\Hm}
    \\
    & \lesssim
      \|  \LL^{-1} g_f\|_{H^1} +
      \|\curl \LL^{-1} \curl g_f\|_{L_2} +
      \| \inc \LL^{-1} \inc  g_f\|_{\Hm}
    \\
    & \lesssim
      \|  \LL^{-1} g_f\|_{H^1} +
      \|\LL^{-1} \curl g_f\|_{H^1} +
      \| \LL^{-1} \inc  g_f\|_{H^1}
    \\
    & =
      \|   g_f\|_{\Ht} +
      \|\curl g_f\|_{\Ht} +
      \| \inc  g_f\|_{\Ht}
    \\
    & \lesssim \| g_f \|_{\Hocc} = \| f\|_{\Hocc^*}
  \end{align*}
  where we have used the Riesz isometry multiple times as well as the
  continuity of the derivative $\d_i : H^m \to H^{m-1}$ for some  
  integers~$m$ (see e.g.,~\cite[Theorem~1.4.4.6]{Grisv85}).
  From~\eqref{eq:f-representation} we also see that
  $\div f = \div \LL^{-1} g_f$ and
  $\div\div f = \div \div \LL^{-1} g_f$, so $f$ is indeed in
  $\Hdd$. Moreover,
  \begin{align*}
    \| f\|_{\Hdd}^2
    & = \| f \|_{\Hm}^2 + \| \div \LL^{-1} g_f\|_{\Hm}^2
      + \| \div \div \LL^{-1} g_f\|_{\Hm}^2
    \\
    & \lesssim
      \| f \|_{\Hm}^2 + \| \LL^{-1} g_f\|_{H^1}^2
      = \| f \|_{\Hm}^2 + \| g_f\|_{\Ht}^2
    \\
    & \lesssim \| f \|_{\Hocc^*},
  \end{align*}
  thus completing the proof of continuity of the embedding of $f$ into $\Hdd$.

  {\em Step 2.}  \underline{$\Hdd \hookrightarrow \Hocc^*$:} Let $\sigma \in
  \Hdd$. Decomposing a $g \in \Hocc$ using \eqref{eq:reg-dec-Hcc} as
  $g = \Socc 0 g + \dfo \Socc 1 g + \hess \Socc 2 g$, we define a
  linear functional $f_\sigma$ acting on $g$ as follows:
  \[
    f_\sigma (g)  = \sigma( \Socc 0 g) - (\div \sigma) ( \Socc 1 g)
    + (\div\div \sigma)(\Socc 2 g).
  \]
  By the continuity of $\Socc i$ (see Theorem~\ref{thm:reg-dec-Hcc}),
  $ |f_\sigma(g) | \lesssim \| \sigma \|_{\Hdd} \| g \|_{\Hocc} $ so
  $f_\sigma $ is in $\Hocc^*$ and
  \begin{equation}
    \label{eq:f-sig-bd}
    \| f_\sigma \|_{\Hocc^*} \lesssim \| \sigma \|_{\Hdd}.
  \end{equation}
  By the previous embedding, we know that $f_\sigma$ is in $\Hdd$ (in
  particular in $\Hm \otimes \bb S$) and is therefore a distribution
  on $\om$. We claim that $f_\sigma$ and $\sigma$ are identical
  distributions. Indeed, for any $\vphi \in \cl D(\om) \otimes \bb S$,
  \begin{align*}
    f_\sigma(\vphi)
    & =  \sigma( \Socc 0 \vphi) - (\div \sigma) ( \Socc 1 \vphi)
      + (\div\div \sigma)(\Socc 2 \vphi)
    \\
    & = \sigma (    \Socc 0 \vphi + \dfo \Socc 1 \vphi + \hess \Socc 2 \vphi)
      = \sigma(\vphi)
  \end{align*}
  where we have used \eqref{eq:sigma-eps} and \eqref{eq:sigma-hess} of
  Lemma~\ref{lem:extensions-to-range}. Combined
  with~\eqref{eq:f-sig-bd}, we have
  $ \| \sigma \|_{\Hocc^*} \equiv \| f_\sigma \|_{\Hocc^*} 
  \lesssim \| \sigma \|_{\Hdd},$ thus completing the proof of
  continuity of the embedding of $\sigma$ into $\Hodd^*$.

  {\em Step 3.} \underline{$\Hocd^* \hookrightarrow \HcdT$:} Let $f \in \Hocd^*$. By
  the Riesz representation theorem, there is a unique
  $\tau_f \in \Hocd$ satisfying $f(\tau) = (\tau_f, \tau)_{\Hocd}$ for
  all $\tau \in \Hocd$. Choosing
  $\tau =\vphi \in \cl D \otimes \bb T$, the definition of
  $\Hocd$-inner product and \eqref{eq:H-minus-product} imply
  \begin{align*}
    f(\vphi)
    & = ( \LL^{-1} \tau_f, \vphi)
    + ( \LL^{-1}\div \tau_f, \div \vphi)
    + (\LL^{-1} \sym\curl \T \tau_f, \sym\curl \vphi)
    \\
    & + (\LL^{-1} \curl\div \tau_f, \curl\div \vphi).   
  \end{align*}
  Hence
  \begin{equation}
    \label{eq:f-Hcd}
    \begin{aligned}      
      f
      &  = \LL^{-1} \tau_f - \grad \LL^{-1} \div \tau_f
      + \T \curl \sym \LL^{-1} \sym\curl \T \tau_f
      \\
      & - \grad \curl \LL^{-1} \curl\div \tau_f.
    \end{aligned}
  \end{equation}
  Applying $\div \T$, the identity \eqref{eq:5} shows that last term
  vanishes and
  \begin{align}
    \label{eq:div-T-f-Hcd}
    \div\T f
    = \div \T \LL^{-1} \tau_f - \div \T \grad \LL^{-1} \div \tau_f. 
  \end{align}
  Moreover, applying $\curl\div\T$ to both sides of~\eqref{eq:f-Hcd}
  and using the identity~\eqref{eq:2},
  \begin{align}
    \label{eq:cdT-f-Hcd}
    \curl\div\T f
    & = \curl\div \T \LL^{-1} \tau_f.
  \end{align}
  Also, 
  \begin{align}
    \label{eq:sc-f-Hcd}
    \sym\curl  f
    & = \sym\curl \LL^{-1} \tau_f
      + \inc \sym \LL^{-1} \sym\curl \T \tau_f.
  \end{align}
  Equations~\eqref{eq:f-Hcd},~\eqref{eq:div-T-f-Hcd},
  \eqref{eq:cdT-f-Hcd} and~\eqref{eq:sc-f-Hcd} imply
  that
  \[
    \| f \|_{\HcdT} \lesssim \| \tau_f \|_{\Hocd} = \| f \|_{\Hocd^*}
  \]
  where we have used the isometry of $\LL^{-1}$ and
  $f \mapsto \tau_f$.  This proves that
  $\Hocd^* \hookrightarrow \HcdT$.

  {\em Step 4.} \underline{$\HcdT \hookrightarrow \Hocd^*$:} Let $\tau \in
  \HcdT$. Define a linear functional $f_\tau$ acting on any $\eta$ in
  $\Hocd$ as follows:
  \[
    f_\tau(\eta) = \tau ( \Socd 0 \eta)
    + (\sym\curl \tau)(\Socd 1 \eta)
    - (\div \T \tau)(\Socd 3 \eta)
    - (\frac 1 2 \curl\div \T\tau)(\Socd 2 \eta).
  \]
  All the terms on the right are well defined duality pairings since
  $ \tau, \sym\curl \tau, \div \T \tau,$ and $ \curl\div \T\tau,$ have
  $\Ht$~components for any $\tau \in \HcdT$ and since each
  $\Socd i \eta$ has $\Ho^1$~components per
  Theorem~\ref{thm:reg-dec-Hcd}. The continuity of $\Socd i$ asserted
  by the same theorem gives
  \begin{equation}
    \label{eq:ftau-eta-bd}
    |f_\tau(\eta)| \lesssim \| \tau\|_{\HcdT} \| \eta\|_{\Hocd}.
  \end{equation}
  Next, observe that for any
  $\vphi \in \cl D(\om) \otimes \bb T$, decomposing $\vphi$ by
  Theorem~\ref{thm:reg-dec-Hcd} into
  $\vphi = \Socd 0 \vphi + \curl \Socd 1 \vphi + \T \dev \grad \Socd 2
  \vphi + \curl\dfo \Socd 3 \vphi,$ and using \eqref{eq:tau-curl},
  \eqref{eq:tau-dev-grad} and \eqref{eq:tau-curl-def} of
  Lemma~\ref{lem:extensions-to-range},
  \begin{align*}
    f_\tau(\vphi)
    & = \tau ( \Socd 0 \vphi) 
    + \tau(\curl \Socd 1 \vphi)
    + \tau(\T\dev\grad \Socd 3 \vphi)
    + \tau( \curl\dfo \Socd 2 \vphi)
    \\
    & = \tau (\vphi).
  \end{align*}
  Hence $f_\tau$ and $\tau$ are the same
  distribution. Using~\eqref{eq:ftau-eta-bd},
  $\| \tau \|_{\Hocd^*} \equiv \|f_\tau \|_{\Hocd^*} \lesssim \| \tau
  \|_{\HcdT}$, which proves  that $\HcdT \hookrightarrow \Hocd^*$ is a
  continuous embedding.

  {\em Step 5.} \underline{$\Hodd^* \hookrightarrow \Hcc$:} Denoting the Riesz
  representative of any $f \in \Hodd^*$ by $\sigma_f \in \Hodd$, its
  defining equation $f(\sigma) = (\sigma_f, \sigma)_{\Hodd}$ for all
  $\sigma \in \Hodd$ gives a formula for $f$ as in the previous cases:
  \[
    f = \LL^{-1} \sigma_f - \grad \LL^{-1} \div \sigma_f
    + \hess \LL^{-1} \div\div \sigma_f.
  \]
  Then $\curl f = \curl \LL^{-1} \sigma_f$ is in $L_2\otimes \bb V$
  and $\inc f = \inc \LL^{-1} \sigma_f $ is in $\Ht \otimes \bb
  S$. Hence $f$ is in $\Hcc$ and by Riesz isometries,
  $\| f \|_{\Hcc} \lesssim \| f\|_{\Hodd^*}$.

  {\em Step 6.} \underline{$\Hcc \hookrightarrow \Hodd^*$:} Let $g \in \Hcc$. We
  will show that $g$ can be identified with the functional $f_g$
  acting on $\sigma \in \Hodd$ by
  \[
    f_g(\sigma) = g(\Sodd 0 \sigma) + (\curl g) (\Sodd 1 \sigma) + (\inc g)(\Sodd 2).
  \]
  By the regular decomposition result of
  Theorem~\ref{thm:reg-dec-Hdd}, this implies that
  \begin{equation}
    \label{eq:fg-sigma}
    |f_g (\sigma) | \lesssim \| g \|_{\Hcc} \| \sigma \|_{\Hodd}
  \end{equation}
  and also using~\eqref{eq:g-sc} and \eqref{eq:g-inc} of
  Lemma~\ref{lem:extensions-to-range}, we find that
  \[
    f_g(\vphi) = g( \Sodd 0 \vphi + \sym\curl \Sodd 1 \vphi + \inc
    \Sodd 2 \vphi)= g( \vphi)
  \]
  for any $\vphi \in \cl D(\om) \otimes \bb S$, i.e., $f_g$ and $g$
  are the same distribution. Hence \eqref{eq:fg-sigma} shows that
  $\| g \|_{\Hodd^*} = \|f_g \|_{\Hodd^*} \lesssim \| g \|_{\Hcc}$
  thus establishing the continuous embedding of $\Hcc$ into
  $ \Hodd^*$.

  To conclude the proof, note that Steps~1 and~2 prove that
  $\Hocc^*=\Hdd$. Steps~3 and~4 prove that $\Hocd^* = \HcdT$ and
  $\HocdT^* = \Hcd$. Steps~5 and~6 prove that $\Hodd^* =\Hcc$. To
  establish the remaining nontrivial duality identities
  in~\eqref{eq:duality-diagram}, it suffices to use the identities of
  Lemma~\ref{lem:duality-vector-cases} and the fact that for any
  Hilbert space $X$ and its closed subspace $Z$, the dual of the
  quotient space $(X /Z)^*$ is isomorphic to the annihilator
  $Z^\perp = \{ x' \in X': x'(z) = 0$ for all $z \in Z\}$ contained in
  $X'$. For example,  with $X = H(\curl)$ and $Z = \ND$, we
  find that $(H(\curl)/\ND)^*$ is isomorphic to the annihilator of
  $\ND^\perp$ contained in $H(\curl)^* = \Ht(\div)$, which is exactly
  the same as $\HtND(\div),$ i.e., $\big(H(\curl)/\ND\big)^* =
  \HtND(\div)$. Taking duals on both sides we obtain, due to the reflexivity
  of Hilbert spaces, $\HtND(\div)^* = H(\curl)/\ND,$ which is one of
  the identities indicated in~\eqref{eq:duality-diagram}.
\end{proof}

\section*{Funding}

This work was supported in part by the National Science Foundation
(USA) under grant DMS-219958, the Austrian Science Fund (FWF) project 10.55776/F65,
a Royal Society University Research Fellowship (URF$\backslash$R1$\backslash$221398) and the European Union (ERC, GeoFEM, 101164551). Views and opinions expressed are however those of the authors only and do not necessarily reflect those of the European Union or the European Research Council. Neither the European Union nor the granting authority can be held responsible for them. For
open-access purposes, the authors have applied a CC BY public copyright
license to any author-accepted manuscript version arising from this
submission.

\bibliographystyle{siam}      
\bibliography{references}

\end{document}